\documentclass[a4paper,10pt]{article}
\usepackage{xcolor}
\usepackage{amsmath,amsthm,amsfonts,amssymb,euscript,dsfont,mathrsfs,enumerate}
\usepackage[english]{babel}
\usepackage[utf8]{inputenc}

\newcommand{\attractor}[1]{}

\DeclareUnicodeCharacter{00A0}{~}
\usepackage{hyperref}

\newtheorem{assumption}{Assumption}
\newtheorem{prop}{Proposition}

\newtheorem{definition}{Definition}
\newtheorem{theorem}{Theorem}
\newtheorem{coro}{Corollary}
\newtheorem{lemma}{Lemma}
\newtheorem{remark}{Remark}

\usepackage{geometry}
\newcommand{\eqdef}{:=}
\newcommand{\un}{\mathds{1}}
\newcommand{\ps}[1]{\langle #1\rangle}
\newcommand{\bP}{{\mathbb P}}
\newcommand{\bE}{{\mathbb E}}
\newcommand{\bR}{{\mathbb R}}
\newcommand{\bI}{{\mathbb I}}
\newcommand{\bN}{{\mathbb N}}
\newcommand{\cA}{{\mathcal A}}
\newcommand{\cI}{{\mathcal I}}
\newcommand{\cB}{{\mathcal B}}
\newcommand{\cV}{{\mathcal V}}
\newcommand{\cC}{{\mathcal C}}
\newcommand{\cD}{{\mathcal D}}

\newcommand{\cS}{{\mathcal S}}
\newcommand{\cP}{{\mathcal P}}
\newcommand{\cG}{{\mathcal G}}
\newcommand{\cM}{{\mathcal M}}

\newcommand{\cN}{{\mathcal N}}

\newcommand{\cF}{{\mathcal F}}
\newcommand{\cK}{{\mathcal K}}

\newcommand{\sV}{{\mathsf V}}

\newcommand{\sW}{{\mathsf W}}

\DeclareMathOperator{\BC}{BC} 
\DeclareMathOperator{\tr}{tr} 
\DeclareMathOperator{\supp}{supp} 
\DeclareMathOperator*{\acc}{acc} 
\DeclareMathOperator{\AC}{AC} 
\DeclareMathOperator{\Tan}{Tan}

\newcommand{\dr}{{d}}

\newcommand{\6}{\partial}

\newcommand\esp[1]{\mathbb{E}\left[#1\right]} 
\newcommand{\indicatrice}{\mathds{1}}

\newcommand\dd[2]{\frac{{d} #1}{{d} #2}}

\newcommand\p[1]{\left( #1 \right)}
\newcommand{\espcond}[2]{\mathbb{E}\mathopen{}\left[#1\middle|#2\right]}
\newcommand{\proba}{\mathbb{P}}

\newcommand\norm[1]{\left\lVert #1 \right\rVert}
\newcommand\abs[1]{\left\lvert #1 \right\rvert}
\newcommand{\pb}[1]{}
\usepackage{authblk}

\usepackage{fancyhdr}

\begin{document}

\pagestyle{plain}

\title{Long run convergence of discrete-time  interacting particle systems of the McKean-Vlasov type \footnote{This is the accepted version of the article published in \textit{Stochastic Processes and their Applications}.\\
 The final version is available at: \url{https://www.sciencedirect.com/science/article/pii/S0304414925000882}}}

\author[1]{{Pascal}~{Bianchi}}
\author[2]{{Walid}~{Hachem}}
\author[1]{{Victor}~{Priser}}
\affil[1]{LTCI, T\'el\'ecom Paris}

\affil[2]{Laboratoire d’informatique Gaspard Monge (LIGM / UMR 8049),
Universit\'e Gustave Eiffel}


\maketitle


\begin{abstract}
  We consider a discrete time system of $n$ coupled random vectors, a.k.a. interacting particles.
  The dynamics involves a vanishing step size, some random centered perturbations,
  and a mean vector field which induces the coupling between the particles.
  We study the doubly asymptotic regime where both the number of iterations and the number $n$ of particles tend to infinity,
  without any constraint on the relative rates of convergence of these two parameters.
  We establish that the empirical measure of the interpolated trajectories of the particles converges in probability,
  in an ergodic sense, to the set of recurrent Mc-Kean-Vlasov distributions.
  A first application example is the granular media equation, where the particles are shown
  to converge to a critical point of the Helmholtz energy.
  A second  example  is the convergence of stochastic gradient descent 
  to the global minimizer of the risk, in a wide two-layer neural networks using random features.
\end{abstract}

\tableofcontents

\section{Introduction}
{Given two integers $n,d > 0$, consider the iterative algorithm defined as
follows. Starting with the $n$--uple $(X_0^{1,n},\dots, X_0^{n,n})$ of random
variables $X_0^{i,n} \in \bR^d$, the algorithm generates at the iteration $k+1$
for $k \in \bN$ the $n$--uple of $\bR^d$--valued random variables
$(X_k^{1,n},\dots, X_k^{n,n})$, referred to as  the \emph{particles}, according
to the dynamics:
\begin{equation}
  \label{eq:main}
  X_{k+1}^{i,n} = X_k^{i,n} 
  +{\gamma_{k+1}} b(X_{k}^{i,n},\mu_k^n)
  +  \sqrt{2\gamma_{k+1}}\xi_{k+1}^{i,n} 
  + \gamma_{k+1}\zeta_{k+1}^{i,n}\,, 
\end{equation}
for each $i\in [n]$ where $[n] \eqdef \{1,\ldots, n\}$, where
\begin{equation}
  \label{eq:muk}
  \mu_k^n := \frac 1n \sum_{i=1}^n \delta_{X_k^{i,n}} \,.
\end{equation}
In this equation, $b:\bR^d\times\cP_p(\bR^d)\to\bR^d$ is a continuous vector field, where, for some $p\in[1,2]$,
$\cP_p(\bR^d)$ is the space of probability measures with a finite $p$-th order, equipped with
the Wasserstein distance.
Moreover, $(\gamma_k)_k$ is a vanishing sequence of deterministic positive step
sizes, $((\xi_{k}^{i,n})_{i\in[n]})_{k\in \bN^*}$ and
$((\zeta_{k}^{i,n})_{i\in[n]})_{k\in \bN^*}$ are $\bR^{d\times n}$--valued
random noise sequences in the time parameter $k$.  We assume that for each $n$,
the $n$--uple $(X_0^{1,n},\dots, X_0^{n,n})$ is exchangeable. We also assume the exchangeability of the $n$--uple of
sequences $((\xi^{1,n}_k)_{k\in\bN^*}, \ldots,
(\xi^{n,n}_k)_{k\in\bN^*})$  and 
$((\zeta^{1,n}_k)_{k\in\bN^*}, \ldots, (\zeta^{n,n}_k)_{k\in\bN^*})$.
Defining, for each $n > 0$, the filtration 
$(\mathcal F_k^n)_{k\in\bN}$ as: 
\begin{equation}
\mathcal{F}_{k}^n := \sigma(
(X_0^{i,n})_{i\in[n]}, ((\xi_\ell^{i,n})_{i\in[n]})_{\ell\leq k},
((\zeta_\ell^{i,n})_{i\in[n]})_{\ell\le k}), \label{eq:filtration}
\end{equation}
we assume that for each $n$, the sequence
$((\xi_k^{i,n})_{i\in[n]})_k$ is a $(\mathcal F_k^n)_k$--martingale increment
sequence \emph{i.e.,} $\bE(\xi_{k+1}^{i,n}|\cF_k^n)=0$. Finally, we assume that
$$
\bE(\xi_{k+1}^{i,n}(\xi_{k+1}^{j,n})^T|\cF_k^n) =
\sigma(X^{i,n}_k, \mu_k^n)\sigma(X^{j,n}_k,\mu_k^n)^T \un_{i=j}$$ for some $\sigma : \bR^d\times\cP_p(\bR^d) \to \bR^{d\times d'}$, with $d'>0$. }

The aim of the paper is to characterize the asymptotic behavior of the empirical measure of the
particles $\mu_k^n$ in the regime where both the time index $k$ and the number of particles $n$
tend to infinity (denoted hereinafter as $(k,n)\to(\infty,\infty)$), without
any constraint on the relative rates of convergence of these two parameters.
To this end, 
we consider for each $i\in[n]$ the random 
continuous process
$\bar X^{i,n} : [0,\infty) \to \bR^d, t\mapsto \bar X_t^{i,n}$ defined as the
piecewise linear interpolation of the particles $(X_k^{i,n})_k$. Specifically,
writing 
\begin{equation}
\label{def:tau} 
\tau_k\eqdef \sum_{j=1}^k\gamma_j 
\end{equation} 
for each $k\in\bN$, we define: 
\begin{equation}
\label{interp} 
\forall t \in [\tau_k, \tau_{k+1}), \quad 
\bar X^{i,n}_t := 
X_k^{i,n} + \frac{t - \tau_k}{\gamma_{k+1}} 
 \left( X_{k+1}^{i,n}-X_k^{i,n} \right) . 
\end{equation} 
The interpolated processes $\bar X^{i,n}$, for $i \in [n]$, are elements of the set
$\cC$ of the $[0,\infty)\to\bR^d$ continuous functions,
equipped with the topology of uniform convergence on compact intervals.
This paper studies the empirical measure of these processes:
\begin{equation}
  \label{eq:m}
  m^n := \frac 1n\sum_{i=1}^n \delta_{\bar X^{i,n}}\,.
\end{equation}
For each $n$ and each $p\in [1,2]$, $m^n$ is a random variable on the space 
$\cP_p(\cC)$ of probability measures on $\cC$ with a finite $p$--moment, 
equipped with the $p$--Wasserstein metric $\sW_p$ (precise definitions of 
these notions provided below). Our aim is to analyze the convergence in 
probability, of the shifted random measures
$$
\Phi_t(m^n) = \frac 1n\sum_{i=1}^n \delta_{\bar X^{i,n}_{t+\,\cdot}}\,,
$$
when both $n$ and $t$ converge to infinity with arbitrary relative rates,
where for every $m\in \cP_p(\cC)$, $\Phi_t(m)\in \cP_p(\cC)$ is defined by
$\Phi_t(m)(f) = \int f(x(t+\,\cdot\,))dm(x)$ for every bounded continuous function $f$ on $\cC$.
Under mild assumptions on the vector field
$b$, and some moment assumptions on the iterates and on the noise sequence
$((\zeta_k^{i,n})_{i\in[n]})_k$, ensuring that the effect of the latter
becomes negligible in our asymptotic regime, we establish the following result, which we explain
hereafter.
\smallskip

\noindent {\bf Main theorem (informal).} The sequence $(\Phi_t(m^n))$ ergodically converges
in probability as $(t,n)\to(\infty,\infty)$
to the set of \emph{recurrent McKean-Vlasov distributions}.
\smallskip

Let us explain what the terms \emph{McKean-Vlasov distribution}, \emph{recurrent}, and \emph{ergodic convergence}
mean in this paper.
Here, a McKean-Vlasov distribution $\rho$ is defined as the law of a $\bR^d$-valued process $(X_t:t\in\bR)$ satisfying
the following condition: for every smooth enough compactly supported function 
$\phi$, the process
$$
\phi(X_t) - \int_0^t  L(\rho_s) (\phi) (X_s)ds
$$
is a martingale, where $\rho_t$ the marginal law of $X_t$,
and where the linear operator $L(\rho_t)$ associates to $\phi$ the function $L(\rho_t)(\phi)$ given by:
{$$
x\mapsto  \ps{b(x,\rho_t),\nabla\phi(x)} + \tr(\sigma(x,\rho_t)^T H_\phi(x) \sigma(x,\rho_t))\,,
$$
 where $H_\phi$ is the Hessian matrix of $\phi$ and $\tr$ denotes the Trace operator. }

A McKean-Vlasov distribution $\rho$ is said recurrent if, for some sequence $(t_k)\to\infty$, $\rho = \lim_{k\to\infty} \Phi_{t_k}(\rho)$.
The $\sW_p$-closure of the set of recurrent McKean-Vlasov distributions will be referred to as the \emph{Birkhoff center},
and denoted by $\BC_p$, following the terminology used for general dynamical systems. 

By \emph{ergodic convergence}, we refer to the fact that the time averaged Wasserstein distance
between the measures $\Phi_t(m^n)$ and the Birkhoff center converges to zero.
Our main theorem can thus be written more precisely:
$$
\frac 1t\int_0^t\sW_p(\Phi_s(m^n),\BC_p)ds \xrightarrow[(t,n)\to(\infty,\infty)]{} 0\,,\ \text{ in probability.}
$$
The Birkhoff center can be characterized in a useful way, provided that one is able to show the existence
of a \emph{Lyapunov function}, namely a function $F$ on $\cP_p(\cC)$ such that, for every McKean-Vlasov distribution $\rho$,
$F(\Phi_t(\rho))$ is non-increasing in the variable $t$. Indeed, in such a situation, the Birkhoff center is included
in the subset $\Lambda$ of McKean-Vlasov distributions which satisfy the property that $t\mapsto F(\Phi_t(\rho))$ is constant
whenever $\rho\in \Lambda$.

Finally, in the case where the McKean-Vlasov dynamics can be cast in the form
of a gradient flow in the space of measures $\cP_p(\bR^d)$, and in case
this gradient flow has a global attractor $A_p$, we show that 
\[
W_p\left( \mu_k^n,A_p \right)
   \xrightarrow[(k,n)\to (\infty,\infty)]{} 0 \quad \text{in probability.} 
\]

{To illustrate our results, we provide an important example of a McKean-Vlasov distribution where these results can be applied: the granular media equation. Additionally, our results can also be applied in several machine learning applications, such as two-layer neural networks or the Stein Variational Gradient Descent (SVGD) algorithm.}

{{\bf Granular media.} Our example is in $\cP_2(\cC)$ and corresponds to 
the scenario where $\sigma(x,\mu) = \sigma I_d$ for some real constant $\sigma\ge 0$, and with a slight abuse of notation the vector field \(b\) takes the form $b(x,\mu) = \int b(x,y)\dr\mu(y) $, with:}
$$
b(x,y) = -\nabla V(x) - \nabla U(x-y)\,,
$$
where the \emph{confinement potential} \(V\) and the \emph{interaction potential} \(U\) denote two real differentiable functions on \(\mathbb{R}^d\), whose gradients satisfy some linear growth condition.
In this case, a Lyapunov function if provided by the \emph{Helmholtz energy}.
As a consequence of our main result, we establish that, when $\sigma>0$, the empirical measures $(\mu_k^n)$ converge ergodically in probability
as $(k,n)\to (\infty,\infty)$ to the set $\cS$ of critical points of the Helmholtz energy, namely:
$$
\frac {\sum_{l=1}^k \gamma_l W_2(\mu_l^n,\cS)}{\sum_{l=1}^k\gamma_l} \xrightarrow[(n,k)\to(\infty,\infty)]{} 0\,,\ \text{ in probability.}
$$
where, this time, $W_2$ 
represents the classical Wasserstein distance, and where
$\cS$ is the set of probability measures $\mu$ on $\bR^d$ which admit a second order moment and a density $d\mu/d\mathscr{L}^d$ w.r.t.
the Lebesgue measure, and such that:
$$
\nabla V(x) + \int \nabla U(x-y)d\mu(y) + \sigma^2 \nabla \log \frac{d\mu}{d\mathscr{L}^d}(x) = 0\,,
$$
for $\mu$-almost every $x$. Our result holds under mild assumptions, and does not require the rather classical
strong convexity or doubling conditions on $U$ and/or $V$.

{\bf Contributions}. Compared to existing works, our contributions are 
threefold.
First, our results hold under mild assumptions on the vector field $b$ aside from continuity and linear growth,
whereas most of the existing works (see below) rely on stronger conditions, such as Lipschitz, doubling or even global boundedness conditions.
Second, we address the case of discrete-time systems with a step size vanishing arbitrarily slowly towards $0$, whereas the continuous time model is more often considered in the literature.
Discrete-time algorithms are important in applications, such as neural networks, transformers, Monte Carlo simulations or numerical solvers.
In particular, stability results are more difficult to establish in this setting.
Finally, our result focuses on a double limit $(k,n)\to(\infty,\infty)$. At the exception of some papers listed below,
the results of the same kind generally consider the case, where
the time window is fixed, while the number of particles grows to infinity, ignoring long time convergence, or assume certain constraints
on the relative rate of convergence of the two variables.
\smallskip

{\bf About the literature}. The first results addressing the limiting behavior of a finite system of particles are provided in the context of the propagation of chaos. These findings are discussed in detail in \cite{Chaintron_2022}. Such results have broad applicability across a variety of particle systems, where the interacting term $b$ can manifest in various forms \cite{Mlard1987APO,Oelschlager1984AMA,sznitman1984nonlinear,10.1214/21-EJP580}.
In our case, if we set aside the transition from continuous to discrete time, such results typically establish the convergence to zero of the expectation of the squared Wasserstein distance between the empirical measure of the particles, over some fixed time interval $[0,T]$, and a McKean-Vlasov distribution with the same initial measure. Under classical assumptions, this convergence occurs at a rate of $1/n$, where $n$ is the number of particles, but with a constant that grows exponentially with $T$. This type of result performs poorly in the long run, making the achievement of the double limit in both time and the number of particles unattainable.

{By imposing additional assumptions, one can derive a bound that is uniform in time, thereby explicitly addressing the double asymptotic regime. However, these uniform-in-time propagation of chaos results are typically established in continuous time. The paper \cite{karimi2024stochastic} bridges the gap between continuous and discrete time in the specific context where uniform-in-time propagation of chaos holds for the continuous-time particle system, allowing for the recovery of our results. They demonstrate that the limiting distribution of the discrete-time particle system coincides with that of the continuous-time particle system. When uniform-in-time propagation of chaos holds, the limiting distributions of the continuous-time particle system converge to the unique stationary distribution of the associated McKean-Vlasov system as time grows. This, in turn, implies the convergence of the discrete-time particle system to the McKean-Vlasov stationary distribution in the doubly asymptotic regime. However, it should be noted that when applying the results of \cite{karimi2024stochastic}, we lose the convergence rate provided by uniform-in-time propagation of chaos, and the resulting result is no better than ours in the restrictive case where it is applicable.

Our contribution lies in the fact that our assumptions are weaker than those requiring uniform-in-time propagation of chaos, which are generally too strong for practical applications. Specifically, the first paper to address uniform-in-time propagation of chaos in the granular media setting is \cite{malrieu2001logarithmic}, which requires the strong convexity of the confinement potential and the convexity of the interaction potential. Later, \cite{cattiaux2008probabilistic} relaxed the strong convexity assumption on the confinement potential. \cite{durmus2020elementary} proposed a uniform-in-time propagation of chaos result when the confinement potential is strongly convex outside a ball, and the interaction potential has a sufficiently small Lipschitz constant. More recently,
\cite{monmarche2024time,chen2024uniform,lacker2023sharp} provide sharp uniform-in-time propagation of chaos results under a Log-Sobolev inequality on the vector field \(b\) and a noise with variance large enough.

As highlighted in \cite{del2019uniform}, achieving uniform propagation of chaos over time is only possible when a unique McKean-Vlasov stationary distribution exists. A condition that \cite{herrmann2010non} has demonstrated is not always met. In this regard, our assumptions are weaker, allowing for the existence of multiple stationary distributions. It is noteworthy that the study of McKean-Vlasov stationary distributions in cases where the uniqueness of such distributions does not hold remains an open area of research. For instance, \cite{cormier2023stability} explores the stability of stationary distributions. Additionally, \cite{budhiraja2015limits} explores a general class of non-linear Markov processes in finite-dimensional space and proposes a method to obtain Lyapunov functions for these processes.


Among papers that address the long-run convergence of discrete-time particle systems, \cite{malrieu2003convergence} employs an implicit Euler scheme for the granular media case, assuming a zero potential function and strongly convex interaction. The work in \cite{benko2024convergenceratesparticleapproximation} studies a Jordan–Kinderlehrer–Otto (JKO) scheme for granular media, assuming a strongly convex confinement potential. The contribution of \cite{veretennikov2006ergodic} is the closest to the present work, considering an equation similar to Eq.~\eqref{eq:main}, but assumes that \( b \) is globally bounded and only addresses the convergence of the expectation of the empirical measure, not convergence in probability. } Lastly, \cite{benaim2000ergodic} is closely related but not specific to McKean-Vlasov processes, as it does not consider particle systems or double limits. However, it establishes ergodic convergence of the empirical measure of a weak asymptotic pseudotrajectory to the Birkhoff center of a flow on a metric space, similar in spirit to our approach.

Finally, let us review some applications of our model.
Particle systems have historically been motivated by statistical physics. However, in recent decades, they have found utility in various models including neural networks, Markov Chain Monte Carlo theory, mathematical biology, and mean fields game, among others. 
A well-known model in statistical physics is granular media \cite{villani2006mathematics}. This model has been extensively studied due to its property of being a gradient system, and the uniform propagation of chaos over time works well within this model. It can also be described by a gradient flow \cite{ambrosio2005gradient}. 
In Markov Chain Monte Carlo theory, the Stein Variational Gradient Descent estimates a target distribution using a particle system \cite{liu2016stein, salim2022convergence}, and the convergence of this algorithm remains an open question.
Wide Neural Networks can also be represented by particle systems. A convergence result to the minimizers of the risk is attainable when both time and the number of particles tend to infinity \cite{chizat2018global}. Here, the authors establish convergence to gradient descent in continuous time and in the double asymptotic regime.
The paper \cite{mei2018mean} establishes the convergence of noisy stochastic gradient descent when the number of iterations depends on the number of particles. See also \cite{rotskoff2022trainability, sirignano2020mean,hu2021mean,chizatmean,pmlr-v151-nitanda22a} for related works.

\section{The setting} 
\label{sec:setting} 

We begin by introducing some notations and by recalling some definitions. 

\subsection{Notations} 
\label{subsec-not} 

\subsubsection{General notations} 
We denote by $\langle \cdot, \cdot \rangle$ and $\| \cdot \|$
the inner product and the corresponding norm in a Euclidean
space. We use the same notation in an infinite dimensional space, to denote the standard dual pairing and the operator norm.

For $k\in\bN\cup\{\infty\}$, we denote by $C^k(\mathbb{R}^d,\mathbb{R}^q)$ the
set of functions which are continuously differentiable up to the order $k$.  We
denote by $C_c(\bR^d,\bR)$ the set of $\bR^d\to\bR$ continuous functions with
compact support. Given $p\in \bN^*\cup \{\infty\}$, we denote as
$C_c^p(\bR^d,\bR)$ the set of compactly supported $\bR^d\to\bR$ functions which
are continuously differentiable up to the order $p$. 

We denote by $\cC$ the set of the $[0,\infty)\to\bR^d$ continuous functions.
It is well-known that the space $\cC$ endowed with the topology of the 
uniform convergence on the compact intervals of $[0,\infty)$ is a Polish space.

We denote by $\mathrm{conv}(A)$ the convex hull of a set $A$.

\subsubsection{Random variables}

The notation $f_\# \mu$ stands for the pushforward of the measure
$\mu$ by the map $f$, that is, $f_\# \mu= \mu\circ f^{-1}$.

For $t\geq 0$, we define the projections $\pi_t$ and $\pi_{[0,t]}$ as 
$\pi_t : (\bR^d)^{[0,\infty)} \to \bR^d, x \mapsto x_t$ and 
$\pi_{[0,t]} : (\bR^d)^{[0,\infty)} \to (\bR^d)^{[0,t]}, x \mapsto 
(x_u\, : \, u\in[0,t])$. 

Let $p\ge1$. For $\rho\in\mathcal{P}_p(\mathcal{C})$, we denote $$\rho_t:=(\pi_t)_\#\rho\,.$$

Let $(\Omega,\cF,\bP)$ be a probability space.  We say that a collection $A$ of
random variables on $\Omega\to E$ is \emph{tight} in $E$, if the family
$\{X_\#\bP:X\in A\}$ is weak$\star$-relatively compact in $\cP(E)$ \emph{i.e.},
has a weak$\star$ compact closure in $\cP(E)$.

We say that a $n$--uple of random variables $(X_1,\dots, X_n)$ is
\emph{exchangeable}, if its distribution is invariant by any permutation on
$[n]$.

Let ${\mathbb T}$ represent either $\bN$ or $[0,+\infty)$.  Let
$(U_t^n:t\in\mathbb{T},n\in \bN)$ be a collection of random variables on a
metric space $(E,\mathsf d)$.  We say that $(U_t^n)$ converges in probability
to $U$ as $(t,n)\to (\infty,\infty)$ if, for every $\epsilon>0$, the net
$(\bP(\mathsf d(U_t^n,U)>\epsilon):t\in\mathbb{T},n\in \bN)$ converges to zero
as $t$ and $n$ both converge to $\infty$.  We denote this by $U_t^n
\xrightarrow[(t,n)\to (\infty,\infty)]{\bP} U$.
{When $(U_t^n)$ is deterministic, we write $U_t^n
\xrightarrow[(t,n)\to (\infty,\infty)]{} U$.}
Moreover, assuming that the collection of random variables $(U_t^n:t\in\mathbb{T},n\in \bN)$ are real valued, we say that the latter collection is \textit{uniformly integrable} if:
\[
\lim_{a\to\infty} \sup_{t\in\mathbb T, n\in\bN^*} \esp{\abs{U_t^n}\indicatrice_{\abs{U_t^n} >a}} =0\,.
\]
{We define $\underset{(t,n)\to (\infty,\infty)}{\lim\sup} \,U_t^n := \underset{t\in \mathbb T,n\in\bN}{\inf}\underset{s\ge t, k\ge n}{\sup} U_s^k$.}
Finally, for any $d\in\mathbb{N}^*$,  $\mathscr{L}^d$ stands for the Lebesgue measure on $\mathbb{R}^d$.

\subsection{Spaces of probability measures}\label{sec:space}

Let $(E,\mathsf{d})$ denote a Polish space.  If $\cA\subset E$ is a subset, we
define $\mathsf{d}(x,\cA) \eqdef \inf\{\mathsf{d}(x,y):y\in \cA\}$, with
$\inf\emptyset=\infty$.  We say that a  net $(\mu_\alpha)$ converges to $\cA$ if
$\mathsf{d}(x_\alpha,\cA)\to_\alpha 0$.

We denote by $\cP(E)$ the set of probability measures on
the Borel $\sigma$-algebra $\cB(E)$.
We equip $\cP(E)$ with the weak$\star$ topology. Note that $\cP(E)$ is a Polish space.
We denote by $d_{L}$ the Levy-Prokhorov distance on $\cP(E)$, which is compatible with
the weak$\star$ topology.
 We define the \textit{intensity} of a random variable $\rho: \Omega \to \cP(E)$,
as the measure $\bI(\rho)\in \cP(E)$ that satisfies 
\[
\forall A \in \cF, \quad \bI(\rho)(A) := \bE\p{\rho(A)}. 
\]
\begin{lemma}[\cite{Mlard1987APO}]
\label{lem:meleardSimple}
A sequence $(\rho^n)$ of random variables on $\mathcal{P}(E)$
is tight if and only if the sequence $(\bI (\rho^n))$ is 
weak$\star$-relatively compact. 
\end{lemma}
Let $p\ge 1$. If $E$ is a Banach space, we define
$$
\cP_p(E)\eqdef \{\mu\in \cP(E)\,:\, \int\|x\|^p\dr\mu(x)<\infty\}\,.
$$
We define the Wasserstein distance of order $p$ on $\cP_p(E)$ by
\begin{equation}\label{eq:WpE}
{ W}_p(\mu,\nu) \eqdef \left(\inf_{\varsigma\in \Pi(\mu,\nu)} 
  \int \|x-y\|^p d\varsigma(x,y)\right)^{1/p}\,,
\end{equation}
where $\Pi(\mu,\nu)$ is the set of measures $\varsigma\in\cP(E\times E)$, such 
that $\varsigma(\,\cdot\,\times E)=\mu$ and $\varsigma(E\times\,\cdot\,)=\nu$.
We denote by $\Pi_p^0(\mu,\nu)$ the set of optimal transport plans \emph{i.e.},
the set of measures $\varsigma\in \Pi(\mu,\nu)$ achieving the infemum in Eq.~(\ref{eq:WpE}).
The set $\cP_p(E)$ is endowed with the distance ${ W}_p$.
Define: 
$$
\cP_p(\cC) {:}= \{\rho\in \cP(\cC)\,:\, \forall T>0,\, \int \sup_{t\in[0,T]} \|x_t\|^p \dr \rho(x)<\infty\}\,.
$$
For every $\rho,\rho'\in\cP_p(\cC)$, we define:
$$
{\mathsf W}_p(\rho,\rho') {:}= \sum_{n=1}^\infty 2^{-n} (1\wedge { W}_p( (\pi_{[0,n]})_\#\rho,(\pi_{[0,n]})_\#\rho'))\,.
$$
We equip $\cP_p(\cC)$ with the distance ${\mathsf W}_p$. 
We say that a subset $\cA\subset \cP_p(\cC)$ has 
\textit{uniformly integrable $p$-moments} if the following condition holds:
\begin{equation}
  \forall T>0, \quad \lim_{a\to\infty}\sup_{\rho\in \cA} \int \indicatrice_{\underset{{t\in[0,T]} }{\sup} \|x_t\| >a } \p{\sup_{t\in[0,T]} \|x_t\|^p}d\rho(x)=0\,. \tag{$p$-UI}\label{eq:condUI}
\end{equation} 
In the same way,  a sequence $(\rho^n)$ has {uniformly integrable $p$-moments} if the condition (\ref{eq:condUI}) holds for the sequence $(\rho^n)$ in place of $\cA$.
Following the same lines as  \cite[Th.~6.18]{villani2008optimal} and~\cite[Prop. 7.1.5]{ambrosio2005gradient}, we obtain the following lemma.
The proof is provided in~\ref{sec:proof-space}.
\begin{prop}\label{prop:space}
  \begin{enumerate}[i)]
  \item The space $\cP_p(\cC)$ is Polish.
  \item  A subset $\cA\subset \cP_p(\cC)$ is relatively compact if and only if, it is weak$\star$-relatively compact in $\cP(\cC)$, and if
 $\cA$ has uniformly integrable $p$-moments.
\end{enumerate}
\end{prop}
Finally, we will also consider $\cP_p(\cC)$-valued sequences of random variables.
Therefore, the following extension of Lem.~\ref{lem:meleardSimple}, will be useful. It is established in~\ref{sec:meleard}.
\begin{lemma}
\label{lem:meleard}
Let $(\rho^n)$ be a sequence of random variables valued in 
 $\cP_p(\cC)$. Assume that $(\bI(\rho^n))$ is relatively compact in $\cP_p(\cC)$.
Then, $(\rho^n)$ is tight in $\cP_p(\cC)$.
\end{lemma}

\subsection{Spaces of McKean-Vlasov measures}
\label{sec:mkv}

{Let $d'\in\bN^*$. Consider a matrix-valued function $\sigma: \bR^d\times \cP_p(\bR^d) \to \bR^{d\times d'}$ and a vector field $b:\bR^d\times\cP_p(\bR^d)\to\bR^d$ satisfying the following assumptions:
\begin{assumption}
  \label{hyp:b}
  The vector field $b:\bR^d\times\cP_p(\bR^d)\to\bR^d$, and $\sigma: \bR^d\times \cP_p(\bR^d) \to \bR^{d\times d'} $ are continuous. Moreover, there exists $C>0$ such that for all $(x,\mu)\in \bR^d\times\cP_p(\bR^d)$,
  $$
  \|b(x,\mu)\|\leq C(1+\|x\|+\int \|y\|\dr \mu(y)),\,
  $$
  and $\norm{\sigma(x,\mu)}\le C$.
\end{assumption}
We define $L(\mu)$ which, to every test
function $\phi\in C_c^2(\bR^d,\bR)$, associates the function $L(\mu)(\phi)$
given by 
\begin{equation}\label{eq:L} 
{L(\mu)(\phi)(x)= \ps{b(x,\mu),\nabla\phi(x)} + \tr\p{\sigma(x,\mu)^T H_\phi(x) \sigma(x,\mu)}\,,
}\end{equation}
{where $H_\phi$ is the hessian matrix of $\phi$. } } 
Let $(X_t:t\in [0,\infty))$ be the canonical process on $\cC$.
Denote by $(\cF_t^X)_{t\geq 0}$ the natural filtration (\emph{i.e.}, the filtration generated by $\{X_s:0\leq s\leq t\}$).
{\begin{definition}
  \label{def:V}
  Let $p\geq 1$. We say that a measure $\rho\in \cP_p(\cC)$ belongs to the class $\sV_p$ if,  for every $\phi\in C_c^2(\bR^d,\bR)$,
$$
\phi(X_t) -\int_0^t L(\rho_s)(\phi)(X_s)ds
$$
is a $(\cF_t^X)_{t\geq 0}$-martingale on the probability space $(\cC,\cB(\cC),\rho)$.
\end{definition}}
{The elements of $\sV_p$ will be referred to as McKean-Vlasov distributions of order $p$. } 
In the sequel, it will be convenient to work with the following equivalent characterization.
The martingale property implies that every measure $\rho\in \sV_p$ satisfies $G(\rho) = 0$, for every function $G:\cP_p(\cC)\to \bR$ of the form:
\begin{equation}
  \label{eq:G}
  G(\rho) := \int \left(\phi(x_t) -\phi(x_s) 
    -\int_s^t L(\rho_u)(\phi)(x_u)\mathrm{d} u\right) 
  \prod_{j=1}^r h_j(x_{v_j}) d\rho(x)\,,
\end{equation}
where $r\in \bN$, $\phi\in C_c^2(\bR^d,\bR)$, $h_1,\dots,h_r\in C_c(\bR^d,\bR)^r$,  
$0\leq v_1\leq \cdots \leq v_r \leq s \leq t$, are arbitrary.
We denote by $\cG_p$ the set of such mappings $G$.
Assumption~\ref{hyp:b} ensures that these mappings are well defined.
By Def.~\ref{def:V}, every $\rho\in \sV_p$ is a root of all $G\in \cG_p$.
As a matter of fact, a measure $\rho\in\cP_p(\cC)$ belongs to the set $\sV_p$, if and only if $G(\rho)=0$
for every $G$ of the form~(\ref{eq:G}). In other words, Def.~\ref{def:V} is equivalent to the following identity:
\begin{equation}\label{eq:Vp}
  \sV_p = \bigcap_{G\in\cG_p}G^{-1}(\{0\})\,.
\end{equation}
The following lemma is proved in~\ref{sec:continuity}. 
\begin{lemma}
\label{lem:continuityG}
Let Assumption~\ref{hyp:b} hold true. Every
$G \in \cG_p$ is a continuous function on $\cP_p(\cC)\to\bR$.
\end{lemma}
The following result is a consequence of Lem.~\ref{lem:continuityG} and Prop.~\ref{prop:space}.
\begin{prop}\label{prop:Vp}
    Under Assumption~\ref{hyp:b}, $\sV_p$ is a closed subset of $\cP_p(\cC)$. Moreover, equipped with the trace topology of $\cP_p(\cC)$, $\sV_p$ is a Polish space. 
\end{prop}
\begin{proof} 
For all $\rho^n\in \sV_p \to \rho^\infty$ in $\cP_p(\cC)$, it holds by Lem.~\ref{lem:continuityG} that $G(\rho^\infty) = 0$ for all $G\in\cG_p$, which shows that 
$\rho^\infty\in\sV_p$ by~\eqref{eq:Vp}. Hence, $\sV_p$ is closed. A closed subset of a Polish space is also Polish. By Prop.~\ref{prop:space}, $\sV_p$ is Polish. 
\end{proof}

\subsection{Dynamical systems}\label{sec:ergo}

Recall the definition of the shift $\Theta_t(x) = x_{t+\cdot}$ defined on
$\cC$. Let us equip the space $\sV_p$ assumed nonempty with the trace topology
of $\cP_p(\cC)$, making it a Polish space (see Prop.~\ref{prop:Vp}).
 With this at hand, one can readily
check that the function $\Phi: [0,\infty) \times \sV_p \to \sV_p$ defined as
$(t,\rho) \mapsto \Phi_t(\rho) = (\Theta_t)_\# \rho$ is a semi-flow on the space
$(\sV_p, \sW_p)$, 
in the sense that $\Phi$ is continuous, $\Phi_0(\cdot)$
coincides with the identity, and $\Phi_{t+s} = \Phi_t \circ \Phi_s$ for all
$t,s\geq 0$, see~\cite{ben-(cours)99} for an exposition of the concepts
related to semi-flows. 
The omega limit set of $\rho\in \sV_p$ for this semi-flow is the set
$\omega(\rho)$ defined by:
\[
\omega(\rho) \eqdef \bigcap_{t>0} \overline{\{\Phi_s(\rho) \, : \, s>t\}}\,.
\]
Equivalently, $\omega(\rho)$ is the set of ${\mathsf W}_p$-limits of sequences
of the form $(\Phi_{t_n}(\rho))$ where $t_n\to\infty$.  A point $\rho\in
\sV_p$ is called recurrent if $\rho\in \omega(\rho)$. The Birkhoff center 
$\BC_p$ is defined as the closure of the set of recurrent points:
\[
\BC_p \eqdef \overline{\{\rho\in \sV_p\,:\, \rho\in \omega(\rho)\}}\,.
\]
{By extension, given a measure $\mu\in \cP_p(\bR^d)$, we say that $\mu$
  is a \emph{recurrent marginal McKean-Vlasov measure} if there exists a recurrent measure $\rho\in \sV_p$ such that $\rho_0=\mu$.
  We denote by $\BC_p^0$ the closure of recurrent marginal McKean-Vlasov measures, that is,
  \begin{equation}
  \BC_p^0 = \overline{\{\rho_0\,:\, \rho\in \sV_p,\,\rho\in \omega(\rho)\}}\,,\label{eq:BC0}
\end{equation}
  or in short, $\BC_p^0 = (\pi_0)_\#(\BC_p)$.
}


\begin{definition}
  \label{def:lyapunov}
Consider the semi-flow $\Phi$ and  a non-empty set $\Lambda\subset\sV_p$.  A lower semi-continuous function
$F:\sV_p\to\bR$ is called a Lyapunov function for the set $\Lambda$ if,
for every $\rho\in \sV_p$ and every $t>0$, $F(\Phi_t(\rho)) \leq F(\rho)$,
and $F(\Phi_t(\rho) )< F(\rho)$ whenever $\rho\notin\Lambda$.
\end{definition}
The following result is standard.
\begin{prop}
  \label{prop:BC}
  Let $p>0$. If $F$ is a Lyapunov function for the set $\Lambda$,
  then $\BC_p \subset \overline\Lambda\,.$
\end{prop}
\begin{proof}
The limit $\ell:=\lim_{t\to\infty}F(\Phi_t(\rho))$ is well-defined because
$F(\Phi_t(\rho))$ is non increasing.  Consider a recurrent point $\rho\in
\sV_p$, say $\rho=\lim_{n} \Phi_{t_n}(\rho)$. Clearly $F(\rho)\geq
F(\Phi_{t_n}(\rho))\geq \ell$.  Moreover, by lower semicontinuity of $F$, $\ell
= \lim_n F(\Phi_{t_n}(\rho))\geq F(\rho)$. Therefore, $\ell$ is finite, and
$F(\rho)=\ell$.  This implies that $t\mapsto F(\Phi_t(\rho))$ is constant. By
definition, this in turn implies $\rho\in \Lambda$, which concludes the proof.
\end{proof}

\section{Main results}
\label{sec:main}


\subsection{\texorpdfstring{Interpolated process and weak$\star$ limits}{Interpolated Process and Weak-star Limits}}

Let $(\Omega,\cF,\bP)$ be a probability space. Let $d > 0$ be an integer.  For
each $n\in\bN^*$, consider the random sequence~\eqref{eq:main} starting with the
$n$--uple $(X_0^{1,n},\dots, X_0^{n,n})$ of random variables $X_0^{i,n} \in
\bR^d$, with $((\xi_{k}^{i,n})_{i\in[n]})_{k\in \bN^*}$ and
$((\zeta_{k}^{i,n})_{i\in[n]})_{k\in \bN^*}$ being $\bR^{d\times n}$--valued
random noise sequences. For each of integer $n > 0$, define the filtration 
$(\mathcal F_k^n)_{k\in\bN}$ as in Eq.~(\ref{eq:filtration}) or, more generally,
as any filtration such that the following random variables
\[
(X_0^{i,n})_{i\in[n]}, ((\xi_\ell^{i,n})_{i\in[n]})_{\ell\leq k},
((\zeta_\ell^{i,n})_{i\in[n]})_{\ell\le k}) 
\]
belong to $\cF_k^n$.
Consider the following assumptions:
\begin{assumption}\label{hyp:gamma}
    The sequence $(\gamma_k)$ is a non-negative deterministic sequence satisfying 
    \[\lim_{k\to\infty}\gamma_k = 0,\, \text{and } \sum_k\gamma_k = +\infty.\]
\end{assumption}

Recall the definition $\mu_k^n\eqdef \frac 1n\sum_{i=1}^n \delta_{X^{i,n}_k}$.
\begin{assumption}
\label{hyp:algo} 
The following hold{s } true.
\begin{enumerate}[i)]
\item 
\label{hypenum: exchangeable}
For each $n$, $( (X_0^{i,n},(\zeta_k^{i,n})_{k\in\bN} ,(\xi_k^{i,n})_{k\in\bN}))_{i\in [n]}$ is exchangeable as a n--uple of $\bR^d\times(\bR^d)^\bN\times (\bR^d)^\bN$--valued random variables.

\item 
\label{hypenum: va} 
{It holds that
$\displaystyle{\sup_{k,n} \bE\|\xi_k^{1,n}\|^4} < \infty$. 
Moreover, for each $n > 0$, and each $i$,$j$, 
{\begin{align*}
           &\espcond{\xi_{k+1}^{1,n}}{\cF_k^n}{=0}\, \\
             &\bE\left[\xi_{k+1}^{i,n}\p{\xi_{k+1}^{j,n}}^T \, | \, \cF_k^n\right] 
  =  \sigma(X_k^{i,n},\mu_k^n)\sigma(X_k^{j,n},\mu_k^n)^T \un_{i=j}\,,
\end{align*}}
}
\item
\label{hypenum: eta}
For each $k$, and each $n$,  $\displaystyle{ \bE\|\zeta_k^{1,n}\|} < \infty$, and
\[
\lim_{(k,n)\to(\infty,\infty)} \bE \left\| 
  \bE\left[ \zeta^{1,n}_{k+1}  \, | \, \cF_k^n \right] \right\| = 0\,.
\] 
\end{enumerate}
\end{assumption}

\begin{remark}
\label{remark:simple}
  Assumption~\ref{hyp:algo}--\eqref{hypenum: exchangeable} holds under the stronger assumption that 
  the $n$-uple $(X_0^{i,n})_{i\in[n]}$ is exchangeable, $(\xi^{i,n}_k)_{i\in[n],k\in\mathbb{N}}$ is an i.i.d. sequence independent of $(X_0^{i,n})_{i\in[n]}$, and $\zeta_k^{1,n} = 0$ for every $k$.
\end{remark}

{ Define $\kappa_k(\mu) = \int \|x\|^kd\mu(x)$.
\begin{assumption}
  \label{hyp:moment}
  The following conditions hold:
  \begin{enumerate}[i)]
  \item $\sup_{n}\bE\norm{X^{1,n}_0}^2 <\infty$ and $\sup_{k,n}\bE\norm{\zeta^{1,n}_k}^2 <\infty$,
    \item  There exist $c,C>0$, such that for all $\mu\in\cP_2(\bR^d)$,
      \begin{equation}\label{eq:dis1}
        \int\ps{x,b(x,\mu)}\dr\mu(x) \le -c\kappa_2(\mu) + C\,.
      \end{equation}
\end{enumerate}
\end{assumption}
\addtocounter{assumption}{-1}

\newcommand{\lastAssumption}{\arabic{assumption}}

\renewcommand{\theassumption}{\lastAssumption$'$}

\begin{assumption}
\label{hyp:mombis}
  In addition to Assumption~\ref{hyp:moment}, the following hold:
  \begin{enumerate}[i)]
  \item $\sup_{n}\bE\norm{X^{1,n}_0}^4 <\infty$ and $\sup_{k,n}\bE\norm{\zeta^{1,n}_k}^4 <\infty$,
    \item There exists constants $c, C>0$ such that for  all $\mu\in\cP_4(\bR^d)$,
    \begin{equation}\label{eq:dis2}
      \int  \ps{x,b(x,\mu)}\norm{x}^2 \dr \mu(x) \le   - c \kappa_4(\mu)+ C\p{1+\kappa_2(\mu) }\!\! \p{1+\sqrt{\kappa_4(\mu)}}\,,
    \end{equation}
\end{enumerate}
\end{assumption}

\renewcommand{\theassumption}{\arabic{assumption}}
Section \ref{sec:GM} includes an example for which Assumptions~\ref{hyp:moment} and \ref{hyp:moment}$'$ are satisfied.
\begin{remark}
Assumption~\ref{hyp:moment} can be replaced by the milder condition that 
\(\sup_{k,n} \mathbb{E}\| X_k^{1,n}\|^2 + \mathbb{E}\| \zeta_k^{1,n}\|^2 < \infty\). 
Similarly, Assumption~\ref{hyp:moment}$'$ can be replaced by the condition that 
\((\|X_k^{1,n}\|^2)_{k,n}\) and \((\|\zeta_k^{1,n}\|^2)_{k,n}\) are uniformly integrable. 
The results of this paper hold under these milder, but less easily verifiable assumptions.
\end{remark}



}

Recalling the definitions of the interpolated processes $\bar X^{i,n}$ 
in~\eqref{interp}, and the definition of the occupation measure $m^n$ in~\eqref{eq:m}, 
we shall consider the \emph{shifted} occupation measure
$$
\Phi_t(m^n) = \frac 1n\sum_{i=1}^n \delta_{\Theta_t(\bar X^{i,n})}\,,
$$
for each $n\in \bN^*$ and each $t\in (0,+\infty)$. Note that $\Phi_t(m^n)$ is a r.v. on $\cP_p(\cC)$. We refer to the set 
\begin{equation}
\label{M-acc} 
  \cM := 
\acc_{(t,n)\to(\infty,\infty)} 
\left(\left\{(\Phi_t(m^n))_\#\bP \right\}\right) 
\end{equation} 
as the set of {weak$\star$} accumulation points of the probability distributions of
$\Phi_t(m^n)$ as $(t,n)\to(\infty,\infty)$.  In other words, $\cM$ is the
set of measures $M\in \cP(\cP_p(\cC))$ for which there is a sequence
$(t_n,\varphi_n)_n$ on $(0,\infty)\times \bN^*$, such that $t_n\to_n\infty$,
$\varphi_n\to_n\infty$, and $(\Phi_{t_n}(m^{\varphi_n}))$ converges in
distribution to $M$.
{

We now state the main results of this paragraph.
Prop.~\ref{prop:tight} shows that the set
$\cM$ is non-empty.
Prop.~\ref{prop:limit-process} 
shows that any $M\in \cM$
is supported by the set of McKean-Vlasov distributions.
\begin{prop}\label{prop:tight}
    Let $1\le p <2$. Let  Assumptions~\ref{hyp:b}--\ref{hyp:moment} hold. Then, $$\sup_{k,n}\bE\| X_k^{1,n}\|^2<\infty\,.$$
    Moreover, for any $j\in \bN^*$, the family of measures $\{(\Phi_t(m^n))_\#\bP:t\geq 0,n\in \bN^*\}$ 
    is relatively compact in  $\cP(\cP_p(\cC))$.

If Assumption~\ref{hyp:moment}$\,'$  holds, the conclusion is still valid for $p=2$, and, moreover, $\sup_{k,n}\bE\| X_k^{1,n}\|^4<\infty$.
\end{prop}
\begin{proof}
  See Sec.~\ref{sec:tight}.
\end{proof}
}

\begin{prop}
\label{prop:limit-process}
Let $1\le p < 2$. Under Assumptions~\ref{hyp:b}--\ref{hyp:moment}, 
$\sV_p$ is a non-empty closed set. Moreover, $M(\sV_p)=1$ for every $M\in \cM$. 
If Assumption~\ref{hyp:moment}$\,'$  holds, the conclusion is still valid for $p=2$.
\end{prop}
{
  \begin{proof}
    See Sec.~\ref{sec:support}.
  \end{proof}}

\subsection{Ergodic convergence}\label{sec:lya}
We provide the proof of the following theorem in Sec.~\ref{sec:ergodicTH}.
\begin{theorem}
\label{th:BC}
 Let $1\le p < 2$. Under Assumptions~\ref{hyp:b}--\ref{hyp:moment}, {$\BC_p$ is non-empty.} 
  $$
  \frac 1t\int_0^t \sW_p(\Phi_s(m^n),\BC_p)\,\dr s \xrightarrow[(t,n)\to (\infty,\infty)]{\bP} 0\,,
  $$
If Assumption~\ref{hyp:moment}$\,'$  holds, the statement is still valid for $p=2$.
\end{theorem}

  {
  \begin{coro}
  \label{cor:S}
 Let $1\le p < 2$. Under Assumptions~\ref{hyp:b}--\ref{hyp:moment},
  $$
  \frac {\sum_{l=1}^k \gamma_l W_p( \mu_l^n,\BC_p^0)}{\sum_{l=1}^k \gamma_l} \xrightarrow[(k,n)\to (\infty,\infty)]{\bP} 0\,.
  $$
  The same statement holds if $W_p(\,\cdot\,,\,\cdot)$ is replaced by
  $W_p(\,\cdot\,,\,\cdot)^p$.
Finally, if Assumption~\ref{hyp:moment}$\,'$ holds, the conclusion is still valid for $p=2$.
\end{coro}
\begin{proof}
The proof is provided in Sec.~\ref{sec:ergodicCO}.  
\end{proof}
\begin{remark}
  The fact that the Birkhoff center \(\BC_p\) is non empty
  follows from the combination of Lem.~\ref{lem:tighM} and Lem.~\ref{lem:invariant}. Specifically, Lem.~\ref{lem:tighM} establishes the existence of measures, which, by Lem.~\ref{lem:invariant}, can only be supported by \(\BC_p\).
\end{remark}
\begin{remark}
  In simple cases, $\BC_p$ is reduced to a singleton, which corresponds to the unique stationary McKean-Vlasov distribution.
  For instance, 
  this happens under sufficient but strong assumptions on $b$ and $\sigma$, which ensure a \emph{uniform-in-time propagation of chaos} \cite{Chaintron_2022,malrieu2001logarithmic,durmus2020elementary}.
  We refer to Sec.~\ref{sec:GM} for a discussion.

  Besides this case, the  McKean–Vlasov process potentially admits multiple stationary measures. In such a case, $\BC_p$ contains multiple points.
  This scenario is common, and interesting regarding practical applications.
  A first example can be found in ~\cite{del2019uniform,herrmann2010non}, in the context of the Granular media equation, see also Rem.~\ref{remark:nonUnicity}.
      A second example is encountered in the case of consensus based optimization methods \cite{carrillo2018analyticalframeworkconsensusbasedglobal,Fornasier_2024}, where,
  under the assumption of a constant noise intensity $\sigma$, the limiting McKean-Vlasov process potentially admits several stationary measures.
  A third example, in the case $\sigma=0$, is given by Stein Variational Gradient Descent (SVGD) algorithm \cite{liu2017stein}.


\end{remark}


 }

{
  Finally, let us review some consequences  regarding linear functionals.
  Denote by $\mathrm{Lip}_L$ the set of Lipschitz continuous functions on $\bR^d\to \bR$,
  whose Lipschitz constant is no larger than $L\geq 0$.
  Define:
  $$
  \BC_p^0(f) := \left\{ \int fd\mu\,:\,\mu\in \BC_p^0\right\}\,.
  $$

\begin{coro}\label{cor:cor2}
  Let $1\le p < 2$, and let Assumptions~\ref{hyp:b}--\ref{hyp:moment} hold true. Then, for every $L\geq 0$
  \begin{align*}
 &           \sup_{f\in\mathrm{Lip}_L}\mathsf{d}\p{\frac {\sum_{i\in[n],l\in[k] } \gamma_lf(X^{i,n}_l)}{n\sum_{l\in[k]}\gamma_l   }, \mathrm{conv}(\BC_p^0(f))} \xrightarrow[(k,n)\to (\infty,\infty)]{\bP} 0\,. 
     \end{align*}
     The conclusion remains valid for $p=2$ under Assumption~\ref{hyp:moment}'.
\end{coro}
\begin{proof}
  See Sec.~\ref{sec:co2}.
\end{proof}

\subsection{The case of a unique recurrent point}

In this subsection, we will present additional results in the special case where the following assumption holds.
\begin{assumption}\label{hyp:sing}
    There exists $\rho^*\in\cP_p(\cC)$ such that $\BC_p \subset \{\rho^*\}$.
\end{assumption}
We observe that, under Assumptions~\ref{hyp:b}--\ref{hyp:moment}, $\BC_p$ is non-empty (see Th.~\ref{th:BC}). Consequently, under Assumptions~\ref{hyp:b}--\ref{hyp:sing}, we have $\BC_p = \{\rho^*\}$.


Let \( n \in \mathbb{N}^* \) and \( j \leq n \).  
One may consider the \emph{law} of the family of random variables \( (X^{1,n}_l, \dots, X^{j,n}_l) \):
\[
I^{j,n}_l := \big(X^{1,n}_l, \dots, X^{j,n}_l \big)_\#\mathbb{P} = \mathbb{P}\big( ( X^{1,n}_l, \dots, X^{j,n}_l ) \in \,\cdot\,\big).
\]
For instance, \( I^{1,n}_l \) is the law of the particle \( X^{1,n}_l \),  
which is equal to the law of \( X^{i,n}_l \) for any \( i \), due to the exchangeability.
\begin{coro}\label{cor:law}
     Under Assumptions~\ref{hyp:b}--\ref{hyp:sing}, we obtain for every $j\in\bN$
\begin{equation}\label{eq:ergLaw}
\frac{\sum_{\ell\in[k]} \gamma_\ell W_p\big( I^{j,n}_\ell, (\rho^*_0)^{\otimes j} \big)}{\sum_{l\in[k]} \gamma_\ell} \xrightarrow[(k,n) \to (\infty, \infty)]{} 0\,,
\end{equation}
where \( W_p \) denotes the Wasserstein distance of order \( p \) on \( \mathcal{P}_p((\mathbb{R}^d)^j) \), and \( (\rho^*_0)^{\otimes j} \) is the \( j \)-fold tensor product of \( \rho^*_0 \).  
\end{coro}
\begin{proof}
    See Sec.~\ref{sec:cor2}.
\end{proof}
Eq.~\eqref{eq:ergLaw} can be interpreted as a propagation of chaos result in the long run. This should be compared to standard propagation of chaos results, which are usually stated over a finite time interval \cite{Chaintron_2022}.

Following \cite{bolte2020longtermdynamicssubgradient,bianchi2023closedmeasureapproachstochasticapproximation}, let us introduce the notion of essential accumulation set.
We say that a measure $\mu \in \cP_p(\bR^d)$ is an \emph{essential accumulation point} of $(I^{1,n}_k)_{k,n}$,
if for every neighborhood $U$ of $\mu$, 
\begin{equation*}
\limsup_{ (k,n) \to (\infty,\infty)} {\frac{\sum_{\ell\in[k]} \gamma_\ell\un_U(I^{1,n}_{\ell})}{\sum_{\ell\in[k]} \gamma_\ell}}  >0\,.
\end{equation*}
This can be interpreted as follows.
An essential accumulation point $\mu$ is an accumulation point, with the property that the particle distribution $I^{1,n}_k= \bP(X^{1,n}_k\in \cdot)$ spends substantial time in the neighborhood of $\mu$.
\begin{coro}\label{cor:essacc}
    Under Assumptions~\ref{hyp:b}--\ref{hyp:sing},
$\rho^*_0$ is the unique essential accumulation point of $(I^{1,n}_k)_{k,n}$.
\end{coro}
\begin{proof}
    See Sec.~\ref{sec:cor3}.
\end{proof}
In other terms, as $(k,n)$ tend to infinity,
the law $\bP(X^{1,n}_k\in\cdot)$ spends most of its time in the neighborhood of $\rho_0^*$.

}
\subsection{Pointwise convergence to a global attractor} 
\label{cvg-A} 

Depending on the vector field $b$, it is often the case that each measure 
$\rho \in \sV_p$ is uniquely determined by its value 
$\rho_0 = (\pi_0)_\# \rho \in \cP_p(\bR^d)$ in the sense that there exists 
a semi-flow $\Psi : [0,\infty) \times \cP_p(\bR^d) \to \cP_p(\bR^d), 
(t,\nu) \mapsto \Psi_t(\nu)$, defined on $[0,\infty) \times \cP_p(\bR^d)$, and 
such that 
\begin{equation}
\label{flot-P(R)} 
\rho \in \sV_p \Leftrightarrow \forall t \geq 0, \rho_t = \Psi_t(\rho_0) . 
\end{equation} 
We shall say that in this situation, the class $\sV_p$ has a semi-flow 
structure on $\cP_p(\bR^d)$. 

The granular media model detailed in Sec.~\ref{sec:GM} below is a typical
example where such a situation occurs. 

In this section, we are interested in the behavior of the measures $\mu_k^n$ as
$(k,n) \to (\infty,\infty)$, termed the ``pointwise'' convergence of these
measures, when the semi-flow $\Psi$ has a global attractor.  We recall here
that a set $A_p \subset \cP_p(\bR^d)$ is said invariant for the semi-flow
$\Psi$ if $\Psi_t(A_p) = A_p$ for all $t \geq 0$; A nonempty compact invariant
set $A_p \subset \cP_p(\bR^d)$ is a global attractor for the semi-flow $\Psi$
if 
\[
\forall \nu\in \cP_p(\bR^d), \quad 
  \lim_{t\to \infty } W_p(\Psi_t(\nu), A_p ) = 0  \,,
\]
and furthermore, if there exists a neighborhood $\cN$ of $A_p$ in
$\cP_p(\bR^d)$ such that this convergence is uniform on $\cN$. Such a
neighborhood is called a fundamental neighborhood of $A_p$. 

The following result is proven in Sec.~\ref{prf-ptwise}.  
\begin{theorem}
\label{th:pwConvergence}
Let $p\in [1,2]$, and let Assumptions~\ref{hyp:b}, \ref{hyp:gamma},
and~\ref{hyp:algo} hold true.  Let
Assumption~\ref{hyp:moment} or the stronger
Assumption~\ref{hyp:moment}\,' hold true according to
whether $p < 2$ or $p=2$ respectively. Assume in addition that the 
$\sV_p$ has a semi-flow structure on $\cP_p(\bR^d)$ as specified 
in~\eqref{flot-P(R)}, and that this semi-flow $\Psi$ admits a global attractor 
$A_p$. Then,  
\[
W_p\left( \mu_k^n,A_p \right)
   \xrightarrow[(k,n)\to (\infty,\infty)]{\bP} 0\,.
\]
\end{theorem}
{
  \begin{remark}
  A typical scenario where the set $A_p$ exists and contains a single element is provided
  in Rem.~\ref{remark:nonUnicity}.
\end{remark}
}


\section{Granular media}
\label{sec:GM}
The proofs of the results relative to this section are  provided in Sec.~\ref{sec:proofGM}.

{In this paragraph, we review some properties of the set $\sV_2$ of McKean-Vlasov processes, in the case where {  $\sigma(x,\mu)=\sigma I_d$ for some real constant $\sigma\ge 0$ } and with a slight abuse of notation $b(x,\mu)=\int b(x,y) d\mu(y)$, with: }
\begin{equation}
b(x,y) := -\nabla V(x) - \nabla U(x-y)\,,\label{eq:bgm}
\end{equation}
where $V, U: \mathbb{R}^d \to \mathbb{R}$ are two functions satisfying the following assumption.
\begin{assumption}[Granular media] 
  \label{hyp:GM}
   The functions $V,U$ belong to $C^1(\bR^d,\bR)$. Moreover, there exists $\lambda,C,\beta>0$, such that for every $x,y\in\bR^d$, the following hold{s}:
  \begin{enumerate}[i)]
    \item $\ps{x,\nabla V(x)} \ge \lambda\norm{x}^2 -C$,
    \item $U(x)=U(-x)$, and $\ps{x,\nabla U(x)} \ge -C $,
    \item $\norm{\nabla V(x)} + \norm{\nabla U(x)} \le C(1+\norm{x})$,
    \item $\norm{\nabla V(x) - \nabla V(y)}+\norm{\nabla U(x)-\nabla U(y)} \le C(\norm{x-y}^\beta \vee \norm{x-y})$.
  \end{enumerate}
\end{assumption}
Under Assumptions~\ref{hyp:GM}, the vector field $b$ ans $\sigma$ satisfies Assumption~\ref{hyp:b}.
We will see later, as a byproduct of Th.~\ref{th:gm}, that the set $\sV_2$ of McKean-Vlasov distributions
associated to the field $b$ in Eq.~(\ref{eq:bgm}), is non empty.
We say $\mu\ll \mathscr L^d$ if $\mu\in\cP_2(\bR^d)$ admits continuously differentiable
density w.r.t. the Lebesgue measure $\mathscr{L}^d$, which we denote by $d\mu/d\mathscr{L}^d$.
Define the functional
$\mathscr H : \cP_2(\bR^d) \to (-\infty, \infty]$ as
$\mathscr H(\mu) = \mathscr F(\mu) + \mathscr V(\mu) + \mathscr U(\mu)$
with
\[
\mathscr F(\mu) = \left\{\begin{array}{ll} 
  \displaystyle{\int \sigma^2 \log\p{ \dd \mu {\mathscr L^d}(x)}\, \dr \mu(x) } 
 & \text{if } \mu \ll \mathscr{L}^d \\
 \infty & \text{otherwise} ,
\end{array}\right.   
\]
\[
\mathscr V(\mu) = \int V(x) \, d\mu(x), \quad \text{and} \quad 
\mathscr U(\mu) = \frac 12 \int\!\!\!\int U(x-y) \, d\mu(x)d\mu(y)  . 
\]
The following central result provides a central properties of the elements of $\sV_2$.
\begin{prop}
  \label{prop:helmholtz}
  Let Assumption~\ref{hyp:GM} hold true, and let $b$ be defined by~(\ref{eq:bgm}). Assume $\sigma>0$.
  Consider $\rho\in\sV_2$. Then, for every $t>0$,
  $\rho_t$ admits a density $x\mapsto \varrho(t,x)$ in $C^1(\bR^d,\bR)$ w.r.t. the Lebesgue measure. For every $t>0$, the functional $t\mapsto \mathscr{H}(\rho_t)$
  is finite, and satisfies for every $t_2>t_1>0$, 
\begin{equation}
  \label{eq:continuityBis-limit} 
 \mathscr H(\rho_{t_2}) -\mathscr H(\rho_{t_1})  =    - \int_{t_1}^{t_2} \int  \|v_t(x)\|^2 \varrho(t,x)dx dt \,,
\end{equation}
where $v_t$ is the vector field defined for every $x\in \bR^d$ by:
  \begin{equation}
  v_t(x) := -\nabla V(x)-\int \nabla U(x-y)d\rho_t(y)-\sigma^2\nabla \log\varrho(t,x)\,.
  \label{eq:vt}
\end{equation}
\end{prop}
Define $\cP_2^r(\bR^d)$ as the set of measures $\mu\in \cP_2(\bR^d)$ such that $\mu\ll\mathscr L^d$. Define:
  \begin{equation}
    \label{eq:S2gm}
  \cS \eqdef \{\mu\in \cP_2^r(\bR^d)\,:\, \nabla V + \int \nabla U(\,\cdot\,-y)d\mu(y) + \sigma^2 \nabla \log \frac{d\mu}{d\mathscr{L}^d} = 0\,\mu\text{-a.e.}\}\,.
\end{equation}

Finally, for every $\epsilon\geq 0$, define:
  \begin{equation}
    \label{eq:Lambda}
    \Lambda_\epsilon := \{\rho\in \sV_2\,:\,\exists\mu\in \cS,\, \forall t\geq \epsilon,\,\rho_t = \mu\}\,.
  \end{equation}

\begin{prop}
  \label{prop:gm}
  We posit the assumptions of Prop.~\ref{prop:helmholtz}. 
  For every $\epsilon>0$, the function $\rho\mapsto \mathscr H(\rho_\epsilon)$ is real valued on $\sV_2$, lower semicontinuous, 
  and is a Lyapunov function for the set $\Lambda_\epsilon$. Moreover, $$\BC_2\subset \overline{\Lambda_0}\,.$$
\end{prop}

We also need to consider a setting where $\sV_2$ has a semi-flow structure on
$\cP_2(\bR^d)$ as in~\eqref{flot-P(R)} in order to set the stage for the
pointwise convergence of the measures $\mu^n_k$ issued from our discrete 
algorithm. To that end, we shall appeal to the theory of the gradient flows in 
the space of probability measures as detailed in the treatise 
\cite{ambrosio2005gradient} of Ambrosio, Gigli and Savar\'e. The following 
additional assumption will be needed:
\begin{assumption}
\label{ass-gf}
The functions $U$ and $V$ satisfy the doubling condition. Namely, there
exists constants $C_U, C_V > 0$ such that
\[
U(x+y) \leq C_U \left( 1 + U(x) + U(y) \right) \quad \text{and} \quad 
V(x+y) \leq C_V \left( 1 + V(x) + V(y) \right). 
\]
\end{assumption}
\begin{prop} 
\label{GF-GM} 
Let Assumption~\ref{hyp:GM} hold true with $\beta = 1$, and let
Assumption~\ref{ass-gf} hold true. Then, for each $\rho \in \sV_2$, the curve 
$t \mapsto \rho_t$ belongs to the set of absolutely continuous functions 
$\AC_{\text{loc}}^2((0,\infty), \cP_2(\bR^d))$ as defined in 
\cite[Sec.~8.3]{ambrosio2005gradient}, and is completely determined by 
$\rho_0 \in \cP_2(\bR^d)$ as being the gradient flow of the functional 
$\mathscr H$ in $\cP_2(\bR^d)$. Thus, $\sV_2$ has a semi-flow structure, and
we write $\rho_t = \Psi_t(\rho_0)$. 
\end{prop} 
For completeness, we recall along \cite[Chap.~8 and 11]{ambrosio2005gradient}
that $t \mapsto \rho_t$ being the solution of the gradient flow of 
$\mathscr H$ in $\cP_2(\bR^d)$ stands to the existence of a Borel vector field
$w_t : \bR^d \to \bR^d$ such that $w_t$ belongs to the tangent bundle 
$\Tan_{\rho_t} \cP_2(\bR^d)$ for $\mathscr{L}^1$--almost all $t > 0$, 
$\| w_t \|_{L^2(\rho_t)} \in L^p_{\text{loc}}(0,\infty)$, the continuity
equation $\partial_t \rho_t + \nabla \cdot \left( \rho_t w_t \right) = 0$
holds in general in the sense of distributions, and finally, 
$w_t \in - \partial \mathscr H(\rho_t)$ for $\mathscr{L}^1$--almost each 
$t > 0$, where $\partial \mathscr H$ is the Fréchet sub-differential as 
defined in \cite[Chap.~10]{ambrosio2005gradient}, which always exists under
our assumptions. Actually, $w_t = v_t$ as given by Equation~\eqref{eq:vt} 
for almost all $t$. 

We now turn to our discrete algorithm. Consider the iterations:
\begin{equation}
  X^{i,n}_{k+1} = X_k^{i,n} -\frac{\gamma_{k+1}}{n} \sum_{j\in[n]} \nabla U(X^{i,n}_k - X^{j,n}_k) - \gamma_{k+1} \nabla V(X^{i,n}_k) + \sqrt{2\gamma_{k+1}}\xi^{i,n}_k\,,
  \label{eq:gm}
\end{equation}
for each $i\in[n]$. This is a special case of Eq.~(\ref{eq:main}) with $b(x,y)$ given by Eq.~(\ref{eq:bgm}) and
$\zeta_k^{i,n} = 0$ for all $k$.
For simplicity, Assumption~\ref{hyp:algo} will be replaced by the following stronger assumption:
\begin{assumption}
  \label{hyp:algo-sans-zeta}
  We assume that the n-tuple $(X_0^{1,n},\dots, X_0^{n,n})$ is exchangeable and $\sup_n\bE(\|X_0^{1,n}\|^4)<\infty$.
  Moreover,  $(\xi_k^{i,n})_{i\in [n],k\in \bN}$ are i.i.d. centered random variables, with variance $\sigma^2 I_d$,  and such that $\bE(\|\xi_1^{1,1}\|^4)<\infty$.
\end{assumption}
The next proposition implies that Assumption~\ref{hyp:mombis} holds.
{\begin{prop}\label{prop:stabGM}
Let Assumptions~\ref{hyp:gamma},~\ref{hyp:GM} and~\ref{hyp:algo-sans-zeta} be satisfied. Then, Eq.~\eqref{eq:dis1} and~\eqref{eq:dis2} hold.\end{prop}}
Putting Assumptions~\ref{hyp:gamma},~\ref{hyp:GM} and~\ref{hyp:algo-sans-zeta} together, the hypotheses of Th.~\ref{th:BC} are satisfied for $p=2$.
\begin{theorem}
  \label{th:gm}
  Let Assumptions~\ref{hyp:gamma},~\ref{hyp:GM} and~\ref{hyp:algo-sans-zeta} be satisfied. Assume $\sigma>0$.
  Then, the set $\cS$ given by Eq.~\eqref{eq:S2gm} is non empty, and 
 furthermore,
  $$
  \frac {\sum_{l=1}^k \gamma_l W_2( \mu_l^n,\cS)}{\sum_{l=1}^k \gamma_l} \xrightarrow[(k,n)\to (\infty,\infty)]{\bP} 0\,.
  $$
\end{theorem}

\begin{proof}
  Use Cor.~\ref{cor:S} with $p = 2$, 
  together with Prop.~\ref{prop:gm}.
\end{proof}

We now turn to the pointwise convergence of the measures $\mu^n_k$. 
\begin{theorem}
\label{th-PW-GM} 
Let Assumption~\ref{hyp:GM} hold true with $\beta = 1$, and let
Assumption~\ref{ass-gf} hold true. Assume that the semi-flow $\Psi$ which
existence is stated by Prop.~\ref{GF-GM} has a global attractor 
$A_2$. In the case where $A_2$ is a singleton, it holds that $\cS = A_2$. In 
any case,  
\[
W_2\left( \mu^n_k,A_2 \right)
   \xrightarrow[(k,n)\to (\infty,\infty)]{\bP} 0 . 
\]
\end{theorem} 
{\begin{remark}\label{remark:nonUnicity}
Many authors have
  been interested in the long-time convergence of granular media
  equations under hypotheses ensuring the uniqueness of
  the stationary distribution \cite{carrillo2003kinetic,carrillo2006contractions,cattiaux2008probabilistic,
    bolley2013uniform,gui-liu-wu-zha-aap22}.
  The most obvious case where such a situation arises, is the case where
  the functions $U$ and $V$ are both strongly convex.; Then, there exists 
$\lambda > 0$ such that $W_2(\Psi_t(\nu), \Psi_t(\nu')) \leq e^{-\lambda t}
W_2(\nu, \nu')$ \cite[Th.~11.2.1]{ambrosio2005gradient}. Here,
Th.~\ref{th-PW-GM} applies, with $A_2$ being reduced to the unique stationary measure.
  
On the other hand, the coexistence of multiple stationary measures typically corresponds to
the case of metastable behaviors, where the Helmholtz energy admits several
    critical points. 
  For instance, this situation arises in the case
  of a multi-well potential with low noise intensity~\cite{herrmann2010non,carrillo2020long}.
  Although it can be challenging to characterize such  phase transition phenomena,
  our work supports the
    assertion that a numerical system with \( n \) particles provides
    an estimate, in the sense that the $n$-system inherits the same
    asymptotic behavior as its mean-field approximation.  
\end{remark}
}

\section{Proofs of Sec.~\ref{sec:main}}
{
\subsection{Proof of Prop.~\ref{prop:tight}} \label{sec:tight}

In this paragraph, consider $1\le p\le 2$.
We recall that, when $p<2$, Assumption~\ref{hyp:moment}
holds and Assumption~\ref{hyp:mombis} holds when $p=2$.
First, we need the following lemma.
\begin{lemma}\label{lem:moment}

Let Assumption~\ref{hyp:moment} with Assumptions~\ref{hyp:b},~\ref{hyp:gamma}, and~\ref{hyp:algo} hold true, it holds that \(\sup_{k,n} \bE\|X_k^{1,n}\|^2 < \infty\). Furthermore, when Assumption~\ref{hyp:mombis} holds, we have \(\sup_{k,n} \bE\|X_k^{1,n}\|^4 < \infty\).

\end{lemma}
\begin{proof}
    First, we will show the first point of the lemma.
We recall the iteration
\[
    X_{k+1}^{i,n} = X_k^{i,n} + \gamma_{k+1}b(X_k^{i,n},\mu_k^n) + \gamma_k\zeta_k^{i,n} + \sqrt{2\gamma_{k+1}}\xi^{i,n}_{k+1}.
\]
In this proof, we denote by \( C > 0 \) a generic constant that is sufficiently large, and by \( c > 0 \) a generic constant that is sufficiently small. We take \( k \) large enough such that  
\( \gamma_k + \gamma_k^2 \leq C \gamma_k \) and \( -c\gamma_k + \gamma_k^2 \leq -c\gamma_k \). 
For simplicity, we remove the superscript \({}^n\) from \(X_k^{i,n}\), $\mu_k^n$, \(\zeta_k^{i,n}\), and \(\xi_k^{i,n}\). Moreover, we remove the subscript \({}_{k+1}\) from \(\gamma_{k+1}\).

By Assumption~\ref{hyp:b}, for $i\in [n]$, we obtain
\begin{equation}\label{eq:etape0}
\begin{split}
   & \norm{X^{i}_{k+1}}^2-\norm{X^{i}_{k}}^2 \\&= \gamma \ps{X^{i}_k, b(X^i_k, \mu_k)} +  \sqrt{2\gamma}\ps{X^i_k, \xi^i_{k+1}} +\gamma\ps{X^{i}_{k}, \zeta^i_{k+1}} \\ &+ \norm{ \gamma b(X^i_k,\mu_k) + \sqrt{2\gamma}\xi^i_{k+1} + \gamma\zeta^i_{k+1}  }^2 \\
    &\le \gamma \ps{X^{i}_k, b(X^i_k, \mu_k)} +  \sqrt{2\gamma}\ps{X^i_k, \xi^i_{k+1}} +\gamma\ps{X^{i}_{k}, \zeta^i_{k+1}}+ \\
    & 6\gamma^2 \norm{X^{i}_{k}}^2 + 6\gamma^2 \int \norm{x}^2 \dr \mu_k(x) + 6\gamma \norm{\xi^i_{k+1}}^2 + 3\gamma^2\norm{\zeta_{k+1}^i}^2.
\end{split}
\end{equation}
Summing the latter with respect to \( i \), with Eq.~\eqref{eq:dis1}, we obtain
\begin{equation}\label{eq:etape1}
\begin{split}
   &\frac 1C\p{\int \norm{x}^2\dr\mu_{k+1}(x)-\int \norm{x}^2\dr\mu_k(x)}\\& \le \gamma \int \ps{x, b(x, \mu_k)} \dr\mu_k(x) +  \frac{\sqrt{\gamma}}n\sum_{i\in[n]} \ps{X^i_k, \xi^i_{k+1}}
   +\frac\gamma n \sum_{i\in[n]} \ps{X_k^i, \zeta^i_{k+1}} \\
    &+ \gamma^2 \int \norm{x}^2 \dr \mu_k(x) +  \frac 1n\sum_{i\in[n]}\p{\gamma\norm{\xi^i_{k+1}}^2 + \gamma^2\norm{\zeta_{k+1}^i}}\\
    &\le -c\gamma \int \norm{x}^2\dr\mu_k(x) + \frac{\sqrt{\gamma}}n\sum_{i\in[n]} \ps{X^i_k, \xi^i_{k+1}}
   +\frac\gamma n \sum_{i\in[n]} \ps{X_k^i, \zeta^i_{k+1}} \\
    &+ \gamma^2 \int \norm{x}^2 \dr \mu_k(x) +  \frac 1n\sum_{i\in[n]}\p{\gamma\norm{\xi^i_{k+1}}^2 + \gamma^2\norm{\zeta_{k+1}^i}^2} +C\gamma \,. \\
\end{split}
\end{equation}
Taking the expectation, by the exchangeability given by Assumption~\ref{hyp:algo}, the assumption on \((\zeta^i_k)_{i,k}\), and Assumption~\ref{hyp:algo}, we obtain
\begin{equation*}
     \bE{\norm{X^{1}_{k+1}}^2}-\bE{\norm{X^{1}_{k}}^2}\le -c\gamma \bE\norm{X^1_k}^2 + C\gamma\,. \\
\end{equation*}
As a consequence, we obtain the first point of the lemma.

Now, we proceed to demonstrate the second point of the lemma. But fist we claim that
\begin{equation}\label{eq:boundeq}
    \sup_{k\in\bN} \bE\p{\int\norm{x}^2\dr\mu_k(x)}^2 <\infty.
\end{equation}
Indeed, by raising to the square Eq.~\eqref{eq:etape1} and taking the expectation, we obtain
\[
\begin{split}
   &\bE\p{\int \norm{x}^2\dr\mu_{k+1}(x)}^2-\bE\p{\int \norm{x}^2\dr\mu_k(x)}^2\\
    &\le -c\gamma \bE\p{\int \norm{x}^2\dr\mu_k(x)}^2  
    + C\gamma\bE\p{\int \norm{x}^2 \dr \mu_k(x)}+C\gamma \, . 
\end{split}
\]
Now, we will obtain the second point of the lemma. By raising to the square Eq.~\eqref{eq:etape0}, we obtain
\[
\begin{split}
    & \frac 1C (\norm{X^{i}_{k+1}}^4-\norm{X^{i}_{k}}^4) \\& \le \gamma \ps{X^{i}_k, b(X^i_k, \mu_k)}\norm{X^{i}_{k}}^2 +  \sqrt{\gamma}\ps{X^i_k, \xi^i_{k+1}} \norm{X^{i}_{k}}^2+\gamma\ps{X^{i}_{k}, \zeta^i_{k+1}}\norm{X^{i}_{k}}^2+ \\
    & \gamma^2 \norm{X^{i}_{k}}^4 + \gamma^2 \int \norm{x}^2 \dr \mu_k(x) \norm{X^{i}_{k}}^2 + \gamma \norm{\xi^i_{k+1}}^2\norm{X^{i}_{k}}^2 + \gamma^2\norm{\zeta_{k+1}^i}^2\norm{X^{i}_{k}}^2\\
    &\gamma^4 \norm{X^{i}_{k}}^4 + \gamma^4 \p{\int \norm{x}^2 \dr \mu_k(x) }^2 + \gamma^2 \norm{\xi^i_{k+1}}^4 + \gamma^4\norm{\zeta_{k+1}^i}^4\,.
\end{split}
\]
Summing over $i\in[n]$, we obtain
\[
\begin{split}
    & \frac 1C \p{\int \norm{x}^4\dr\mu_{k+1}(x)-\int \norm{x}^4\dr\mu_{k}(x)} \\& \le \gamma \int \ps{x, b(x, \mu_k)}\norm{x}^2\dr\mu_k(x) +  \frac {\sqrt{\gamma}}{n} \sum_{i\in[n]} \ps{X^i_k, \xi^i_{k+1}} \norm{X^i_k}^2 +\\ &\frac \gamma n\sum_{i\in[n]}\ps{X^i_k, \zeta^i_{k+1}}\norm{X^i_k}^2+ 
     \gamma^2 \int \norm{x}^4\dr\mu_k(x) + \gamma^2 \p{\int \norm{x}^2 \dr \mu_k(x)}^2  +\\& \frac \gamma n\sum_{i\in[n]} \norm{\xi^i_{k+1}}^2 \norm{X^i_n}^2+ \frac{\gamma^2}n\sum_{i\in[n]} \norm{\zeta_{k+1}^i}^2\norm{X^i_k}^2+\gamma^4 \int \norm{x}^4\dr \mu_k(x) + \\
    &\gamma^4 \p{\int \norm{x}^2 \dr \mu_k(x) }^2 + \frac{\gamma^2}n \sum_{i\in[n]} \norm{\xi^i_{k+1}}^4 + \frac{\gamma^4}n \sum_{i\in[n]}\norm{\zeta_{k+1}^i}^4\,.
\end{split}
\]
Taking the expectation, by Eq.~\eqref{eq:dis2}, and by the assumption on $(\zeta_k^i)_{k,i}$, we obtain
\[
\begin{split}
      & \frac 1C \p{\bE \norm{X_{k+1}^1}^4- \norm{X^1_k}^4} \\& 
      \le -c \gamma\bE\norm{X_k^1}^4 + \gamma\bE \p{\int \norm{x}^2 \dr\mu_k(x) \p{\int \norm{x}^4\dr\mu_k(x)}^{1/2}}
      \\ &  
     +\gamma^2 \bE \norm{X_k^1}^4 + \gamma^2 \bE\p{\int \norm{x}^2 \dr \mu_k(x)}^2 + 
      \gamma\bE\norm{X^1_n}^2 \\&+ \gamma \bE\p{\int \norm{x}^4\dr\mu_k(x)}^{1/2} +{\gamma} 
       \,.  
\end{split}
\]
Cauchy-Schwartz inequality yields
\[
\begin{split}
      & \frac 1C \p{\bE \norm{X_{k+1}^1}^4- \norm{X^1_k}^4} \\&
      \le  -c \gamma\bE\norm{X_k^1}^4 + \gamma\p{\bE \p{\int \norm{x}^2 \dr\mu_k(x)}^2}^{1/2}\p{\bE\norm{X_k^1}^4 }^{1/2}  \\
      & +\gamma^2 \bE \norm{X_k^1}^4 + \gamma^2 \bE\p{\int \norm{x}^2 \dr \mu_k(x)}^2 + 
      \gamma\p{\bE\norm{X^1_n}^4}^{1/2} + {\gamma}
       \,.  
\end{split}
\]
Finally, by Eq.~\eqref{eq:boundeq}, we obtain
\[
\begin{split}
      & \frac 1C \p{\bE \norm{X_{k+1}^1}^4- \norm{X^1_k}^4} \\&
      \le  -c \gamma\bE\norm{X_k^1}^4  +\gamma^2 \bE \norm{X_k^1}^4 + 
      \gamma\p{\bE\norm{X^1_n}^4}^{1/2} + {\gamma}
       \,,
\end{split}
\]
which concludes the proof. 
\end{proof}

Note that $(\Phi_t(m^n))$ belongs to $\cP_p(\cC)$.

In the light of
Lem.~\ref{lem:meleard} and Prop~\ref{prop:space}, we should establish two points: first, the 
weak$\star$-relatively compactness of the family of intensities 
$\{\bI(\Phi_t(m^n))\}_{t,n}$; second, a uniform integrability condition
of the $p$th order moments of the measures $\bI(\Phi_t(m^n)(x))$. These results are respectively
stated in Lem.~\ref{lem:Iui} and \ref{lem:TMui} below.
\begin{lemma}
  \label{lem:Iui}
  We posit the assumptions of Prop.~\ref{prop:tight}.
   The family of intensities 
 $\{\bI(\Phi_t(m^n))\}_{t,n}$ is weak$\star$-relatively compact
 in $\cP(\cC)$.
\end{lemma}
\begin{proof}
Let us establish the first point. 
For every bounded continuous function $\phi : \cC \to \bR$, we have 
\[
\begin{split}
  \bI(\Phi_t(m^n)))(\phi) := 
  \esp{\int \phi(x) \dr \p{\Phi_t(m^n) (x)}} &= 
 \frac 1n \sum_{i\in[n]} \esp{\phi(\bar X^{i,n}_{t+\cdot})} \\& = \esp{\phi(\bar X^{1,n}_{t+\cdot})} \,,
\end{split}
\]
where we used the exchangeability stated 
in~Assumption~\ref{hyp:algo}-\eqref{hypenum: exchangeable}. Let us define the
measure $\hat \bI^n_t \in \mathcal{P}(\mathbb{R}^d)$ as 
\[
 \hat \bI^n_t(\phi) := \mathbb{E}\left[\psi(\bar X_t^{1,n})\right] ,
\] 
for each measurable function $\psi : \bR^d \to \bR_+$.
According to Th. 7.3 in~\cite{billingsley2013convergence}, the weak$\star$-relative compactness of the sequence \( (\bI^n_t)_{t,n} \) in $\cP(\cC)$ is guaranteed
if and only if the weak$\star$-relative compactness of \( (\hat \bI^n_t)_{t,n} \) in $\cP(\bR^d)$ is ensured, and if the following equicontinuity condition 
\begin{equation}\label{eq:bil}
  \lim_{\delta\to 0} \limsup_{t,n} \mathbb{P}\left(w^{T}_{\bar X^{1,n}_{t+\cdot}}(\delta) \geq \varepsilon\right) = 0 
\end{equation}
is met for every $\varepsilon,T > 0$, where $w^{T}_x(\delta)$ is the modulus 
of continuity of a function $x$ on the interval $[0,T]$.
The weak$\star$-relative compactness of $(\hat \bI^n_t)_{t,n}$ in $\cP(\bR^d)$, follows directly from Lem.~\ref{lem:moment}.
Using the notation $k_t  := \inf \{ k\,:\, \sum_{i=1}^{k}\gamma_{i} \ge t  \}$,
and using the definition in Eq.~\eqref{eq:main}, we obtain the decomposition:
\begin{align}\label{eq:decomp}
  &\bar X_t^{1,n} -\bar X_s^{1,n} =  P_{s,t}^n+ N_{s,t}^n+  U_{s,t}^n \,,\\
  &P_{s,t}^n :=\left . \sum_{k = k_s}^{k_t-2} {\gamma_{k+1}}b(X_k^{1,n},\mu_k^n)  \right. \nonumber \\ &\phantom{P_{s,t}^n :=} + \left. \p{\tau_{k_s} - s}b(X_{{k_s-1}}^{1,n},\mu_{{k_s-1}}^{n}) + \p{\tau_{k_t} - t}b(X_{k_t-1}^{1,n},\mu_{k_t-1}^{n})\right .\nonumber\\
  &N_{s,t}^n := \sum_{k = k_s}^{k_t-2}\sqrt{\gamma_{k+1}}\xi^{1,n}_{k+1}   +  \frac{\tau_{k_s} - s}{\gamma_{k_s}}\sqrt{\gamma_{k_s}}\xi^{1,n}_{k_s} + \frac{\tau_{k_t} - t}{\gamma_{k_t}}\sqrt{\gamma_{k_t}}\xi^{1,n}_{k_t} \nonumber\\
&U_{s,t}^n :=
  \sum_{k = k_s}^{k_t-2} {\gamma_{k+1}} \zeta_{k+1}^{1,n} 
  + \p{\tau_{k_s} - s} \zeta_{k_s}^{1,n} + 
   \p{\tau_{k_t} - t} \zeta_{k_t}^{1,n} \,.\nonumber
\end{align}
Let the sequence \( (\tilde{\gamma}_k) \) be defined by:
$\tilde \gamma_{k_s} := \tau_{k_s} -s$,
$\tilde \gamma_{k_t} := \tau_{k_t} -t$
and $\tilde \gamma_{k} := {\gamma}_{k}$ for all $k \neq k_{t_s}, k_{t_t}$.
Note that:
\begin{equation}\label{eq:tildeGamma}
    \sum_{k=k_s-1}^{k_t-1} \tilde \gamma_{k+1} = t-s \,.
\end{equation}
Moreover, we have: \begin{equation}\label{eq:bTau}
  \frac{\tau_{k_s} - s}{\gamma_{k_s}}\sqrt{\gamma_{k_s}}\le \sqrt{\tilde \gamma_{k_s}}\,,\quad \text{and} \quad \frac{\tau_{k_t} - t}{\gamma_{k_t}}\sqrt{\gamma_{k_t}}\le \sqrt{\tilde \gamma_{k_t}}\,.
\end{equation}
The term \(N_{s,t}^n\) is expressed as a sum of martingale increments, with respect to the filtration $\mathcal{F}_{k}^n$.
Let \(\| \cdot \|_\alpha\) denote the \(\alpha\)-norm in \(\mathbb{R}^d\).
We apply Burkholder's inequality stated in ~\cite[Th. 1.1]{burkholder1972integral} to the components of the vector \(N_{s,t}^n\) in $\bR^d$. As Eq.~\eqref{eq:tildeGamma} and~\eqref{eq:bTau} hold:
\[
  \bE\p{\norm {N_{s,t}^n}_4^{4}} \le C(t-s) \esp{ { \sum_{k = k_s-1}^{k_t-1}{\tilde \gamma_{k+1}}\norm{\xi^{1,n}_{k+1}}^4_4}}\,,
\] 
where $C$ is a constant independent $s,t$ and $n$. 
As Assumption~\ref{hyp:algo}-\eqref{hypenum: va} holds, there exists a constant $C>0$ independent of $s,t$, and $n$, such that 
\begin{equation}\label{eq:N}
 \sup_{n\in\bN} \bE\p{\norm {N_{s,t}^n}^{4}} \le C(t-s)^{2}\,.
\end{equation}
Furthermore, using Jensen's inequality along with Eq.~\eqref{eq:tildeGamma}, 
we obtain 
\begin{equation*}
  {\norm {P_{s,t}^n}^2} \le {(t-s)}
  \sum_{k = k_s-1}^{k_t-1}\tilde \gamma_{k+1}{ {\norm{b(X^{1,n}_k,\mu_k^{n})}^2}}\,.
\end{equation*}
Using Assumptions~\ref{hyp:b} and Lem.~\ref{lem:moment}, there exists a constant 
$C$, independent of $s,t,n$, such that 
\begin{equation}\label{eq:P}
  \sup_{n\in\bN} \bE\p{\norm {P_{s,t}^n}^{2}} \le C (t-s)^2\,.
\end{equation}
Also, by Jensen's inequality, we have 
\begin{equation*}
  {\norm {U_{s,t}^n}^2} \le {(t-s)}
  \sum_{k = k_s-1}^{k_t-1}
  \tilde \gamma_{k+1}\norm{\zeta_{k+1}^{1,n}}^2\,.
\end{equation*}
Since, by Assumption~\ref{hyp:algo}, we have 
$\sup_{k,n} \bE{[\|{\zeta_k^{1,n}}\|^2]} <\infty$, 
there exists a constant $C$ independent of $n,s$, and $t$, such that:
\begin{equation}\label{eq:U}
  \sup_{n\in\bN} \bE\p{\norm {U_{s,t}^n}^{2}} \le C (t-s)^2\,.
\end{equation}
Combining Equations \(\eqref{eq:P}\), \eqref{eq:N} and \(\eqref{eq:U}\), we have shown:
\begin{equation}
 \label{eq:Kolmogorov}
  \sup_{n\in \bN} \esp{\norm{P^n_{s,t}}^{2} +\norm{N_{s,t}^n}^{4}+\norm{U_{s,t}^n}^{2}} \le C {\p{t-s}^{2}}\,,
\end{equation}
where $0\le s<t<\infty$, and $C$ is a positive constant, independent of $s,t,n$.
Using \cite[Th. 2.8]{leoni2023first} and Markov's inequality, Eq.~(\ref{eq:bil}) hold.
\end{proof}

\begin{lemma}
  \label{lem:TMui}
  We posit the assumptions of Prop.~\ref{prop:tight}.
For every $T>0$,  \[
\lim_{a\to\infty}\sup_{t\in\bR_+,\, n\in \bN^*} \esp{\int \sup_{s\in [0,T]} \norm{x_s}^p
   \indicatrice_{\sup_{s\in [0,T]} \norm{x_s} \ge a}\dr \Phi_t(m^n)(x) }=0.
\] 
\end{lemma}
\begin{proof}
  By the exchangeability stated in Assumption~\ref{hyp:algo}-\eqref{hypenum: exchangeable}, we obtain:
\begin{multline*}
     \esp{\int \sup_{u\in [0,T]} \norm{x_u}^p \indicatrice_{\underset{u\in [0,T]}{\sup} \norm{x_u} > a}\dr \Phi_t(m^n)(x) }=\\ \esp{ \sup_{u\in [0,T]} \norm{\bar X^{1,n}_{t +u}}^p  \indicatrice_{\underset{u\in [0,T]}{\sup} \norm{\bar X_{t+u}^{1,n}} > a} }\,,
\end{multline*}
for every $k,t,n$.
Recalling the decomposition introduced in~Eq.~\eqref{eq:decomp}, for every $u\in [0,T]$:
\[
  \norm{\bar X^{1,n}_{t+u}}^p \le 4^{p-1}\p{ \norm{\bar X^{1,n}_{t}}^p + \norm{N^n_{t,t+u}}^p + \norm{P^n_{t,t+u}}^p+  \norm{U^n_{t,t+u}}^p}\,.
\]
Hence,
  \begin{multline*}
        \norm{\bar X^{1,n}_{t+u}}^p\indicatrice_{\underset{u\in [0,T]}{\sup} \norm{\bar X^{1,n}_{t+u}} > a} \le 4^{p} \left( \norm{\bar X^{1,n}_{t}}^p\indicatrice_{ \norm{\bar X^{1,n}_{t}}> \tfrac a4} \right . \\ \left . +\norm{N^n_{t,t+u}}^p \indicatrice_{ \underset{u\in [0,T]}{\sup}\norm{N^n_{t,t+u}}> \tfrac a4} \right . \\
   \left .+     \norm{P^n_{t,t+u}}^p \indicatrice_{ \underset{u\in [0,T]}{\sup}\norm{P^n_{t,t+u}}> \tfrac a 4}+  \norm{U^n_{t,t+u}}^p \indicatrice_{ \underset{u\in [0,T]}{\sup}\norm{U^n_{t,t+u}}> \tfrac a 4}         \right)\,.
  \end{multline*}
  
Consequently, for each $T>0$, it suffices to obtain the uniform integrability of the four collections of random variables: 
$ ( \|\bar X_t^{1,n}\|^p\,:\, t\in\bR_+,n\in\bN^*)$, 
$( \sup_{u\in [0,T]}\|N^n_{t,t+u}\|^p \,:\, t\in\bR_+,n\in\bN^*)$, 
$\!({ \sup_{u\in [0,T]}\|{P^n_{t,t+u}}\|^p \,:\,t\in\bR_+,n\in\bN^*})$ and 
$({ \sup_{u\in [0,T]}\|{U^n_{t,t+u}}\|^p\,:\, t\in\bR_+,n\in\bN^*})$.

 $(\|\bar X_t^{1,n} \|^p : k \in \mathbb{R}_+, n \in \mathbb{N}^*)$ is uniformly integrable by the first point of Lem.~\ref{lem:moment} when $p<2$, and by the second point of  Lem.~\ref{lem:moment} when $p=2$.
As obtained in Eq.~\eqref{eq:N}, Burkholder inequality stated in~\cite[Th 1.1]{burkholder1972integral} yields:
\[
  \esp{\sup_{u\in [0,T]}\norm{N^n_{t,t+u}}^4 } \le C T^2\,,  
\]
where $C$ is a constant independent of $t,n$, and $T$. Hence, since $p<4$, we obtain the uniform integrability of $\{ \sup_{u\in [0,T]}\|{N^n_{t,t+u}}\|^p \,:\, t\in\bR_+,n\in\bN^*\}$.
As obtained in Eq.~\eqref{eq:P} and Eq.~\eqref{eq:U}, we derive:
\[
  {\sup_{u\in [0,T]}{\norm{P^n_{t,t+u}}^p}} \le {C T^{p-1}}\sum_{k=k_t-1}^{k_{t+T}-1} \tilde \gamma_{k+1} \norm{b(X^{1,n}_k,\mu_k^n)}^p \,,  
\]
and
\[
  {\sup_{u\in [0,T]}{\norm{U^n_{t,t+u}}^2}} \le {C T} \sum_{k=k_t-1}^{k_{t+T}-1} \tilde \gamma_{k+1}{\norm{\zeta^{1,n}_k}^2} \,,  
\]
where $C$ remains a constant independent of $n$ and $t$. 
Using the first point of Lem.~\ref{lem:moment} when $p<2$, and the second point of Lem.~\ref{lem:moment} when $p=2$, by de la Vallée Poussin theorem, there exists a non-decreasing, convex, and non-negative function $F:\mathbb{R}^*_+\rightarrow \mathbb{R}$ such that
\[
  \lim_{h\to\infty} \frac {F(h)}h = \infty,\,\text{and } \sup_{k\in\bN,n\in\bN^*}\esp{F\p{\norm{b(X^{1,n}_{k},\mu_k^n)}^p}} <\infty .
\]
Hence, by Jensen's inequality, 
\[
  \esp{F\p{\sup_{u\in [0,T]}{\norm{P^n_{t,t+u}}^p}}} \le \frac 1T \sum_{k=k_t-1}^{k_{t+T}-1} \tilde \gamma_{k+1}\esp{F\p{CT^p \norm{ b(X^{1,n}_k,\mu^n_k)     }^p}} .
\]
Consequently,
\begin{align*}
 \sup_{t\in\bR_+,n\in\bN^*}  \esp{F\p{\sup_{u\in [0,T]}{\norm{P^n_{t,t+u}}^p}}} &< \infty\,.
\end{align*}

Therefore, de la Vallée Poussin theorem yields the uniform integrability of the collection $(\sup_{u\in [0,T]}\|{P^n_{t,t+u}}\|^p \,:\,t\in\bR_+,n\in\bN^*)\,.$ The uniform integrability of the collection $({ \sup_{u\in [0,T]} \|{U^n_{t,t+u}}\|^p\,:\, t\in\bR_+,n\in\bN^*})$ is obtained, by the same arguments.
This completes the proof.
\end{proof}

To conclude the proof of Prop.~\ref{prop:tight}, it is sufficient to remark that the tightness conditions provided
in Lem.~\ref{lem:meleard} are satisfied, thanks to Lem.~\ref{lem:Iui} and \ref{lem:TMui}, with Prop.~\ref{prop:space}.

\subsection{Proof of Prop.~\ref{prop:limit-process}}\label{sec:support}

The core of the proof is provided by the following proposition. 
\begin{prop} 
\label{prop:Gmn0}
Let Assumptions~\ref{hyp:b},~\ref{hyp:gamma},~\ref{hyp:algo} and~\ref{hyp:moment} hold, 
\[
\lim_{(t,n)\to(\infty,\infty)} 
  \bE\abs{G(\Phi_t(m^{n}))} =0 \,,
\]
for each function $G\in \cG_p$. 
\end{prop} 

\begin{proof} 
We need to show that for each $\bR_+\times\bN$--valued sequence 
$(t_n,\varphi_n)\to(\infty,\infty)$ as $n\to\infty$, the convergence  
$\bE\abs{G(\Phi_{t_n}(m^{\varphi_n}))} \to 0$ holds true, where 
$G=G_{r,\phi,h_1,\dots,h_r,t,s,v_1,\dots,v_r}$ has the form of Eq.~(\ref{eq:G}),
with $0\leq v_1\leq \cdots \leq v_r \leq s \leq t$. 
We take $\varphi_n = n$ for notational simplicity, and we write 
$\mathsf{m}_n :=\Phi_{t_n}(m^{{n}}) \in \cP_p(\cC)$. We have 
\begin{multline}
\label{eq:Gmn}
   G(\mathsf{m}_n) =\\  \frac 1n\sum_{i\in[n]} \left( 
  \phi(\bar X^{i,{n}}_{t_n + t})  -\phi(\bar X^{i,{n}}_{t_n + s}) 
  -\int_{t_n+s}^{t_n+t} 
  \psi(\bar X^{i,{n}}_{u},m^n_u)
   \dr u\right ) Q^{i,n},
\end{multline}

where we set $\psi(x,\mu) := \ps{\nabla \phi(x),b(x,\mu)}+\tr\p{\sigma(x,\mu)^TH_\phi(x)\sigma(x,\mu)}$, 
and 
\[
Q^{i,n} := \prod_{j=1}^r h_j(\bar X^{i,{n}}_{t_n + v_j}).  
\]
We note right away that $| Q^{i,n} | \leq C$ where $C$ depends on the 
functions $h_j$ only, and furthermore, the random variables 
$\{Q^{i,n}\}_{i\in[n]}$ are $\cF^n_{k_{t_n+s}}$--measurable, where we recall
that the integer $k_t$ is defined  by $k_t  := \inf \{ k\,:\, \sum_{i=1}^{k}\gamma_{i} \ge t  \}$.

In the remainder, we suppress the superscript $^{(n)}$ from most of our 
notations for clarity. 
To deal with the right hand side of~\eqref{eq:Gmn}, we begin by expressing 
$\phi(\bar X^{i}_{t_n + t})  -\phi(\bar X^{i}_{t_n + s})$ as a 
telescoping sum in the discrete random variables $X^{i}_{k}$: 
\begin{align*} 
 \phi(\bar X^{i}_{t_n + t})  -\phi(\bar X^{i}_{t_n + s})
  &= \sum_{k= k_{t_n+s}}^{k_{t_n+t}-2} \p{\phi( X^{i}_{k+1}) 
      -\phi( X^{i}_{k})}  \\
  &\phantom{=} + \phi(\bar X^{i}_{t_n + t}) 
    -\phi(X^{i}_{k_{t_n+t}-1}) 
  + \phi(X^{i}_{k_{t_n+s}})  -\phi(\bar X^{i}_{t_n + s}). 
\end{align*} 
The summands at the r.h.s. of this expression can be decomposed as
follows.  Remember the form~\eqref{eq:main} of our algorithm. 
Denoting as $H_\phi$ the Hessian matrix of $\phi$, 
by the Taylor-Lagrange formula, there exists 
 ${\theta_{k+1}}\in [\tau_k,\tau_{k+1}]$ such that  
 {
\begin{align*} 
 &\phi( X^{i}_{k+1})  -\phi( X^{i}_{k}) \\
  &= \ps{\nabla\phi(X^{i}_{k}), X^{i}_{k+1}-X^{i}_{k}} 
   +  \frac 12\p{X^{i}_{k+1}-X^{i}_{k}}^T 
    H_\phi(\bar X^{i}_{\theta_{k+1}}) \p{X^{i}_{k+1}-X^{i}_{k}} \\
  &= \gamma_{k+1} 
\ps{\nabla\phi(X^{i,{n}}_{k}), b(X^i_k, \mu_k^n)} 
 \\&+     { \gamma_{k+1}}\tr\p{\sigma(X^i_k,\mu^n_k)^T H_\phi(X^i_k) \sigma(X^i_k,\mu_k^n)}\\ 
&
 + \sqrt{2\gamma_{k+1}} \ps{\nabla\phi(X^{i}_{k}), \xi^i_{k+1}} 
+  \frac 12\p{X^{i}_{k+1}-X^{i}_{k}}^T 
    H_\phi(\bar X^{i}_{\theta_{k+1}}) \p{X^{i}_{k+1}-X^{i}_{k}} \\ 
  & 
  -     { \gamma_{k+1}}\tr\p{\sigma(X^i_k,\mu^n_k)^T H_\phi(X^i_k) \sigma(X^i_k,\mu^n_k)}+ \gamma_{k+1} \ps{\nabla\phi(X^{i}_{k}), \zeta^i_{k+1}}   \\
&= {\gamma_{k+1}} \psi(X^i_k, \mu_k^n) 
 +  \frac 12\p{X^{i}_{k+1}-X^{i}_{k}}^T 
    H_\phi(\bar X^{i}_{\theta_{k+1}}) \p{X^{i}_{k+1}-X^{i}_{k}} 
\\ 
  & 
  + \gamma_{k+1} \ps{\nabla\phi(X^{i}_{k}), \zeta^i_{k+1}} 
  -     { \gamma_{k+1}}\tr\p{\sigma(X^i_k,\mu^n_k)^T H_\phi(X^i_k) \sigma(X^i_k,\mu^n_k)}   \\
  &+ \sqrt{2\gamma_{k+1}} \ps{\nabla\phi(X^{i}_{k}), \xi^i_{k+1}} \\
&=  {\gamma_{k+1}} \psi(X^i_k, \mu^n_k) 
 +\frac 12\p{X^{i}_{k+1}-X^{i}_{k}}^T 
    H_\phi(\bar X^{i}_{\theta_{k+1}}) \p{X^{i}_{k+1}-X^{i}_{k}}  \\ 
&
   + \gamma_{k+1} \ps{\nabla\phi(X^{i}_{k}), \zeta^i_{k+1}} 
    -     { \gamma_{k+1}}\tr\p{\sigma(X^i_k,\mu^n_k)^T H_\phi(X^i_k) \sigma(X^i_k,\mu^n_k)}+
  \\
&
 \sqrt{2\gamma_{k+1}} \ps{\nabla\phi(X^{i}_{k}), \xi^i_{k+1}} 
   + \gamma_{k+1} (\xi^i_{k+1})^T H_\phi(X^{i}_{k}) \xi^i_{k+1}  - \gamma_{k+1} (\xi^i_{k+1})^T H_\phi(X^{i}_{k}) \xi^i_{k+1} 
\end{align*} 
In this last expression, the terms  
$\psi(X^i_k, \mu^n_k)$ will be played against the 
integral term at the right hand side of~\eqref{eq:Gmn}, and the other terms
will be proven to have negligible effects. Since
$\tr(\xi^i_{k+1} (\xi^i_{k+1})^T H_\phi(\bar X^{i}_{k})) = (\xi^i_{k+1})^T H_\phi(X^{i}_{k}) \xi^i_{k+1}  $, the term 
\begin{multline*}
  \eta^i_{k+1} := 
  \sqrt{2\gamma_{k+1}} \ps{\nabla\phi(X^{i}_{k}), \xi^i_{k+1}} 
    + \gamma_{k+1} (\xi^i_{k+1})^T H_\phi(X^{i}_{k}) \xi^i_{k+1} \\
   -  { \gamma_{k+1}}\tr\p{\sigma(X^i_k,\mu^n_k)^T H_\phi(X^i_k) \sigma(X^i_k,\mu^n_k)}    
\end{multline*}

in the expression above is a martingale increment term with respect to the
filtration $(\cF^{n}_k)_k$, thanks to 
Assumption~\ref{hyp:algo}--\eqref{hypenum: va}.  
}

To proceed, considering the integral at the right hand side of~\eqref{eq:Gmn}, 
we can write 
\begin{multline*} 
\int_{t_n+s}^{t_n+t}  \psi(\bar X^{i}_{u},\bar m^{n}_{u})
   \dr u =  \int_{\tau_{k_{t_n+s}}}^{\tau_{k_{t_n+t}-1}} 
 \psi(\bar X^{i}_{u},m^n_u)\dr u \\ 
     + \int_{t_n+s}^{\tau_{k_{t_n+s}} } 
  \psi(\bar X^{i}_{u}, m^{n}_{u})\dr u 
 + \int_{\tau_{k_{t_n + t}-1}}^{t_n+t} \psi(\bar X^{i}_{u},m^{n}_{u})
  \dr u, 
\end{multline*} 
and with these decompositions, we obtain 
$G(\mathsf m_n) = \sum_{l=1}^8 \chi^n_l$, where:
\begin{align*}
\chi_1^n &:=\! 
    \frac 1{n} \!\sum_{i\in[n]} \Bigl\{ 
 \sum_{k=k_{t_n+s}}^{k_{t_n + t}-2} 
\gamma_{k+1} \psi(X^{i}_k,\mu^{n}_k) 
  - \int_{\tau_{k_{t_n+s}}}^{\tau_{k_{t_n+t}-1}}
  \psi(\bar X^{i}_{u},m^{n}_{u})\dr u \Bigr\} Q^i ,\\ 
\chi_2^n &:= 
     \frac 1n \sum_{i\in[n]} \Bigl\{ 
  \phi(\bar X^{i}_{t_n + t}) -\phi(X^{i}_{k_{t_n+t}-1}) 
  + \phi(X^{i}_{k_{t_n+s}})  -\phi(\bar X^{i}_{t_n + s}) \Bigr\} Q^{i}, \\ 
\chi_3^n &:= \!
     -\frac 1{n} \sum_{i\in[n]} \Bigl\{ 
   \int_{t_n+s}^{\tau_{k_{t_n+s}} }
 \psi(\bar X^{i}_{u}, m^{n}_{u})\dr u 
  +\!\int_{\tau_{k_{t_n + t}-1}}^{t_n+t} 
 \!\psi(\bar X^{i}_{u},m_u^n)\dr u\, \Bigr\} Q^{i} \\ 
\chi_4^n &:= \frac 1n\sum_{i\in[n]}  \sum_{k=k_{t_n+s}}^{k_{t_n+t-2}}
  \gamma_{k+1}\ps{\nabla\phi\p{X_k^{i}},\zeta_{k+1}^{i}}Q^i, \\
\chi_5^n &:= \frac 1n \sum_{i\in[n]}
 \sum_{k= k_{t_n+s}}^{k_{t_n+t}-2}\gamma_{k+1}
  \p{\xi^{i}_{k+1}}^T \p{H_\phi(\bar X^i_{\theta_{k+1}}) -H_\phi( X^i_k)} 
  \p{\xi^i_{k+1}} Q^i,  \\
\chi_6^n &:= \frac 1{{n}}\sum_{i\in[n]}
  \sum_{k= k_{t_n+s}}^{k_{t_n+t}-2} 
  \left({\sqrt{2}}\gamma_{k+1}^{3/2}b(X^i_{k},\mu_k^n)^T 
   H_\phi(\bar X^i_{{\theta_{k+1}}})\xi^i_{k+1}\right )Q^i \\ 
&\phantom{=} + \frac 1{{n}}\sum_{i\in[n]}
  \sum_{k= k_{t_n+s}}^{k_{t_n+t}-2}  \left( \frac 12{\gamma_{k+1}^2} 
  b\p{X^i_k,\mu_k^n}^T H_\phi(\bar X^i_{{\theta_{k+1}}}) b\p{X^i_k,\mu_k^n} 
  \right)Q^i, \\
\chi_7^n &:= \frac 1{n}\sum_{i\in[n]} 
  \sum_{k=k_{t_n+s}}^{k_{t_n+t-2}}\gamma_{k+1}^{3/2} 
  \left(\p{\sqrt{\gamma_{k+1}}\p{b(X_k^i,\mu_k^n)+\frac{\zeta_{k+1}^i}2} 
  + \sqrt{2}\xi_{k+1}^i}^T \right. \\
     &\phantom{=} \quad\quad\quad\quad\quad\quad\quad\quad\quad\quad 
\left. H_\phi(\bar X^i_{\theta_{k+1}})\zeta_{k+1}^i\right) Q^i, 
   \quad \text{and} 
\end{align*}  
\[
\chi_8^n := \frac 1n  \sum_{i\in[n]}
  \sum_{k= k_{t_n+s}}^{k_{t_n+t}-2} \eta^i_{k+1} Q^i. 
\]
To prove our proposition, we show that $\bE| \chi_l^n | \to 0$ for all 
$l\in[8]$. The notation $E^\times_\times$ will be generically used to refer
to error terms. 

{Let us start with $\bE| \chi_1^n|$. For $i,j,\ell\in[n]$, writing 
\[
E_{i}^n := 
\sum_{k=k_{t_n+s}}^{k_{t_n + t}-2} \gamma_{k+1} \psi(X^i_k,\mu_k^n) 
 - \int_{\tau_{k_{t_n+s}}}^{\tau_{k_{t_n+t}-1}} 
 \psi(\bar X^i_{u},m_u^n)\dr u 
\]
and using the boundedness of $Q^i$ and the exchangeability as stated by 
Assumption~\ref{hyp:algo}-\eqref{hypenum: exchangeable}, we obtain that
\[
\bE | \chi_1^n | \leq C \bE |E^n_1|. 
\]
We begin by providing a bound on the second moments of $E_{i}^n$. Recalling the definition of $\psi$, and using the compactness of 
the support of $\phi$ along with Assumption~\ref{hyp:b}, we obtain that 
\begin{align*} 
\bE (E_{i}^n)^2 &\leq 
2 (t-s)^2 \max_{u \in [t_n+s, t_n+t]} 
  \bE \| \psi(\bar X^i_{u},m_u^n) \|^2  \\ 
 &\leq C (t-s)^2 \left(1 + \sup_{u\geq 0} \bE \norm{b(\bar X^1_u,m_u^n)}^2 \right) \\ 
 &\leq C(t-s)^2   
\end{align*} 
thanks to Lem.~\ref{lem:moment}. 
To obtain that $\bE| \chi_1^n| \to 0$, we thus need to show that 
$\bE | E_{1}^n | \to 0$. }

By Prop.~\ref{prop:tight} above, the sequence $(\mathsf{m}_n)$ of
$\cP(\cC)$--valued random variables is tight.  By Lemma
\ref{lem:meleardSimple}, this is equivalent to the weak$\star$-relative
compactness of the sequence of intensities $(\bI(\mathsf{m}_n))$. For each
Borel set $A \in \cB(\cC)$, we furthermore have that 
\[
\bI(\mathsf{m}_n)(A) = 
  \frac 1n \sum_{i\in[n]} \bP\left[ \bar X^{i,n}_{t_n+\cdot} \in A \right] 
 = \bP\left[ \bar X^{1,n}_{t_n+\cdot} \in A \right] 
\]
by the exchangeability, thus, the sequence of random variables 
$(\bar X^{1,n}_{t_n+\cdot})_n$ is tight. Let us work on the r. v.
$U_n := \pi_{[0,t-s]}\# \bar X^1_{t_n+s+\cdot}$ defined 
on the set $\cC([0,t-s])$ of continuous functions on the interval $[0,t-s]$.
Since $(\bar X^{1,n}_{t_n+\cdot})_n$ and $((\Theta_t)_\#m^n)_{t,n}$ are tight (by prop.~\ref{prop:tight}), given an arbitrary 
$\varepsilon > 0$, there exists two compact sets 
$K_\varepsilon \subset \cC([0,t-s])$ and $\cK_\varepsilon\subset \cP_p(\cC)$ such that 
\[
\forall n \in \bN^*, \quad 
\bP\left[ U_n \not\in K_\varepsilon \right] + \bP[\Phi_{t_n+s}(m^n) \notin \cK_\varepsilon] \leq \varepsilon . 
\]
{
Writing 
$\bar\gamma_l = \sup_{k\geq l} \gamma_k$, we now have 
\begin{align*}
&\abs{E^n_{1}} \leq  \sum_{k=k_{t_n+s}}^{k_{t_n + t}-2} 
  \gamma_{k+1}\max_{\delta\in[0,\gamma_{k+1}]} 
  \abs{\psi(\bar X^1_{\tau_k + \delta},m^n_{\tau_{k+\delta}}) 
   - \psi(\bar X^1_{\tau_k},m^n_{\tau_k}) } \\
&\leq (t-s) \max_{\substack{u,v\in[0,t-s] \\
 |u-v|\leq \bar\gamma_{k_{t_n+s}}}}  
  \left| \psi(U_n(u), m^n_{t_n+s+u}) - \psi(U_n(v), m^n_{t_n+s+v}) \right| . 
\end{align*}

We thus can write 
\begin{align} 
&\bE \abs{E^n_{1}} = 
  \bE \abs{E^n_{1}} \un_{(U_n,\Phi_{t_n+s}(m^n)) \in K_\varepsilon\times \cK_\varepsilon} 
  + \bE \abs{E^n_{1}} \un_{(U_n, \Phi_{t_n+s}(m^n)) \not\in K_\varepsilon\times \cK_\varepsilon} 
   \nonumber \\ 
&\leq 
(t-s) \sup_{f,\rho \in K_\varepsilon\times\cK_\varepsilon} 
   \max_{\substack{u,v\in[0,t-s] \\
   |u-v|\leq \bar\gamma_{k_{t_n+s}}}}  
  \left| \psi(f(u),\rho_u) - \psi(f(v), \rho_v) \right| 
  \nonumber\\ 
  & \phantom{\le}+\sqrt{\bE (E^n_{1})^2} 
  \sqrt{ \bP\left[ U_n \not\in K_\varepsilon \right] + \bP[\Phi_{t_n+s}(m^n))\notin \cK_\varepsilon]} . 
\label{borneE} 
\end{align}
By the Arzelà-Ascoli theorem, the functions in $K_\varepsilon$ are 
uniformly equicontinuous and bounded. 
Moreover, the set $\{u\in[0,t-s]\mapsto \rho_u ,\quad \rho \in \cK_\varepsilon\}$ is also uniformly equicontinuous and $\{ \rho_u \,:\, u\in [0,t-s], \rho\in\cK_\varepsilon \}$ is include in a compact subspace of $\cP_p(\bR^d)$.

Since $\psi$ is a continuous function, by Heine theorem, $\psi$ is equicontinuous, when we restrict $\psi$ to a compact space. 
Therefore, one can easily check that the set of functions $\cS$ on $[0,t-s]$ defined as
\[
\cS := \left\{ u\mapsto \psi(f(u), \rho_u) \ : \ (f, \rho) \in K_\varepsilon \times \cK_\varepsilon
 \right\} 
\]
is a set of uniformly equicontinuous functions. As a consequence, the first
term at the right hand side of the inequality in~\eqref{borneE} converges to
zero as $n\to\infty$, since $\bar\gamma_{k_{t_n+s}} \to 0$. The second term is
bounded by $C \sqrt{\varepsilon}$ thanks to the bound we obtained on $\bE
(E^n_{1})^2$. Since $\varepsilon$ is arbitrary, we obtain that $\bE |
E_{1}^n | \to 0$, thus, $\bE| \chi_1^n | \to 0$. }

The terms $\chi_n^2$, $\chi_n^3$, and $\chi_n^5$ are dealt with similarly to
$\chi_n^1$. Considering $\chi_n^2$, we have by the exchangeability that
$\bE | \chi_n^2 | \leq C \bE | E^n_1 |$, with 
\begin{align*} 
E^n_1 &= 
  \phi(\bar X^{1}_{t_n + t}) -\phi(X^{1}_{k_{t_n+t}-1}) 
  + \phi(X^{1}_{k_{t_n+s}})  -\phi(\bar X^{1}_{t_n + s})  \\
&= \phi(\bar X^{1}_{t_n + t}) -\phi(\bar X^{1}_{\tau_{k_{t_n+t}-1}}) 
  + \phi(\bar X^{1}_{\tau_{k_{t_n+s}}})  -\phi(\bar X^{1}_{t_n + s}) .
\end{align*} 
Keeping the notations $U_n := \pi_{[0,t-s]}\# \bar X^1_{t_n+s+\cdot}$
and $\bar\gamma_l$ introduced above, we have 
\[
|E^n_1| \leq 2 \max_{\substack{u,v\in[0,t-s] \\
 |u-v|\leq \bar\gamma_{k_{t_n+s}}}}  
  \left| \phi(U_n(u)) - \phi(U_n(v)) \right| . 
\]
Taking $\varepsilon > 0$, selecting the compact 
$K_\varepsilon \subset \cC([0, t-s])$ as we did for $\chi_1^n$, and 
recalling that the function $\phi$ is bounded, we have 
\begin{align*} 
\bE \abs{E^n_1} &\leq 2 
 \sup_{f \in K_\varepsilon} 
   \max_{\substack{u,v\in[0,t-s] \\
   |u-v|\leq \bar\gamma_{k_{t_n+s}}}}  
  \left\| \phi(f(u)) - \phi(f(v)) \right\| 
   + C \bP\left[ U_n \not\in K_\varepsilon \right] , 
\end{align*}
and we obtain the $\bE|\chi_n^2| \to 0$ by the same argument as for $\chi_n^1$. 

The treatment of $\chi_n^3$ is very similar to $\chi_n^2$ and is omitted. Let
us provide some details for $\chi_n^5$. Here we have by exchangeability that
\[
\bE | \chi_5^n | \leq \sum_{k= k_{t_n+s}}^{k_{t_n+t}-2} 
  \gamma_{k+1} \bE | E^{1,n}_k |, 
\]
where 
\[
   E^{1,n}_k := 
  \p{\xi^1_{k+1}}^T \p{H_\phi(\bar X^1_{\theta_{k+1}}) 
  - H_\phi( X^1_k)} \p{\xi^1_{k+1}} Q^1. 
 \]
satisfies 
\[
| E^{1,n}_k | \leq C \| \xi^1_{k+1} \|^2 \max_{\substack{u,v\in[0,t-s] \\
 |u-v|\leq \bar\gamma_{k_{t_n+s}}}}  
  \left\| H_\phi(U_n(u)) - H_\phi(U_n(v)) \right\| . 
\]
Therefore, 
\begin{align*} 
\bE \abs{E^{1,n}_k} &= 
  \bE \abs{E^{1,n}_k} \un_{U_n \in K_\varepsilon} 
  + \bE \abs{E^{1,n}_k} \un_{U_n\not\in K_\varepsilon} \nonumber \\ 
&\leq  C \bE \| \xi_{k+1} \|^2 
 \sup_{f \in K_\varepsilon} 
   \max_{\substack{u,v\in[0,t-s] \\
   |u-v|\leq \bar\gamma_{k_{t_n+s}}}}  
  \left\| H_\phi(f(u)) - H_\phi(f(v)) \right\| 
  \nonumber \\ &\phantom{\le} + \sqrt{\bE (E^{1,n}_{k})^2} 
  \sqrt{\bP\left[ U_n \not\in K_\varepsilon \right]} . 
\end{align*}
Since $\bE \| \xi_{k+1} \|^2$ and $\bE (E^{1,n}_{k})^2$ are bounded, we obtain
that $\bE | \chi_5^n |  \to 0$. 

Considering the term $\chi_n^4$, we have by exchangeability 
\begin{align*}
\bE|\chi_4^n| &\leq C \bE\left| \sum_{k=k_{t_n+s}}^{k_{t_n+t-2}}
    \gamma_{k+1}\ps{\nabla\phi\p{X_k^{1}},\zeta_{k+1}^{1}} \right| \\ 
&\leq C \bE\left| \sum_{k=k_{t_n+s}}^{k_{t_n+t-2}}
 \gamma_{k+1}\ps{\nabla\phi\p{X_k^{1}},\bE[ \zeta_{k+1}^{1} \, | \, \cF^n_k]}
  \right| 
 \\&+  C \bE\left| \sum_{k=k_{t_n+s}}^{k_{t_n+t-2}}
    \gamma_{k+1}\ps{\nabla\phi\p{X_k^{1}},\mathring{\zeta}_{k+1}^1} \right| \\
 &:= \bE| \chi_{4,1}^n| + \bE | \chi_{4,2}^n |, 
\end{align*}
where $\mathring{\zeta}_k^1 = \zeta^1_k - 
  \bE[ \zeta_k^{1} \, | \, \cF^n_{k-1}]$ is a martingale increment with 
respect to the filtration $(\cF^n_k)_k$. We have 
\[
\bE| \chi_{4,1}^n| \leq C (t-s) 
  \sup_{l \geq k_{t_n+s}} 
  \bE \left\| \bE[ \zeta_{l+1}^{1} \, | \, \cF^n_l] \right\|, 
\]
which converges to zero by Assumption~\ref{hyp:algo}--\eqref{hypenum: eta}.
By the martingale property, we furthermore have 
\[
\bE(\chi_{4,2}^n)^2 \leq C \sum_{k=k_{t_n+s}}^{k_{t_n+t-2}} \gamma_{k+1}^2 
 \leq C \bar\gamma_{k_{t_n+s}} (t-s) , 
\]
which also converges to zero. Thus, $\bE|\chi_4^n| \to 0$. 

We now turn to $\chi_6^n$. Here we write 
\[
   \chi_6^n = \frac 1{{n}}\sum_{i\in[n]}
  \sum_{k= k_{t_n+s}}^{k_{t_n+t}-2} \gamma_{k+1}^{3/2} E^i_k,  
\]
where 
\[
 \begin{split}
   E^i_k &:=  
  \sqrt{2} b(X^i_k, \mu_k^n)^T H_\phi(\bar X^i_{\theta_{k+1}})\xi^i_{k+1}Q^i \\ 
   &\phantom{=}
   +  
  \frac 12 \sqrt{\gamma_{k+1}} b(X^i_k,\mu_k^n)^T H_\phi(\bar X^i_{\theta_{k+1}})
   b(X^i_k, \mu_k^n) Q^i  
 \end{split}
\]
satisfies 
\[
| E^i_k | \leq  C \| b(X^{i}_k,\mu_k^n) \|\| \xi^i_{k+1} \| 
   + {C} \sqrt{\gamma_{k+1}} \| b(X^{i}_k,\mu_k^n) \|^2 .
\]
We readily obtain from Assumptions~\ref{hyp:b},~\ref{hyp:algo} and Lem.~\ref{lem:moment} that $\bE |  E^i_k | \leq C$,
which leads to $\bE | \chi_6^n | \to 0$. 

The treatment of the term $\chi_7^n$ is similar and is omitted. 

We finally deal with $\chi_8^n$ that involves the martingale increments
$\eta^i_k$. We decompose this term by writing 
{\begin{align*}
&\chi_8^n = 
  \sum_{k= k_{t_n+s}}^{k_{t_n+t}-2}   
 \frac 1n  \sum_{i\in[n]}
 \sqrt{2\gamma_{k+1}} \ps{\nabla\phi(X^{i}_{k}), \xi^i_{k+1}} Q^i + \\
&
  \sum_{k= k_{t_n+s}}^{k_{t_n+t}-2}   
 \!\!\frac {\gamma_{k+1}}{n}  \!\!\sum_{i \in[n]}
    \left( (\xi^i_{k+1})^T H_\phi(X^{i}_{k}) \xi^i_{k+1} 
  - \tr\p{ \sigma(X^i_k,\mu_k^n)^T H_\phi(X^i_k)\sigma(X^i_k,\mu_k^n }\right) Q^i \\
&:= \chi_{8,1}^n + \chi_{8,2}^n.
\end{align*}}
Since the random vectors
$\xi_{k+1}^1, \ldots, \xi_{k+1}^n$ are decorrelated conditionally to 
$\cF^n_k$ by Assumption~\ref{hyp:algo}-\eqref{hypenum: va}, we obtain that
\[
 \espcond{\Bigl(\frac 1{{n}}\sum_{i\in[n]} 
 \sqrt{2\gamma_{k+1}} \ps{\nabla\phi(X^i_k), \xi^i_{k+1}} Q^i\Bigr)^2} 
   {\cF_{k}^n} \leq C \frac{\gamma_{k+1}}{n} ,  
\]
and by the martingale property, 
\[
\bE (\chi_{8,1}^n)^2 \leq \sum_{k= k_{t_n+s}}^{k_{t_n+t}-2} 
   C \frac{\gamma_{k+1}}{n} \leq \frac{C (t-s)}{n}. 
\]
Using the martingale property again along with the inequality 
$(\sum_1^n a_i)^2 \leq n \sum_1^n a_i^2$, we also have 
{\begin{align*} 
\bE (\chi_{8,2}^n)^2 &\leq \sum_{k= k_{t_n+s}}^{k_{t_n+t}-2} 
 \gamma_{k+1}^2 \bE \left( \frac 1{n}\sum_{i\in[n]} 
   \left( (\xi^i_{k+1})^T H_\phi(X^{i}_{k}) \xi^i_{k+1} \right . \right . \\
 &\left. \left. - \tr\p{ \sigma(X^i_k,\mu^n_k)^T H_\phi(X^i_k)\sigma(X^i_k,\mu_k^n}\right) Q^i \right)^2  \\
 &\leq C \sum_{k= k_{t_n+s}}^{k_{t_n+t}-2} 
 \gamma_{k+1}^2  \\
 &\leq \bar\gamma_{k_{t_n+s}} C (t-s) . 
\end{align*} }
It results that $\bE(\chi_8^n)^2 \to 0$. The proof of 
Prop.~\ref{prop:Gmn0} is completed. 
\end{proof}
 }
\begin{proof}[Proof of Prop.~\ref{prop:limit-process}] 
Let $(t_n,\varphi_n)_n$ be a $\bR_+\times \bN^*$--valued sequence such that the
distribution of $(\Phi_{t_n}(m^{\varphi_n}))_n$ converges to a measure $M
\in \cM$, which exists thanks to the tightness of
$(\Phi_{t_n}(m^{\varphi_n}))_n$ as established by
Prop.~\ref{prop:tight}.  Let $G \in \cG_p$. By the continuity of $G$ as
established by Lem.~\ref{lem:continuityG}, $G(\Phi_{t_n}(m^{\varphi_n}))$
converges in distribution to $G_\# M \in \cP(\bR)$. On the other hand, we know
by the previous proposition that $G(\Phi_{t_n}(m^{\varphi_n}))$ converges
in probability to zero.  Therefore, $G_\# M = \delta_0$. 

Let $\supp(M) \subset \cP_p(\cC)$ be the support of $M$, and let $\rho \in
\supp(M)$. By definition of the support, $M(\cN) > 0$ for each neighborhood
$\cN$ of $\rho$. Therefore, since $G_\# M = \delta_0$,  there exists a sequence
$(\rho_l)_{l\in\bN}$ such that $\rho_l \in \supp(M)$, $G(\rho_l) = 0$, and
$\rho_l \to_l \rho$ in $\cP_p(\cC)$. By the continuity of $G$, we obtain that
$G(\rho) = 0$, which shows that $\supp(M) \subset G^{-1}(\{0\})$.  Since $G$ is
arbitrary, we obtain that $\text{supp}(M) \subset \sV_p =\bigcap_{G\in\cG_p} G^{-1}(\{0\})$, 
and the theorem is proven. 
\end{proof}

\subsection{Proof of Th.~\ref{th:BC}} 
\label{sec:ergodicTH}
Throughout this paragraph, we assume that $1\le p\le2$.

We define the following collection $(M_t^n:t\geq 0,n\in\bN^*)$ of r.v. on $\cP(\cP_p(\cC))$:
\begin{equation}\label{eq:Mergo}
  M_{t}^{n}  := \frac 1{t}\int_{0}^{t} \delta_{\Phi_s(m^n)}\dr s\,.
\end{equation}

\begin{lemma}\label{lem:tighM}
The collection of r.v. $(M_t^n,\quad t\ge 0, n\in\bN^*)$ is tight in $\cP(\cP_p(\cC))$.
\end{lemma}
\begin{proof}
  Based on Lem.~\ref{lem:meleardSimple}, we just need to establish that the family of measures $(\bI(M_t^n))$ is relatively compact in the space $\cP(\cP_p(\cC))$.
  Recall that $\bI(M_t^n)$ is the probability measure which, to every Borel subset $A\subset \cP_p(\cC)$, associates:
  $$
  \bI(M_t^n)(A) = \frac 1t \int_0^t \bP(\Phi_s(m^n)\in A)ds
  $$
  Consider $\varepsilon>0$. By Prop.~\ref{prop:tight}, there exists a compact set $\cK\in \cP_p(\cC)$ such that
  $\bP(\Phi_s(m^n)\in \cK)>1-\varepsilon$, for all $s,n$. As a consequence, $ \bI(M_t^n)(\cK)>1-\varepsilon$. The proof is completed.
\end{proof}
Let us denote by $\mathscr{M}$ the set of weak$\star$ accumulation points of the net $((M_t^n)_\#\bP:t\geq 0,n\in\bN^*)$, as $(t,n)\to(\infty,\infty)$.
By Lem.~\ref{lem:tighM}, $\mathscr{M}$ is a non empty subset of $\cP(\cP(\cP_p(\cC)))$.
Define:
$$
\cV_p = \{M\in \cP(\cP_p(\cC))\,:\, M(\sV_p)=1\}\,.
$$
\begin{lemma}
\label{lem:ergodic}
  For every $\Upsilon\in\mathscr{M}$, 
  $\Upsilon(\cV_p)=1$. 
\end{lemma}
\begin{proof}
   Consider $\Upsilon\in \mathscr{M}$. Without restriction, we write $\Upsilon$ as the weak$\star$ limit of some
 sequence of the form $(M_{t_n}^n)_\#\bP$.
 The distance $\sW_p(\,.\,,\sV_p)$ to the set $\sV_p$ (which is non empty by Prop.~\ref{prop:limit-process})
 is a continuous function on $\cP_p(\cC)$.
  Denoting by $\ps{\,.\,,\,.\,}$ the natural dual pairing on $C_b(\cP_p(\cC))\times \cP(\cP_p(\cC))$,
  the function $\ps{\sW_p(\,.\,,\sV_p),\,\cdot\,}$ is  a  continuous on $\cP(\cP_p(\cC))$.
  Thus, the sequence of real r.v. $\ps{\sW_p(\,.\,,\sV_p),M_{t_n}^n}$ converges in distribution to
   $\ps{\sW_p(\,.\,,\sV_p),\,\cdot\,}_\#\Upsilon$.
  These variables being bounded, we obtain by taking the limits in expectation:
  \begin{align*}
    \int\int \sW_p(m,\sV_p)dM(m) d\Upsilon(M) &= \lim_{n\to\infty} \bE(\ps{\sW_p(\,.\,,\sV_p),M_{t_n}^n}) \\
&= \lim_{n\to\infty} \frac 1{t_n}\int_0^{t_n}\bE(\sW_p(\Phi_s(m^n),\sV_p)) ds \\
&\leq \limsup_{(t,n)\to(\infty,\infty)} \bE(\sW_p(\Phi_t(m^n),\sV_p)) = 0\,,
  \end{align*}
  where the last equality is due to Prop.~\ref{prop:limit-process}. As $\sV_p$ is closed by Prop.~\ref{prop:Vp}, this concludes the proof.
\end{proof}
Recall the definition of the shift $\Theta_t : x\mapsto x_{t+\cdot}$ defined in $\cC$.
For every $t\geq 0$, define $(\Theta_t)_{\#\#} = ((\Theta_t)_\#)_\#$.
Define:
$$
\cI := \{M\in \cP(\cP_p(\cC))\,:\, \forall t> 0, M=(\Theta_t)_{\#\#} M\}\,.
$$
In other words, for every $M\in \cI$ and for every $t>0$,
 $(\Theta_t)_\#$ preserves $M$.
\begin{lemma}\label{lem:invariant}
For every $\Upsilon\in\mathscr{M}$, $\Upsilon(\cI)=1$.
\end{lemma}
\begin{proof}
  Similarly to the proof of Lem.~\ref{lem:ergodic}, we assume without restriction
  that $\Upsilon = \lim_{n\to\infty} (M_{t_n}^n)_\#\bP$ in the weak$\star$ sense.
  Set $t>0$. The map $M\mapsto d_L(M,(\Theta_t)_{\#\#}M)$ is continuous on $\cP(\cP_p(\cC))$,
  where we recall that $d_L$ stands for the Lévy-Prokhorov distance. Thus, by Fatou's lemma,
  \begin{align}
    \label{eq:dL-shift}
    \int d_L(M,(\Theta_t)_{\#\#}M)d\Upsilon(M) &\leq  \limsup_{n\to\infty} \bE(d_L(M_{t_n}^n,(\Theta_t)_{\#\#}M_{t_n}^n))\,.
  \end{align}
Note that: 
\begin{align*}
  (\Theta_t)_{\#\#}M_{t_n}^n &= \frac 1{t_n}\int_t^{t+t_n}\delta_{(\Theta_{s})_\#m^n} ds\,.
\end{align*}
In particular, for every Borel set $A\subset \cP_p(\cC)$,
$| (\Theta_t)_{\#\#}M_{t_n}^n (A)-M_{t_n}^n(A)|\leq 2t/t_n$.
The Lévy-Prokhorov distance being bounded by the total variation distance,
$d_L(M_{t_n}^n,(\Theta_t)_{\#\#}M_{t_n}^n)\leq 2t/t_n$ which tends to zero.
The l.h.s. of Eq.~(\ref{eq:dL-shift}) is zero, which proves the statement for a fixed value of $t$.
The proof of the statement for all $t$, is easily concluded by a using dense denumerable subset argument.
\end{proof}
Define:
$
\cB_p = \{M\in \cP(\cP_p(\cC))\,:\, M(\BC_p)=1\}\,.
$
\begin{prop}\label{prop:BC-2}
  For every $\Upsilon\in\mathscr{M}$, $\Upsilon(\cB_p)=1$.
\end{prop}
\begin{proof}
Consider an arbitrary sequence of the form $((M_{t_n}^n)_\#\bP)$ where $t_n\to\infty$,
converging in distribution to some measure $\Upsilon\in \mathscr{M}$ as $n\to\infty$.
By Lem.~\ref{lem:invariant}, the map $(\Theta_t)_\#:\cP_p(\cC)\to\cP_p(\cC)$ preserves the measure $M$, for all $M$ $\Upsilon$-a.e.,
and for all $t$. By Lem.~\ref{lem:ergodic}, $M(\sV_p)=1$.
Thus, the restriction of the map $(\Theta_t)_\#$ to $\sV_p$, still denoted by $(\Theta_t)_\#:\sV_p\to\sV_p$ preserves the measure $M$ as well,
for all $M$ $\Upsilon$-a.e..
By the Poincaré recurrence theorem, stated in Th. 2.3 of~\cite{mane}, it follows that $M(\BC_p)=1$ for all $M$ $\Upsilon$-a.e.
\end{proof}
\begin{proof}[Proof of Th.~\ref{th:BC}]
By Lemma~\ref{lem:tighM}, the set $\mathscr{M}$ is non-empty. Consequently, by Proposition~\ref{prop:BC-2}, $\cB_p$ is non-empty, which implies that $\BC_p$ is also non-empty.
To conclude, assume by contradiction that the conclusion of Th.~\ref{th:BC} does not hold.
Then, there exists $\varepsilon>0$ and a sequence, which, without restriction, we may assume to have the form
$((M_{t_n}^n)_\#\bP)$, such that for all $n$ large enough,
\begin{equation}
\bE(\ps{\sW_p(\,.\,,\BC_p),M_{t_n}^n})>\varepsilon\,,\label{eq:contr}
\end{equation}
where  $\ps{\,.\,,\,.\,}$ is the natural dual pairing on $C_b(\cP_p(\cC))\times \cP(\cP_p(\cC))$.
Using Lem.~\ref{lem:tighM}, one can extract an other subsequence, which we still denote by $((M_{t_n}^n)_\#\bP)$,
converging to $\Upsilon\in \mathscr{M}$. As a consequence,
$$
\lim_{n\to\infty}\bE((\sW_p(\,.\,,\BC_p),M_{t_n}^n)) = \int\int \sW_p(m,\BC_p)dM(m)d\Upsilon(M) = 0\,,
$$
where we used the fact that, due to Prop.~\ref{prop:BC-2}, $\int \sW_p(m,\BC_p)dM(m)=0$ for $\Upsilon$-amost all $M$.
This contradicts Eq.~(\ref{eq:contr}).
\end{proof}

\subsection{Proof of Corollary~\ref{cor:S}}\label{sec:ergodicCO}
{Throughout this paragraph, we assume that $1\le p\le 2$.
We define the functions, for $\mu\in\cP_p(\bR^d)$,
\[
  g_1(\mu) := W_p(\mu,\BC_p^0)\,,
\]
and
\[
  g_2(\mu) := W_p(\mu, \BC_p^0)^p\,.
\]
Consider the r.v. $$Y_{n,\ell}(s)\eqdef g_\ell \p{\frac 1n\sum_{i=1}^n\delta_{\bar X_s^{i,n}}}\,,$$ for $\ell \in[2]$.
\begin{lemma}\label{lem:Yui}
The r.v. $(Y_{n,\ell}(s):s>0,n\in \bN)$ are uniformly integrable for $\ell\in[2]$.
\end{lemma}
  \begin{proof}Let $\ell\in [2]$.
   Note that $Y_{n,\ell}(s)\leq C(1+\frac 1n\sum_i\|\bar X^{i,n}_s\|^p)$.
   Hence for a convex, and increasing function $F:\bR_+^*\to\bR$, by the exchangeability stated in Assumption~\ref{hyp:algo}, we obtain $\bE (F(Y_{n,\ell}(s))) \le \bE (F(C(1+\tfrac 1n\sum_i \|\bar X_s^{i,n}\|^p)))\le F(C(1 + \bE(\|\bar X_s^{1,n}\|^p)))$.
 By de la Vallée Poussin theorem, the random variables $(Y_{n,\ell}(s):s>0,n\in \mathbb{N})$ are uniformly integrable if the random variables $(\|{\bar X_s^{1,n}}\|^p :s>0,n\in \mathbb{N})$ are uniformly integrable.
We conclude using Lem.~\ref{lem:moment}.
\end{proof}
Let $\ell\in[2]$, recall the definition of $M_t^n$ in Eq.~(\ref{eq:Mergo}), and recall that $\mathscr{M}$ is the set of cluster
points of $((M_{t}^n)_\#\bP:t\geq 0,n\in \bN^*)$ as  $(t,n)\to(\infty,\infty)$.
Consider an arbitrary sequence $t_n\to\infty$, such that $(M_{t_n}^n)_\#\bP$ converges to some measure $\Upsilon\in \mathscr{M}$.
Consider $\varepsilon>0$. By Lem.~\ref{lem:Yui}, there exists $a>0$ such that $\sup_{n,s}\bE(Y_{n,\ell}(s)\un_{Y_{n,\ell}(s)>a})<\varepsilon$.
Using the inequality $y\leq a\wedge y + y\un_{y>a}$, we obtain:
\begin{align}
  \bE\left(\frac 1{t_n}\int_0^{t_n}Y_{n,\ell}(s) ds\right) &\leq \bE\left(\frac 1{t_n}\int_0^{t_n}a\wedge Y_{n,\ell}(s) ds\right) + \varepsilon \nonumber\\
  &= \bE\left(\int a\wedge g_\ell((\pi_0)_\#m ) dM_{t_n}^n(m)\right) + \varepsilon \label{eq:YasM}
\end{align}
The restriction of $\pi_0$ to $\cP_p(\cC)$, which we still denote by $\pi_0$, is a continuous function
on $(\cP_p(\cC),\sW_p)\to (\cP_p(\bR^d),W_p)$, where $W_p$ represents the $p$-th order Wasserstein
distance on $\cP(\bR^d)$. As a consequence, the pushforward map $(\pi_0)_\#:\cP(\cP_p(\cC))\to\cP(\cP_p(\bR^d))$
is continuous. Therefore, as $(\pi_0)_\#\BC_p$ is non empty by Prop.~\ref{prop:BC-2}, 
the function $M\mapsto \int  a\wedge g_\ell((\pi_0)_\# m) dM(m)$ is bounded and continuous on $\cP(\cP_p(\cC))$.
Recall that $M_{t_n}^n$ converges in distribution to $\Upsilon$, and noting that, by  Prop.~\ref{prop:BC-2},
$$
\int\int g_\ell((\pi_0)_\#m)dM(m)d\Upsilon(M) = 0\,.
$$
Hence, by letting $n\to\infty$ in Eq.~(\ref{eq:YasM}), we obtain $\limsup_n\bE(\tfrac 1{t_n}\int_0^{t_n}Y_n(s) ds)\leq \varepsilon$. As $\varepsilon$
is arbitrary,
\begin{align}
  \lim_{n\to\infty}\bE\left(\frac 1{t_n}\int_0^{t_n}Y_n(s) ds\right) & =0\,. \label{eq:intY}
\end{align}
In order to establish the statement of Corollary~\ref{cor:S}, we now should consider replacing the integral in Eq.~(\ref{eq:intY})
by a sum. This last part is only technical. Recall the definition of $k_t  := \inf \{ k\,:\, \sum_{i=1}^{k}\gamma_{i} \ge t  \}$, and $\tau_k$ in Eq.~\eqref{def:tau}. 
Let $(\alpha_n)$ be a sequence of integers tending to infinity. By the triangular inequality,
\begin{align*}
  \bE\left(\frac{\sum_{l=1}^{\alpha_n}\gamma_lg_\ell(\mu_l^n)}{\sum_{l=1}^{\alpha_n}\gamma_l}\right) 
&= \bE\left(\frac 1{\tau_{\alpha_n}}\int_0^{\tau_{\alpha_n}} g_\ell(\mu_{k_s}^n)ds\right)\\
  &\leq  \bE\left(\frac 1{\tau_{\alpha_n}}\int_0^{\tau_{\alpha_n}} W_p(\mu_{k_s}^n, \frac 1n\sum_{i\in[n]}\delta_{\bar X^{i,n}_s} ) \indicatrice_{\ell=1}\dr s\right)\\
  &+  \bE\left(\frac 1{\tau_{\alpha_n}}\int_0^{\tau_{\alpha_n}}W_p(\mu_{k_s}^n, \frac 1n\sum_{i\in[n]}\delta_{\bar X^{i,n}_s})^p \indicatrice_{\ell=2} ds \right) \\
  &+\bE\left(\frac 1{\tau_{\alpha_n}}\int_0^{\tau_{\alpha_n}} Y_n(s)ds\right)\,.
\end{align*}
The third term in the righthand side of the above inequality tends to zero by Eq.~(\ref{eq:intY}) with $t_n=\tau_{\alpha_n}$.
We should therefore establish that the first and the second term vanish. For an arbitrary integer $l$ and $s\in [\tau_{l},\tau_{l+1}]$,
\begin{align*}
  \esp{W_p\p{\mu_l^n,\frac 1n\sum_{i\in[n]}\delta_{\bar X_s^{i,n}}}}&\leq \bE\left(\left(\frac 1n\sum_{i\in[n]} \|X_l^{i,n}-\bar X_s^{i,n}\|^p\right)^{1/p}\right) \\&\leq (\bE (\|X_l^{1,n}-\bar X_s^{1,n}\|^p))^{1/p}\,.
\end{align*}
where the last inequality uses Jensen's inequality and the exchangeability assumption. Continuing the estimation,
\begin{align*}
  \bE (\|X_l^{1,n}-\bar X_s^{1,n}\|^p) &\leq \bE (\|X_{l+1}^{1,n}-X_l^{1,n}\|^p) \\ 
  & \le \esp{ 3^{p-1}\gamma^p_{l+1}  \norm{b(X_l^{1,n},\mu_l^n)}^p}   \\
  &+ \esp{3^{p-1}\gamma_{l+1}^{p/2} \norm{\xi_{l+1}^{1,n}}^p}+ \esp{3^{p-1}\gamma_{l+1}^p\norm{\zeta_{l+1}^{1,n}}^p}.\\
  &\le C(\gamma_{l+1}^{p/2}  +\gamma_{l+1}^p) ,
\end{align*}
where we used Assumptions~\ref{hyp:b}, and~\ref{hyp:algo}.
Consequently,
\[
  \bE\left(\frac 1{\tau_{\alpha_n}}\int_0^{\tau_{\alpha_n}} W_p(\mu_{k_s}^n,\frac 1n\sum_{i\in[n]}\delta_{\bar X_s^{i,n}})ds\right)  \le \frac{\sum_{l=1}^{\alpha_n}\gamma_l \p{C(\gamma_{l+1}^{p/2}  +\gamma_{l+1}^p)}^{1/p}} {\sum_{l=1}^{\alpha_n}\gamma_l}\,.
\]
and, by the same computation,
\[
  \bE\left(\frac 1{\tau_{\alpha_n}}\int_0^{\tau_{\alpha_n}} W_p(\mu_{k_s}^n,\frac 1n\sum_{i\in[n]}\delta_{\bar X_s^{i,n}})^p ds\right)  \le \frac{\sum_{l=1}^{\alpha_n}\gamma_l \p{C(\gamma_{l+1}^{p/2}  +\gamma_{l+1}^p)}} {\sum_{l=1}^{\alpha_n}\gamma_l}\,.
\]
As Assumption~\ref{hyp:gamma} holds, $C(\gamma_{l+1}^{p/2}  +\gamma_{l+1}^p) \to_{l\to\infty} 0$, and $\sum_{l\ge 1} \gamma_l =\infty$. Therefore, by Stolz-Cesàro theorem, the r.h.s. of the above inequality converges to $0$ when $n\to\infty$.
Hence,
\[
  \lim_{n\to \infty}\bE\left(\frac{\sum_{l=1}^{\alpha_n}\gamma_l g_\ell(\mu_l^n)}{\sum_{l=1}^{\alpha_n}\gamma_l}\right) =0\,,
\]
for an arbitrary sequence $(\alpha_n)$ diverging to $\infty$. By Markov's inequality, Corollary~\ref{cor:S} is proven.}
{
\subsection{Proof of Cor.~\ref{cor:cor2}}
\label{sec:co2}
Let $A\subset \cP_p(\bR^d)$, we define
\[
    {\text{conv}}(A):={\left \{\sum_{i\in[n]}\lambda_i \mu_i \,:\, (\mu_i)_{i\in[n]} \in A^n, \lambda_i >0 ,\sum_{i\in[n]}\lambda_i=1, n\in\bN^* \right \}}\,.
\]
    Let $(\mu_i,\nu_i)_{i\in[2]}$ be measures in $\cP_p(\bR^d)$, and let $\lambda\in[0,1]$. We claim that
    \begin{equation}\label{eq:claim}
    W_p(\lambda\mu_1 + (1-\lambda)\mu_2,\lambda\nu_1 + (1-\lambda)\nu_2)^p \le  \lambda W_p(\mu_1,\nu_1)^p + (1-\lambda) W_p(\mu_2,\nu_2)^p\,.
    \end{equation}
    Indeed let $(\pi_1^\varepsilon,\pi_2^\varepsilon) \in \Pi(\mu_1,\nu_1)\times\Pi(\mu_2,\nu_2)$ satisfying for $i\in[2]$:
    \[
    \abs{\int \norm{x-y}^p \dr\pi_i^\varepsilon(x,y) -W_p(\mu_i,\nu_i)^p} \le \varepsilon \,.
    \]
    Since $\lambda \pi_1^\varepsilon + (1-\lambda)\pi_2^\varepsilon \in \Pi(\lambda \mu_1 + (1-\lambda)\mu_2,\lambda \nu_1 + (1-\lambda)\nu_2) $, we obtain
    \begin{align*}
        &W_p(\lambda\mu_1 + (1-\lambda)\mu_2,\lambda\nu_1 + (1-\lambda)\nu_2)^p \\
        & \le \lambda\int \norm{x-y}^p \dr\pi_1^\varepsilon(x,y) + (1-\lambda) \int \norm{x-y}^p \dr\pi_2^\varepsilon(x,y)\\
        &\le  \lambda W_p(\mu_1,\nu_1)^p + (1-\lambda) W_p(\mu_2,\nu_2)^p + 2\varepsilon\,.
    \end{align*}
Since it is true for every $\varepsilon>0$, this proves our claim.

Now, let $A\subset\cP_p(\bR^d)$, there exists $\nu_1^\varepsilon, \nu_2^\varepsilon\in A$ satisfying
\[
\lambda W_p(\mu_1,A)^p + (1-\lambda) W_p(\mu_2,A)^p \ge \lambda W_p(\mu_1,\nu_1^\varepsilon)^p + (1-\lambda) W_p(\mu_2,\nu_2^\varepsilon)^p - 2\varepsilon\,.
\]
Since this is true for every $\varepsilon>0$, by Eq.~\eqref{eq:claim}:
\begin{equation}\label{eq:convexity}
    \lambda W_p(\mu_1,A)^p + (1-\lambda) W_p(\mu_2,A)^p \ge W_p(\lambda \mu_1 + (1-\lambda)\mu_2, \text{conv}(A) )^p  \,.
\end{equation}
Applying Eq.~\eqref{eq:convexity} to the second claim of Cor.~\ref{cor:S}, we obtain
\[
W_p\p{\frac{\sum_{l\in[k]}\gamma_l\mu_l^n}{n\sum_{l\in[k]}\gamma_l},\text{conv}((\pi_0)_\#(\BC_p))} \xrightarrow[(k,n)\to (\infty,\infty)]{\bP} 0\,.
\]
Since \( W_p(\cdot,\cdot) \geq W_1(\cdot,\cdot) \), we can apply the Kantorovich duality theorem and interchange the \(\inf\) and \(\sup\) to obtain our result.

\subsection{Proof of Cor.~\ref{cor:law}}
\label{sec:cor2}
Let $i\le n$. We define the $\bar I_t^{i,n}$ as
\[
    \bar I_t^{i,n} := (\bar X^{1,n}_t,\dots, \bar X^{i,n}_t)_\#\bP\,. 
\]
Define the measure 
\begin{equation}
    J^{i,n}_t = \frac 1t \int_0^t \delta_{\bar I_s^{i,n}}\dr s\,.
\end{equation}
We define the measure $\tilde J^{i,n}_t \in \cP_p((\bR^d)^i)$
\[
    \tilde J^{i,n}_t(A) := \int \mu(A) \dr J_t^{i,n}(\mu)\,,
\]
for every $A\in\cB((\bR^d)^i)$.
Recalling the definition of $M_t^n$ in Eq.~\eqref{eq:Mergo}, we remark that
\[
    (\pi_0)_{\#\#} M_t^n = \frac 1t \int_0^t \delta_{m^n_s}\dr s\,.
\] 
We define the measure $\tilde M_t^{i,n}\in \cP_p((\bR^d)^i)$ as
\[
\tilde M_t^{i,n}(A):= \bE\p{\int \mu^{\otimes i}(A)\dr (\pi_0)_{\#\#}M_t^n(\mu)}\,,
\]
for every $A\in\cB((\bR^d)^i)$.
\begin{lemma}\label{lem:law}
    There exits a constant $C$, independent of $t$ and $n$, such that
    \[
   \sup_{A\in \cB((\bR^d)^i)} \abs{\tilde M_t^{i,n}(A) - \tilde J^{i,n}_t(A)} \le \frac Cn\,,
    \]
    for every $t,n$.
\end{lemma}
\begin{proof}
Fist we assume that $A=(A_1,\dots,A_i)\in(\cB(\bR^d))^i$.
    Remark that
    \[
        \tilde J^{i,n}_t(A) = \frac 1t \int_0^t \bP((\bar X_s^{1,n},\dots,\bar X_s^{i,n})\in (A_1,\dots,A_i))\dr s\,,
    \]
and
    \[
   \tilde M_t^{i,n}(A) = \frac 1{tn^i} \int_0^t \sum_{j_1,\dots,j_i =1}^{n}  \bP((\bar X_s^{j_1,n}, \dots, \bar X_s^{j_i,n})\in(A_1,\dots, A_i))\dr s\,.
    \]
    By exchangeability, $$\abs{\tilde M_t^{i,n}(A) - \tilde J^{i,n}_t(A)} \le \frac Cn\,,$$ for our specific choice of $A$ and a constant $C$ independent of $A$. We conclude by a density argument. 
\end{proof}

Since the total variation distance is greater than the Lévy-Prokhorov distance denoted by $d_L$, 
by the triangular inequality and Lem.~\ref{lem:law} 
\[
d_L\p{\tilde J^{i,n}_t, (\rho^*_0)^{\otimes i}}\le     d_L\p{\tilde M_t^{i,n}, (\rho^*_0)^{\otimes i}} + \frac Cn\,.
\]
By Assumption~\ref{hyp:sing} and Prop.~\ref{prop:BC-2}, we obtain $\mathscr M =\{ \delta_{\delta_{\rho^*}}\}$. Consequently, for every $(t_n,\varphi_n)\to (\infty,\infty)$, $d_L(\tilde M^{i,\varphi_n}_{t_n}, (\rho^*_0)^{\otimes i} )\to 0$, which means that $\tilde J^{i,n}_t$ converges to $(\rho_0^*)^{\otimes i}$ in $\cP((\bR^d)^i)$. By \cite[Lem. 3.14]{Chaintron_2022}, $J^{i,\varphi_n}_{t_n}$ converges to $\delta_{(\rho_0^*)^{\otimes i}}$ in $\cP(\cP((\bR^d)^{i}))$.

By an application of Prop.~\ref{prop:tight} with Lem.~\ref{lem:meleard}, $\{J_{t_n}^{i,\varphi_n}\,:\, t\ge 0,n\in\bN \}$ is a compact subspace of $\cP(\cP_p((\bR^d)^i))$. Consequently, for every $(t_n,\varphi_n)\to (\infty,\infty)$, $J^{i,\varphi_n}_{t_n}$ converges to $\delta_{(\rho_0^*)^{\otimes i}}$ in $\cP(\cP_p((\bR^d)^{i}))$. The conclusion follows from the same proof as in Cor.~\ref{cor:S}.

\subsection{Proof of Cor.~\ref{cor:essacc}}
\label{sec:cor3}
In the proof of Cor.~\ref{cor:law}, we showed that for every subsequence $(t_n,\varphi_n)\to (\infty,\infty)$, $J^{1,\varphi_n}_{t_n} \to \delta_{\delta_{\rho_0^*}}$.

Let $U\subset\cP_p(\bR^d)$ be an open neighborhood of $\rho_0^*$. 
Then, by the Portmanteau theorem,
\[
    \underset{n\to \infty}{\lim \sup} \,J^{1,\varphi_n}_{t_n}(U) \ge  \underset{n\to \infty}{\lim \inf}\, J^{1,\varphi_n}_{t_n}(U) \ge \un_{\rho_0^*\in U}=1,
\]
for every $(t_n,\varphi_n) \to (\infty,\infty)$.
By similar arguments as in Cor.~\ref{cor:S}, $\rho_0^*$ is an essential accumulation point of $(I_k^{1,n})_{k,n}$.

Let $\tilde \mu$ be a essential accumulation point of $(I^{1,n}_k)_{n,k}$. Then, by similar arguments as in Cor.~\ref{cor:S}, for every open neighborhood $U\subset\cP_p(\bR^d)$ of $\tilde \mu$
\[
    \underset{n\to \infty}{\lim \sup} \,J^{1,\varphi_n}_{t_n}(U) > 0,
\]
for some $(t_n,\varphi_n) \to (\infty,\infty)$.
Assume that $\tilde \mu \neq \rho_0^*$. Define the closed set $F_0 := \{\mu \in\cP_p(\bR^d)\,:\, W_p(\mu,\tilde \mu) \le W_p(\rho_0^*,\tilde \mu)/2 \}$. 
The open set $U_0= \{\mu \in\cP_p(\bR^d)\,:\, W_p(\mu,\tilde \mu) < W_p(\rho_0^*,\tilde \mu)/2 \}$ is a neighborhood of $\tilde \mu$ satisfying $U_0\subset F_0$. Then by Portmanteau theorem
\[
0 = \un_{\rho_0^*\in F_0} \ge \underset{n\to \infty}{\lim \sup}\, J^{1,\varphi_n}_{t_n}(F_0) \ge  \underset{n\to \infty}{\lim \sup}\, J^{1,\varphi_n}_{t_n}(U_0) >  0,
\]
for every $(t_n,\varphi_n) \to (\infty,\infty)$. This contradicts our claim: $\tilde \mu \neq \rho_0^*$.
Consequently, $\rho_0^*$ is the unique accumulation point of $(I^{1,n}_k)_{k,n}$.}
\subsection{Proof of Th.~\ref{th:pwConvergence}}
\label{prf-ptwise} 

We let the assumptions of the theorem hold. 

\begin{lemma}
\label{lem:compact} 
For a nonempty compact set $K\subset \cP_p(\bR^d)$, it holds that 
  \[
   \lim_{t\to\infty} \max_{\nu\in K}W_p(\Psi_t(\nu), A_p) = 0 \,.
  \]
\end{lemma}
\begin{proof}
Assume for the sake of contradiction that 
\[
\exists \varepsilon >0, \exists (\nu_n) \subset K, 
 \exists (t_n) \to \infty \quad \text{such that} \quad 
     {W_p(\Psi_{t_n}(\nu_n), A_p ) > \varepsilon\,.}
\]
Choose $\delta > 0$ small enough so that the $\delta$--neighborhood
$A_p^\delta$ of $A_p$ for the distance $W_p$ is included in the
fundamental neighborhood of $A_p$. Up to taking a subsequence, we can 
assume by the compactness of $K$ that there exists $\nu_\infty \in K$ such
that $\nu_n \to_n \nu_\infty$. Since $A_p$ is a global attractor, there
exists $T > 0$ such that $W_p( \Psi_T(\nu_\infty), A_p ) \leq \delta / 2$.
Furthermore, by the continuity of $\Psi$, there exists $n_0$ such that 
\[
\forall n \geq n_0, \quad W_p(\Psi_T(\nu_n), \Psi_T(\nu_\infty)) 
 \leq \delta / 2. 
\]
This implies that $\Psi_T(\nu_n) \in A_p^\delta$ for all $n \geq n_0$. 
Since $A_p^\delta$ is included in the fundamental neighborhood of $A_p$, 
there exists $\widetilde T > 0$ such that 
\[
\forall n \geq n_0, \forall t \geq \widetilde T, \quad 
 W_p(\Psi_{\widetilde T + t}(\nu_n), A_p) \leq \varepsilon, 
\]
and we obtain our contradiction. 
\end{proof}

We now prove Th.~\ref{th:pwConvergence}. Recall that the collection
$\{\Phi_t(m^n) \}$ is tight in $\cP_p(\cC)$ by Prop.~\ref{prop:tight}. Let
$(t_n,\varphi_n)$ be a sequence such that $(t_n,\varphi_n) \to_n
(\infty,\infty)$ and such that $(\Phi_{t_n}(m^{\varphi_n}))_n$ converges
in distribution to $M\in \cM$ as given by~\eqref{M-acc}. 
To prove Th.~\ref{th:pwConvergence}, it will be enough to show that 
\[
\forall \delta, \varepsilon > 0, \exists T > 0, \quad 
\limsup_n \bP\left(W_p\left(m^{\varphi_n}_{t_n+T}, A_p\right) 
  \geq \delta\right) \leq \varepsilon. 
\] 
This shows indeed that 
\[
W_p\left( m^n_t,A_p \right)
   \xrightarrow[(t,n)\to (\infty,\infty)]{\bP} 0 , 
\]
and by taking $t = \tau_k$ and by recalling that $m^n_{\tau_k} = \mu_k^n$, 
we obtain our theorem. 

Fix $\delta$ and $\varepsilon$.  By the tightness of the family
$\{\Phi_t(m^n)\}$, there exists a compact set $\cD\subset \cP_p(\cC)$ such that
$\bP(\Phi_t(m^n) \in \cD) \geq 1 - \varepsilon / 2$ for each couple $(t,n)$.
This implies that $M(\cD) \geq 1 - \varepsilon / 2$ by the Portmanteau theorem.
Since $\sV_p$ is closed by Prop.~\ref{prop:limit-process}, the set $\cK = \cD
\cap \sV_p$ is compact in $\cP_p(\cC)$, and by consequence, it is compact in
$\sV_p$ for the trace topology.  By the same proposition, $M(\sV_p) = 1$,
therefore, $M(\cK) \geq 1 - \varepsilon / 2$. 

Since $\cP_p(\cC)$ is Polish, we can apply Skorokhod's representation theorem
\cite[Th.~6.7]{billingsley2013convergence} to the sequence
$(\Phi_{t_n}(m^{\varphi_n}))$, yielding the existence of a probability space
$(\widetilde \Omega, \widetilde \cF, \widetilde \bP)$, a sequence of
$\cP_p(\cC)$--valued random variables $(\tilde m^n)$ on $\widetilde\Omega$ and
a $\cP_p(\cC)$--valued random variable $\tilde m^\infty$ on
$\widetilde\Omega$ such that $(\tilde m^{n})_\#\widetilde\bP =
(\Phi_{t_n}(m^{\varphi_n}) )_\#\bP$, $(\tilde m^\infty)_\#\widetilde \bP = M$,
and $\tilde m^n \to \tilde m^\infty$ pointwise on $\widetilde\Omega$.  Noting
that $m^{\varphi_n}_{t_n+T}$ and $\tilde m^n_T$ have the same probability
distribution as $\cP_p(\bR^d)$--valued random variables, we show that 
\begin{equation}
\label{eq:Ptilde}
 \exists T > 0, \quad 
  \limsup_n \widetilde\bP \left(W_p\left(\tilde{m}^{n}_T, A_p\right) 
  \geq \delta\right) \leq \varepsilon. 
\end{equation}
to establish our theorem. Observing that the function 
$\rho \mapsto (\pi_0)_\# \rho$ is a continuous $\cP(\cC) \to \cP_p(\bR^d)$
function, the set $K = (\pi_0)_\# \cK$ is a nonempty compact set of 
$\cP_p(\bR^d)$. Applying Lem.~\ref{lem:compact} to the semi-flow $\Psi$ and 
to the compact $K$, we set $T > 0$ in such a way that 
\[
\max_{\nu\in K} W_p(\Psi_T(\nu), A_p) \leq \delta / 2. 
\]
By the triangular inequality, we have 
\[
W_p\left(\tilde{m}^{n}_T, A_p\right) \leq 
  W_p\left(\tilde{m}^{n}_T, \tilde{m}^{\infty}_T\right) 
  + W_p\left(\tilde{m}^{\infty}_T, A_p\right).
\]
The first term at the right hand side converges to zero for each
$\tilde\omega\in\widetilde\Omega$ by the continuity of the function $\rho
\mapsto (\pi_T)_\# \rho$, thus, this convergence takes place in probability.
We also know that for $\widetilde\bP$--almost all
$\tilde\omega\in\widetilde\Omega$, it holds that $\tilde m^\infty \in \sV_p$.
Thus, regarding the second term, we have $\tilde{m}^{\infty}_T =
\Psi_T(\tilde{m}^{\infty}_0)$ for these $\tilde\omega$, and we can write    
\begin{multline*}
    \widetilde\bP\left( 
  W_p\left(\tilde{m}^{\infty}_T, A_p\right) \geq\delta 
  \right) \\ \leq 
\widetilde\bP\left( \tilde{m}^{\infty} \not\in \cK \right) 
 + 
 \widetilde\bP\left( 
  \left( 
   W_p\left(\Psi_T(\tilde{m}^{\infty}_0), A_p\right) \geq\delta 
  \right) \cap \left( \tilde{m}^{\infty}_0 \in K \right) \right) .
\end{multline*}

When $\tilde{m}^{\infty}_0 \in K$, it holds that 
$W_p\left( \Psi_T(\tilde{m}^{\infty}_0), A_p\right) \leq \delta / 2$,
thus, the second term at the right hand side of the last inequality is 
zero. The first term satisfies 
$\widetilde\bP\left( \tilde{m}^{\infty} \not\in \cK \right) = 
 1 - M(\cK) \leq \varepsilon / 2$, and the statement~\eqref{eq:Ptilde} 
follows. Th.~\ref{th:pwConvergence} is proven.



\section{Proofs of Sec.~\ref{sec:GM}}
\label{sec:proofGM}


  The Assumptions~\ref{hyp:GM} and $\sigma>0$ are standing in this section.

\subsection{Proof of Prop.~\ref{prop:helmholtz}}

\begin{lemma}
\label{lem:densityBis}
 Let $\rho\in\sV_2$. For every $t>0$, $\rho_t$ admits a density $x\mapsto \varrho(t,x) \in C^1(\bR^d,\bR)$. 
  For every $R>0, t_2>t_1>0$, there exists a constant $C_{R,t_1,t_2}>0$ such that:
  \begin{equation}\label{eq:binfdensity}
  \inf_{t\in[t_1,t_2], \norm{x}\le R}  \varrho(t,x) \ge C_{R,t_1,t_2}  \,,
  \end{equation}
    and there exist a constant $C_{t_1,t_2}>0$, such that
\begin{equation}\label{eq:boundrho}
    \sup_{x\in\bR^d, t\in[t_1,t_2]} \norm{\nabla \varrho(t,x)}  + \varrho(t,x)\le C_{t_1,t_2}\,.
\end{equation}
Finally,
  \begin{equation}
    \label{eq:bsupdensity}
  \sup_{t\in[t_1,t_2]}\int (1+\|x\|^2)\norm{\nabla \varrho(t,x)}\dr x <\infty \,.
  \end{equation}

\end{lemma}
\begin{proof}
  The result is an application of Th.1.2 in~\cite{menozzi2021density} with the non homogeneous vector field
  $\tilde b(t,x) :=\int  b(x,y)\dr \rho_t(y)$. The proof consists in verifying the conditions of the latter theorem.
By Assumption~\ref{hyp:GM}, for every $(x,y,T)\in(\bR^d)^2\times  \bR_+$,
\begin{align*}
\sup_{t\in[0,T]}\norm{\tilde b(t,x) -\tilde b(t,y)} &\le \norm{\nabla V(x) -\nabla V(y)} \\ &\phantom{\le} + \sup_{t\in[0,T]}\int \norm{\nabla U(x-z) -\nabla U(y-z)}\dr\rho_t(z)\\ 
                                                    &\le  C (\norm{x-y}^\beta\vee \norm{x-y})\,,
\end{align*}
Moreover,
\begin{equation}
  \label{eq:borne-bt}
   \sup_{t\in[0,T]}\tilde b(t,x) \le C (1+\norm{x} + \int \sup_{t\in[0,T]} \norm{y_t} \dr \rho(y))\le C (1+\norm{x}) \,.
\end{equation}
As $\sigma>0$,  \cite[Th. 1.2]{menozzi2021density} applies: $\rho$ admits a density $x\mapsto \varrho(t,x)\in C^1(\bR^d)$, for $0<t\le T$,
and there exists four constants $(C_{i,T}, \lambda_{i,T})_{i\in[2]}$, such that:
\begin{gather*}
  \frac {1} { C_{1,T}t^{d/2}} \int \exp\p{-\frac {\norm{x-\theta_t(y)}^2}{\lambda_{1,T} t} }\dr\rho_0(y) \le {\varrho(t,x)} \\
 {\varrho(t,x)} \le   \frac {C_{1,T}}{ t^{d/2}}\int \exp\p{-\frac {\lambda_{1,T}}{t} \norm{x-\theta_t(y)}^2}\dr\rho_0(y) \\
 \norm{\nabla \varrho(t,x)} \le \frac {C_{2,T}}{t^{(d+1)/2}}\int \exp\p{-\frac {\lambda_{2,T}}{ t} \norm{x-\theta_t(y)}^2}\dr\rho_0(y)\,,
\end{gather*}
where the map $t\mapsto \theta_t(y)$ is a solution to the ordinary differential equation:
$\dd{\theta_t(y)}{t} = \tilde b(t,\theta_t(y))$ with initial condition $\theta_0(y) = y$. 
By Grönwall's lemma and Eq.~(\ref{eq:borne-bt}), there exists a constant $C_T$ such that $\|\theta_t(y)\| \le C_T\|y\|$, for every $n,y$, and $t\le T$. 
For every $t_1\le t\le t_2$, and every $x$, we obtain using a change of variables:
\begin{multline*}
(C_{1,t_2} {t_1}^{d/2})^{-1}\ge {\varrho(t,x)} \\ \ge  C_{1,t_2}{ t_2^{-d/2}}\exp\p{-\frac 2{\lambda_{1,t_2}t_1}\|x\|^2}\int \exp\p{ -\frac {2C_{t_2}}{\lambda_{1,t_2}t_1} \norm{y}^2}\dr\rho_0(y) 
\end{multline*}
\begin{multline*}
 \int (1+\|x\|^2)\norm{\nabla \varrho(t,x)} \dr x \\ \le  C_{2,t_2}{t_1^{-(d+1)/2}}\int (1+2\|x\|^2+2C_{t_2}^2\int \|y\|^2d\rho_0(y))\exp\p{- \lambda_{2,t_2} t_2^{-1} \norm{x}^2}\dr x\,, 
\end{multline*}
and $\norm{\nabla \varrho(t,x)} \le C_{2,t_2} t_1^{-(d+1)/2}\,.$
Consequently, $\rho$ satisfies Eq.~\eqref{eq:binfdensity}, Eq.~\eqref{eq:boundrho} and Eq.~\eqref{eq:bsupdensity}.
\end{proof}

For every $\rho\in \sV_2$ and every $t>0$, recall the definition of the velocity field $v_t$ in Eq.~(\ref{eq:vt}):
$ v_t(x) := -\nabla V(x)-\int \nabla U(x,y)d\rho_t(y)-\sigma^2\nabla \log\varrho(t,x)$,
where $\varrho(t,x)$ is the density of $\rho_t$ defined in Lem.~\ref{lem:densityBis}.
\begin{lemma}
  \label{lem:continuity}
  For every $\rho\in \sV_2$, and every ${t_2}>{t_1}>0$,
  \begin{equation}
  \int_{t_1}^{t_2}\int\|v_t(x)\|d\rho_t(x)dt<\infty\,.\label{eq:vt-int}
\end{equation}
  Moreover, for every $\psi\in C_c^\infty(\bR_+\times\bR^d,\bR)$,
  \begin{multline} 
  \int\psi({t_2},x)d\rho_{t_2}(x) - \int\psi({t_1},x)d\rho_{t_1}(x) \\=
  \int_{t_1}^{t_2} \int (\partial_t\psi(t,x) +  \ps{\nabla_x\psi(t,x),v_t(x)}) \rho_t(dx)dt\,.\label{eq:continuity}
\end{multline}
\end{lemma}
\begin{proof}
The first point is a direct consequence of Lem.~\ref{lem:densityBis}.
  Consider $\phi\in C_c^\infty(\bR^d, \bR)$ and $\eta\in C^\infty_c(\bR_+,\bR)$.
  Using Eq.~\eqref{eq:G} and~(\ref{eq:Vp}) with $h_1=\cdots=h_r=1$, we obtain that for each $\psi\in C_c^\infty(\bR_+\times\bR^d,\bR)$
  of the form $\psi(t,x) =  g(t)\phi(x)$,
  \begin{multline}\label{eq:ito}
  \int\psi({t_2},x)d\rho_{t_2}(x) - \int\psi({t_1},x)d\rho_{t_1}(x) = \\
  \int_{t_1}^{t_2} \int (\partial_t\psi(t,x) +   \ps{\nabla\psi(s,x),b(x,\rho_t)} + \sigma^2\Delta\psi(t,x)) \rho_t(dx)dt\,.
\end{multline}
As the functions of the form $(t,x)\mapsto g(t)\phi(x)$ are dense in  $C_c^\infty(\bR_+\times\bR^d,\bR)$, Eq.~(\ref{eq:ito}) holds in fact
for any smooth compactly supported $\psi$. Using Lem.~\ref{lem:densityBis} and an integration by parts of the Laplacian term, Eq.~(\ref{eq:continuity}) follows.
\end{proof}

The goal now is to establish that the functional $\mathscr H$ 
is a Lyapunov function.
This claim will follow from the application of Eq.~(\ref{eq:continuity}) to the functional
$(t,x)\mapsto \sigma^2\log(\varrho(t,x)) + V(x) + \int U(x-y) \varrho(t,y)  \dr y$.
However, this function is not necessarily smooth nor compactly supported. In order to be able to apply Lem.~\ref{lem:continuity},
mollification should be used. In the sequel, consider two fixed positive numbers $t_2>t_1$.

Define a smooth, compactly supported, even function $\eta:\bR^d\to\bR_+$
such that $\int \eta(x)dx=1$, and
define $\eta_\varepsilon(x) := \varepsilon^{-d}\eta(x/\varepsilon)$
for every $\varepsilon>0$. For every $t>0$, we introduce
the density $\varrho_\epsilon(t,\cdot) := \eta_\varepsilon *
\rho_\epsilon(t,\cdot)$, and we denote by
$\rho^\varepsilon_t(dx) = \varrho_\epsilon(t,x)dx$ the corresponding
probability measure.  Finally, we define:
$$
v_t^\varepsilon := \frac{\eta_\varepsilon * (v_t\varrho(t,\cdot))}{\varrho_\epsilon(t,\cdot)}\,.
$$
With these definitions at hand, it is straightforward to check that the statements of Lem.~\ref{lem:continuity}
hold when $\rho_t,v_t$ are replaced by $\rho_t^\varepsilon,v_t^\varepsilon$. More specifically,
we shall apply Eq.~(\ref{eq:continuity}) using a specific smooth function $\psi = \psi_{\varepsilon,\delta,R}$, which we will
define hereafter for fixed values of $\delta,R>0$, yealding our main equation:
\begin{multline}
  \label{eq:continuityBis}
  \int\psi_{\varepsilon,\delta,R}({t_2},x)\varrho_\varepsilon(t_2,x)dx - \int\psi_{\varepsilon,\delta,R}({t_1},x)\varrho_\varepsilon(t_1,x)dx =\\
  \int_{t_1}^{t_2} \int (\partial_t\psi_{\varepsilon,\delta,R}(t,x) +  \ps{\nabla \psi_{\varepsilon,\delta,R}(t,x),v_t^\varepsilon(x)}) \varrho_\varepsilon(t,x)dx dt\,.
\end{multline}
We now provide the definition of the function $\psi_{\varepsilon,\delta,R}\in C_c^\infty(\bR_+\times\bR^d,\bR)$ used in the above equality.
Let $\theta\in C_c^\infty(\bR,\bR)$ be a nonnegative function supported by the interval $[-t_1,t_1]$ and satisfying
$\int \theta(t)dt=1$. For every $\delta\in (0,1)$, define $\theta_\delta(t) = \theta(t/\delta)/\delta$.
We define $\varrho^{\varepsilon,\delta}(\cdot,x) := \theta_\delta* \varrho^{\varepsilon}(\cdot,x)$.
The map $t\mapsto \varrho^{\varepsilon,\delta}(t,\dot)$ is well defined on $[t_1,t_2]$, non negative, and smooth in both variables $t,x$.
In addition, we define $V_\varepsilon := \eta_\varepsilon * V$,
$U_\varepsilon := \eta_\varepsilon * U$. Finally, we introduce a smooth function $\chi$ on $\bR^d$ equal to one on the unit ball and to zero outside the ball of radius $2$,
and we define $\chi_R(x) := \chi(x/R)$. For every $(t,x)\in [t_1,t_2]\times \bR$, we define:
\begin{equation}\label{eq:psi}
     \psi_{\varepsilon,\delta,R}(t,x) :=(\sigma^2\log \varrho^{\varepsilon,\delta}(t,x) + V_\varepsilon(x)+ \int U_\varepsilon(x-y )\chi_R(y) \varrho^{\varepsilon,\delta}(t,y) \dr y) \chi_R(x)\,.
\end{equation}
We extend $\psi_{\varepsilon,\delta,R}$ to a smooth compactly supported function on $\bR_+\times \bR^d$, and we apply Eq.~(\ref{eq:continuityBis}) to the latter.
We now investigate the limit of both sides of the equality~(\ref{eq:continuityBis}) as $\delta,\varepsilon,R$ successively tend to $0,0,\infty$.
First consider the lefthand side. Note that for all $t\in [t_1,t_2]$,
\begin{multline*}
\lim_{ \varepsilon \to 0} \lim_{\delta\to 0}\psi_{\varepsilon,\delta,R}(t,x)\varrho_\varepsilon(t,x) \\:=\p{ \sigma^2\log \varrho(t,x) + V(x)+ \int U(x-y )\chi_R(y) \varrho(t,y) \dr y}\varrho(t,x) \chi_R(x)\,.
\end{multline*}
The domination argument that allows to interchange limits and integrals is provided by Lem~\ref{lem:densityBis}.
Indeed, for a fixed $R>0$, there exists a constant $C_R$ such that  $\varrho^{\varepsilon,\delta}(t,x)\leq C_R$ and
$\psi_{\varepsilon,\delta,R}(t,x)\leq C_R$ for all $\|x\|\leq R$ and all $t\in [t_1,t_2]$. As a consequence,
\begin{multline*}
  \lim_{ \varepsilon \to 0} \lim_{\delta\to 0}\int \psi_{\varepsilon,\delta,R}(t,x)\varrho_\varepsilon(t,x) =
  \sigma^2\int\chi_R(x) \varrho(t,x)\log \varrho(t,x) dx + \\ \int V(x) \chi_R(x) d\rho_t(x) + \int U(x-y )\chi_R(y)\chi_R(x) \varrho(t,x)\varrho(t,y)dx dy \,.
\end{multline*}
Since $\rho_t\in\cP_2(\bR^d)$, $\int\varrho(t,x)|\log \varrho(t,x)| dx<\infty$, and the first term in the r.h.s. of the above equation converges to $\sigma^2\int\varrho(t,x)\log \varrho(t,x) dx$ as $R\to\infty$.
Similarly, $\int V(x) \chi_R(x) d\rho_t(x)$  tends to $\int Vd\rho_t$ as $R\to\infty$, by use of the linear growth condition on $\nabla V$ in Assumption~\ref{hyp:GM},
along with the fact that $\rho_t$ admits a second order moment. The same holds for the last term. Finally, we have shown that, for every $t\in [t_1,t_2]$,
$$
\lim_{R\to\infty} \lim_{ \varepsilon \to 0} \lim_{\delta\to 0} \int\psi_{\varepsilon,\delta,R}(t,x)\varrho_\varepsilon(t,x)dx
  = \mathscr H(\rho_t)   + \frac 12\int\!\!\!\int U(x-y)\dr \rho_t(y)d\rho_t(x)\,,
  $$
recalling
$\mathscr H(\rho_t) := \sigma^2\int \log \varrho(t,\cdot)d\rho_t + \int V d\rho_t + \frac 12 \int\!\!\!\int U(x-y)\dr \rho_t(y)d\rho_t(x)\,.$
As $\delta,\varepsilon,R$ successively tend to $0,0,\infty$, we have shown that the l.h.s. of Eq.~(\ref{eq:continuityBis}) converges to:
\begin{multline}
  \mathscr H(\rho_{t_2}) - \mathscr H(\rho_{t_1})\\  + \frac 12\int\!\!\!\int U(x-y)\dr \rho_{t_2}(y)d\rho_{t_2}(x)- \frac 12\int\!\!\!\int U(x-y)\dr \rho_{t_1}(y)d\rho_{t_1}(x).
  \label{eq:lhs-helm}
\end{multline}
We should now identify the above term with the limit of the r.h.s. of Eq.~(\ref{eq:continuityBis}) in the same regime.
The latter is composed of two terms. First consider the second term:
\begin{multline*}
\int_{t_1}^{t_2} \int \ps{\nabla \psi_{\varepsilon,\delta,R}(t,x),v_t^\varepsilon(x)} \rho_t^\varepsilon(dx)dt
\\=  \int_{t_1}^{t_2} \int \ps{\nabla \psi_{\varepsilon,\delta,R}(t,x),\eta_\varepsilon * (v_t(x)\varrho(t,x))} dx dt \,.
\end{multline*}
We can let $\delta\to 0$ in this equation and interchange the limit and the integral. This is justified by Lem.~\ref{lem:densityBis}, which implies that
for every $R>0$, there exists a constant $C_{R}$ such that for every $\varepsilon>0$,  $\delta\in (0,1)$,  $t\in [t_1,t_2]$, $x\in \bR^d$,
\begin{equation}\label{eq:boundnabla}
\norm{\nabla \psi_{\varepsilon,\delta,R}(t,x)}  \le C_{R}\,.
\end{equation}
Using Eq.~(\ref{eq:boundnabla}) along with Eq.~(\ref{eq:vt-int}), the dominated convergence applies. Letting $\varepsilon\to 0$ in a second step,
the exact same argument applies, and we obtain:
\begin{align*}
\lim_{\varepsilon\to 0}\lim_{\delta\to 0}\int_{t_1}^{t_2} \int & \ps{\nabla \psi_{\varepsilon,\delta,R}(t,x),v_t^\varepsilon(x)} \varrho_\varepsilon(t,x)dx dt \\
  &=  \int_{t_1}^{t_2} \int \lim_{\varepsilon\to 0}\lim_{\delta\to 0} \ps{\nabla \psi_{\varepsilon,\delta,R}(t,x),\eta_\varepsilon * (v_t(x)\varrho(t,x))} dx dt \\
 &=  \int_{t_1}^{t_2} \int  \ps{\nabla (\lim_{\varepsilon\to 0}\lim_{\delta\to 0}\psi_{\varepsilon,\delta,R}(t,x)),v_t(x)} \varrho(t,x)dx dt \,,
\end{align*}
where the interchange between $\nabla$ and the limits is again a consequence of Lem.~\ref{lem:densityBis}.
We now write the gradient in the above inner product. Note that:
$$
\lim_{\varepsilon\to 0}\lim_{\delta\to 0}\psi_{\varepsilon,\delta,R}(t,x) = 
(\sigma^2\log \varrho(t,x) + V(x)+ \int U(x-y )\chi_R(y) \varrho(t,y) \dr y) \chi_R(x)\,.
$$
We obtain: 
\begin{multline}
  \label{eq:psi-nablabla}
\lim_{\varepsilon\to 0}\lim_{\delta\to 0}\int_{t_1}^{t_2} \int \ps{\nabla \psi_{\varepsilon,\delta,R}(t,x),v_t^\varepsilon(x)} \varrho_\varepsilon(t,x)dx dt =
\\ - \int_{t_1}^{t_2} \int  \|v_t(x)\|^2\chi_R(x) \varrho(t,x)dx dt \\
-\int_{t_1}^{t_2} \int \ps{v_t(x),\int(1-\chi_R(y))\nabla U(x-y)d\rho_t(y)}\chi_R(x)d\rho_t(x) \\
-\int_{t_1}^{t_2} \int \ps{v_t(x), \nabla \chi_R(x)(V(x)+\int U(x-y)\chi_R(y)d\rho_t(y))}d\rho_t(x)\,.
\end{multline}
By the dominated convergence theorem, Assumption~\ref{hyp:GM} and Eq.~(\ref{eq:bsupdensity}),
the last two terms in the r.h.s. of Eq.~(\ref{eq:psi-nablabla}) tend to zero as $R\to\infty$,
while the first term is handled by the monotone convergence theorem. We thus obtain:
\begin{multline}
  \label{eq:psi-nabla}
\lim_{R\to\infty}\lim_{\varepsilon\to 0}\lim_{\delta\to 0}\int_{t_1}^{t_2} \int \ps{\nabla \psi_{\varepsilon,\delta,R}(t,x),v_t^\varepsilon(x)} \varrho_\varepsilon(t,x)dx dt 
 \\= - \int_{t_1}^{t_2} \int  \|v_t(x)\|^2 \varrho(t,x)dx dt \,.
\end{multline}
As a last step, we should evaluate the limit of the first term in the r.h.s. of Eq.~(\ref{eq:continuityBis}), which writes:
$
  \int_{t_1}^{t_2} \int \partial_t\psi_{\varepsilon,\delta,R}(t,x) \varrho_\varepsilon(t,x)dx dt\,.
$
Here the domination argument allowing to interchange limits and integrals requires more attention,
and is justified by the following lemma,
whose proof is provided at the end of the section.
\begin{lemma}\label{lem:boundcdom}
Let $t_2>t_1>0$ be fixed. For every $R,\varepsilon>0$, there exists a constant $C_{R,\varepsilon}$ such that for every $\delta\in (0,1)$, $t\in [t_1,t_2]$,  $x\in \bR^d$,
        \begin{equation}\label{eq:bounddt}
   \abs{\6_t \psi_{\varepsilon,\delta,R}(t,x)} \le C_{R,\varepsilon}\,,
   \end{equation}
   for every $t\le T$, $\delta>0$, and every $x\in\bR^d$. 
\end{lemma}
By Eq.~\eqref{eq:bounddt} and by the continuity of the map $t \mapsto\6_t \varrho^\varepsilon$ (see the proof of Lem.~\ref{lem:boundcdom}),
we can expand the first term in the r.h.s. of Eq.~(\ref{eq:continuityBis}) as:
\begin{equation}
\label{eq:derivativet}
\int_{t_1}^{t_2}\int \6_t \psi_{\varepsilon,\delta,R}(t,x)  \dr \rho^\varepsilon_t(x) \dr t = \int_{t_1}^{t_2}\int \6_t \psi_{\varepsilon,\delta,R}(t,x)   \varrho^{\varepsilon,\delta}(t,x)\dr x \dr t + o_{\varepsilon,R}(\delta)\,,
\end{equation}
where $o_{\varepsilon,R}(\delta)$ represents a term which tends to zero as $\delta\to 0$, for fixed values of $\varepsilon,R$.
Note that:
  \begin{equation}
\6_t\psi_{\varepsilon,\delta,R}(t,x)  = \sigma^2\frac {\6_t \varrho^{\varepsilon,\delta}(t,x)}{\varrho^{\varepsilon,\delta}(t,x)}\chi_R(x) + \int U_\varepsilon (x-y)\chi_R(y)\chi_R(x) \6_t\varrho^{\varepsilon,\delta}(t,y)\dr y \,.\label{eq:partial_psi}
\end{equation}
Plugging this equality into~(\ref{eq:derivativet}) and noting that $U_\varepsilon$ is even (because $U$ and $\eta_\varepsilon$ are),
we obtain:
\[
    \begin{split}
      &\int_{t_1}^{t_2}\int \6_t \psi_{\varepsilon,\delta ,R}(t,x) \varrho^{\varepsilon,\delta}(t,x)\dr x\dr t\\
      & = \sigma^2\int_{t_1}^{t_2} \int \6_t \varrho^{\varepsilon}(t,x)\chi_R(x)\dr x\dr t 
      \\ &\phantom{=}+  \frac 12 \int_{t_1}^{t_2} \int\!\!\!\int  U_\varepsilon(x-y)\6_t(\varrho^{\varepsilon,\delta}(t,y)  \varrho^{\varepsilon,\delta}(t,x))  \chi_R(x)\chi_R(y)  \dr x\dr y\dr t      \\
      & =  \sigma^2 \int \varrho^{\varepsilon,\delta}(t_2,x)\chi_R(x)\dr x - \sigma^2 \int \varrho^{\varepsilon,\delta}(t_1,x)\chi_R(x)\dr x\\
     &\phantom{=}+ \frac 12 \int\int U_\varepsilon(x-y) \chi_R(x)\chi_R(y)\varrho^{\varepsilon,\delta}(t_2,x)\varrho^{\varepsilon,\delta}(t_2,y) \dr x\dr y \\
&\phantom{=}- \frac 12\int\int U_\varepsilon(x-y) \chi_R(x)\chi_R(x)\varrho^{\varepsilon,\delta}(t_1,x) \varrho^{\varepsilon,\delta}(t_1,y)  \dr x \dr y\,.
    \end{split}
  \]
  By the dominated convergence theorem, we finally obtain:
  \begin{multline} 
    \label{eq:term-partial}
 \lim_{R\to\infty}\lim_{\varepsilon\to 0} \lim_{\delta\to 0} \int_{t_1}^{t_2}\int \6_t \psi_{\varepsilon,\delta,R}(t,x)  \dr \rho^\varepsilon_t(x) \dr t =\\
     \frac 12 \int\int U(x-y) \varrho(t_2,x)\varrho(t_2,y) \dr x\dr y 
- \frac 12\int\int U(x-y) \varrho(t_1,x) \varrho(t_1,y)  \dr x \dr y\,.
\end{multline}
Putting together Eq.~(\ref{eq:lhs-helm}), (\ref{eq:psi-nabla}) and (\ref{eq:term-partial}), and passing to the limit in
the continuity equation~(\ref{eq:continuityBis}), the statement of Prop.~\ref{prop:helmholtz} follows.
\medskip

\noindent {\bf Proof of Lem.~\ref{lem:boundcdom}}.  Using Eq.~(\ref{eq:continuityBis}) and integration by parts,
  \begin{multline*}
    \varrho^\varepsilon(t_2,x)-\varrho^\varepsilon(t_1,x) \\= \int_{t_1}^{t_2}\!\! \int\ps{\nabla\eta_\varepsilon(x-y),b(y,\rho_s)} \dr \rho_s(y) \dr s  + \sigma^2\int_{t_1}^{t_2} \!\!\int \Delta\eta_\varepsilon(x-y) \dr\rho_s(y)\dr s \,.
  \end{multline*}
  Since $\rho \in \cP_2(\cC)$, $\sup_{t\in[1,T]} \|b(y,\rho_t)\| \le C(1+\norm{y}) + C\int  \sup_{t\in[1,T]} \norm{x_t} \dr \rho(x)$.
  As a consequence, $\sup_{t\in[1,T]} \|b(y,\rho_t)\| \le  C(1+\norm{y}) \,.$
Along with the observation that, for any fixed $\varepsilon$,  $\nabla \eta_\varepsilon$ and $\Delta \eta_\varepsilon$ are bounded,
it follows that $t\mapsto \varrho^\varepsilon(t,x)$ is Lipschitz continuous on $[t_1,t_2]$, and that its derivative almost everywhere is given by:
 $\6_t\varrho^\varepsilon(t,x) =\int (\ps{\nabla\eta_\varepsilon(x-y),b(y,\rho_t)}  + \Delta\eta_\varepsilon(x-y)) \dr\rho_t(y)$. 
Thus, there exists a constant $C_{\varepsilon}>0$, such that:
\[
\sup_{t\in[t_1,t_2],x\in\bR^d}\6_t \varrho^{\varepsilon} (t,x) \le C_{\varepsilon}\,.
\]
Considering the second term in the r.h.s. of Eq.~(\ref{eq:partial_psi}), the presence of the product of the compactly supported functions
$\chi_R(x)\chi_R(y)$ implies that the former is bounded in absolute value:
$$
\left| \int U_\varepsilon (x-y)\chi_R(y)\chi_R(x) \6_t\varrho^{\varepsilon,\delta}(t,y)\dr y \right|\leq C_{R,\varepsilon}\,.
$$
On the otherhand, using the lower bound \eqref{eq:binfdensity}, the first term in the r.h.s. of
Eq.~(\ref{eq:partial_psi}), is also bounded, and finally, Eq.~\eqref{eq:bounddt} follows.

\subsection{Proof of Prop.~\ref{prop:gm}}

The map $\overline{\mathscr H}: \rho\mapsto\mathscr H(\rho_\epsilon)$ is real valued and lower semicontinuous by Prop.~\ref{prop:helmholtz} and Fatou's lemma.
Moreover, for every $\rho\in \sV_2$, $\overline{\mathscr H}(\Phi_t(\rho)) - \overline{\mathscr H}(\rho) =\mathscr H(\rho_{t+\epsilon})-\mathscr H(\rho_{\epsilon}) = -\int_\epsilon^{t+\epsilon} \int \|v_s\|^2d\rho_sds$.
Therefore, $\overline{\mathscr H}(\Phi_t(\rho))$ is decreasing w.r.t. $t$, and, as such, $\overline{\mathscr H}$ is a Lyapunov function.
In addition, the identity $\overline{\mathscr H}(\Phi_t(\rho)) = \overline{\mathscr H}(\rho)$ for all $t$, is equivalent to:  $v_t=0$ $\rho_t$-a.e., for every $t\geq \epsilon$.
By Lem.~\ref{lem:continuity}, this implies that $\rho_t=\rho_\epsilon$ for all $t\geq \epsilon$.
Thus,  $\overline{\mathscr H}(\Phi_t(\rho)) = \overline{\mathscr H}(\rho)$ for all $t$, if and only if $v_\epsilon=0$ and $\rho_t=\rho_\epsilon$ for all $t$.
This means that $\overline{\mathscr H}$ is a Lyapunov function for the set $\Lambda_\epsilon$. The first point is proven.

Consider a recurrent point $\rho\in \sV_2$, say $\rho=\lim \Phi_{t_n}(\rho)$. By Prop.~\ref{prop:BC}, $\rho\in\Lambda_\epsilon$, for any $\epsilon>0$.
This means that there exists $\mu\in \cS$ such that $\rho_t=\mu$ for all $t>0$.
By continuity of the map $(\pi_0)_\#$, $\rho_0=\lim \rho_{t_n}$. Thus, $\rho_0 = \mu$.
This means that $\rho_t=\mu$ for all $t\geq 0$, which writes $\rho\in \Lambda_0$. The proof is complete.

\subsection{Proof of Prop.~\ref{GF-GM}}

Since $\beta = 1$, we obtain by Assumption~\ref{hyp:GM} that $\nabla U$ 
and $\nabla V$ are Lipschitz continuous, therefore, the functions $U$ and $V$ 
are weakly convex. 
Thus, we obtain from our assumptions that the functions $U$ and $V$ with $U$ 
being even are differentiable, weakly convex, and they satisfy the doubling 
assumption. In these conditions, the following facts hold true by 
\cite[Th.~11.2.8]{ambrosio2005gradient} (see also, \emph{e.g.},
\cite{dan-sav-10}): for each measure $\nu_0 \in \cP_2(\bR^d)$, there exists an 
unique function $t\mapsto \nu_t \in \cP_2(\bR^d)$ that satisfies the following 
properties: 
\begin{enumerate}[i)]
\item $\nu_t \to \nu_0$ as $t\downarrow 0$. 
\item $\displaystyle{\sup_{t\in [0,T]} \int \|x\|^2 \nu_t(dx) < \infty}$ 
 for each $T > 0$. 
\item The measure $\nu_t$ has a density $\eta_t = d \nu_t  / d \mathscr{L}^d$
 for each $t > 0$. This density satisfies 
 $\eta_t \in L^1_{\text{loc}}((0,\infty); W_{\text{loc}}^{1,1}(\bR^d))$. 
\item The continuity equation 
\[
\partial_t \nu_t + \nabla \cdot \left( \nu_t w_t \right) = 0 
\]
is satisfied in the distributional sense, where 
\[
w_t(x) = - \frac{\sigma^2 \nabla \eta_t(x)}{\eta_t(x)} - \nabla V(x) - 
 \int \nabla U(x-y) \eta_t(y) dy . 
\]
\item 
$\left\| w_t \right\|_{L^2(\nu_t)} \in L_{\text{loc}}^2(0,\infty)$. 

\end{enumerate}
Furthermore, the function $t\mapsto \nu_t$ is the solution of the gradient flow
in $\cP_2(\bR^d)$ of the functional $\mathscr H$ provided in the statement, and
$w_t \in - \partial \mathscr H(\nu_t)$, where $\partial \mathscr H$ is the
Fréchet sub-differential of $\mathscr H$. From the general properties of
the gradient flows detailed in \cite[Chap.~11]{ambrosio2005gradient}, one can 
then check that we can write $\nu_t = \Psi_t(\nu_0)$ where $\Psi$ is a 
semi-flow on $\cP_2(\bR^d)$.

With this at hand, all we have to do is to check that for each $\rho \in
\sV_2$, the function $t \mapsto \rho_t$ satisfies the five properties stated
above.  The first two hold true for each $\zeta \in \cP_2(\cC)$: to check the
first one, let $X\sim \zeta$. Observe that $X_t \to_{t\to 0} X_0$ by continuity
and that $\| X_t - X_0 \|^2 \leq 2 \sup_{s\in[0,1]} \| X_s \|^2$ for $t$ small,
and use the Dominated Convergence.  The second property follows from the very
definition of $\cP_2(\cC)$.  Property~3 follows from
Lem.~\ref{lem:densityBis}.  By Lem.~\ref{lem:continuity}, the continuity
equation is satisfied by the function $t\mapsto \rho_t$ with $v_t = w_t$, hence
Property~4.  Finally, Property~5 follows from Prop.~\ref{prop:helmholtz},
Equation~\eqref{eq:continuityBis-limit}. This completes the proof of 
Prop.~\ref{GF-GM}. 

{
\subsection{Proof of Prop.~\ref{prop:stabGM}}

First, we will show Eq.~\eqref{eq:dis1}. Let $\mu\in\cP_2(\bR^d)$,
\[
    \int \ps{x,b(x,\mu)}\dr\mu(x) = -\int \ps{x,\nabla V(x)}\dr\mu(x) -\int \!\!\!\int\ps{x,\nabla U(x-y)}\dr\mu(x)\dr\mu(y).\\
\]
Since $U$ is even, $\nabla U(-x) = -\nabla U(x)$. Therefore,
\[
\int \!\!\!\int\ps{x,\nabla U(x-y)}\dr\mu(x)\dr\mu(y) = \frac 12 \int \!\!\!\int\ps{x-y,\nabla U(x-y)}\dr\mu(x)\dr\mu(y).
\]
Recalling that $\ps{\nabla U(x),x} \ge -C$ and $\ps{x,\nabla V(x)} \ge \lambda \norm{x}^2$, Eq.~\eqref{eq:dis1} holds:
\[
\int \ps{x,b(x,\mu)}\dr\mu(x) \le -\lambda\int \norm{x}^2\dr\mu(x) +C\,.
\]
Eq.~\eqref{eq:dis2} is obtained by the same computation as above, where in addition, we used $\norm{\nabla U(x)}\le C(1+\norm{x})$.
Let $\mu\in\cP_2(\bR^d)$,
\[
\begin{split}
        &\int \ps{x,b(x,\mu)}\norm{x}^2\dr\mu(x)\\
        &= -\int \ps{x,\nabla V(x)}\norm{x}^2\dr\mu(x) -\int \!\!\!\int\ps{x,\nabla U(x-y)}\norm{x}^2\dr\mu(x)\dr\mu(y)\\
        & =-\int \ps{x,\nabla V(x)}\norm{x}^2\dr\mu(x) -\int \!\!\!\int\ps{x-y,\nabla U(x-y)}\norm{x}^2\dr\mu(x)\dr\mu(y)\\
        &-\int \!\!\!\int\ps{y,\nabla U(x-y)}\norm{x}^2\dr\mu(x)\dr\mu(y)\\
        &\le -\lambda \int \norm{x}^4\dr\mu(x) + C \int \norm{x}^2 \dr\mu(x) + C\int \!\!\!\int\norm{x}^2\norm{y} \dr\mu(x)\dr\mu(y) \\
        &+ C\int \!\!\!\int\norm{x}^3\norm{y} \dr\mu(x)\dr\mu(y) + C\int \!\!\!\int\norm{x}^2\norm{y}^2 \dr\mu(x)\dr\mu(y) .
\end{split}
\]
By Cauchy-Schwartz's inequality,
\[
\begin{split}
   & \int \!\!\!\int\norm{x}^3\norm{y}\dr\mu(x)\dr\mu(y) \le \int\norm{x}^2\dr\mu(x)\p{\int\norm{x}^4\dr\mu(x) }^{1/2}  \,,\\
    & \int \!\!\!\int\norm{x}^2\norm{y}\dr\mu(x)\dr\mu(y) \le \int\norm{x}^2\dr\mu(x)\p{\int\norm{x}^4\dr\mu(x) }^{1/4}\,. \\ 
\end{split}
\]
Therefore, we obtain Eq.~\eqref{eq:dis2}
\begin{multline*}
    \int \ps{x,b(x,\mu)}\norm{x}^2\dr\mu(x) \\ \le -\lambda \int \norm{x}^4\dr\mu(x) + C\p{ 1+\int \norm{x}^2 \dr\mu(x) } \p{1+ \p{\int \norm{x}^4 \dr\mu(x)}^{1/2}} .
\end{multline*}

}
\subsection{Proof of Th.~\ref{th-PW-GM}}
The convergence provided in the statement follows at once from 
Prop.~\ref{GF-GM} and Th.~\ref{th:pwConvergence}. We need to prove
that $\cS = A_2$ when $A_2 = \{ \rho^\infty \}$. For an absolutely 
continuous probability measure $d\nu(x) = \eta(x) dx \in \cP_2(\bR^d)$ with 
$\eta \in C^1(\bR^d, \bR)$, write 
\[
u_\nu(x) = -\nabla V(x)-\int \nabla U(x-y) \eta(y) dy 
   -\sigma^2\nabla \log\eta(x)\,.
\]
With this at 
hand, using Eq.~\eqref{eq:continuityBis-limit} in conjunction with 
the identity $\rho^\infty = \Psi_t(\rho^\infty)$ for each $t \geq 0$ shows
that $u_{\rho^\infty}(x) = 0$ for $\rho^\infty$--almost all $x$. 
This shows that $\rho^\infty \in \cS$. On the other hand, for 
$\nu \neq \rho^\infty$ in $\cP_2(\bR^d)$, we obtain from 
Eq.~\eqref{eq:continuityBis-limit} that the function 
$t\mapsto \mathscr H(\Psi_t(\nu))$ is strictly decreasing. Thus, 
$\int \| u_\nu \|^2 d\nu > 0$ which shows that $\nu \not\in\cS$.

\appendix

\section{Technical proofs} 

  \subsection{Proof of Prop.~\ref{prop:space}}
  \label{sec:proof-space}
Let $I\subset \bR$, we denote by $C(I,\bR^d)$ the set of continuous function from $I$ to $\bR^d$.    
 One can show, that $(\rho^n)$ is a Cauchy sequence in the complete space $(\cP_p(C([0,k],\bR^d)),W_p)$. 
  Thus, there exists a sequence of compact sets $(K_k)$ in $C([0,k],\bR^d$) such that:
  \[
    (\pi_{[0,k]})_\# \rho^n( K_k) > 1-\frac{\varepsilon}{2^k} \,,
  \]
  for all $k\in\bN^*$.
  Let
  $
    \mathcal{K} := \bigcap_{k\ge 1} \pi_{[0,k]}^{-1}(K_k)\subset \cC
  $.
  The union bound yields 
  $
    \rho^n(\cK)> 1-\varepsilon
  $.
  Referring to~\cite[Th. 2, Sec. X, Chapter 5]{bourbaki1967elements}, $\cK$ has a compact closure in $\cC$.
Hence, there exists a converging subsequence $(\rho_{\varphi_n})$ converging to $\rho \in \mathcal{P}(\mathcal{C})$. Following the proof of~\cite[Th. 6.18]{villani2008optimal}, one can readily check that $\lim_{n \to \infty} W_p((\pi_{[0,k]})_\# \rho_{n}, (\pi_{[0,k]})_\# \rho) = 0$, for every $k$. Consequently, $\lim_{n \to \infty} \mathsf{W}_p(\rho^n, \rho) = 0$, which means the completeness of $\mathcal{P}_p(\mathcal{C})$. It remains to obtain its separability.

As $\mathcal{C}$ is Polish, there exists a dense sequence $(x_n)$ in $\mathcal{C}$. Following the proof of \cite[Th. 6.18]{villani2008optimal}, one can construct a sequence $(\rho^n)$ in $\mathcal{P}_p(\mathcal{C})$ from $(x_n)$, such that $((\pi_{[0,k]})_\#\rho^n )$ is dense in $C([0,k],\mathbb{R}^d)$ for every $k$. With this result, it can be verified that $(\rho^n)$ is dense in $\mathcal{P}_p(\mathcal{C})$.

\subsection{Proof of Lem.~\ref{lem:meleard}} 
\label{sec:meleard}
   Since Prop.~\ref{prop:space} holds, $(\mathbb I(\rho^n))$ is a weak$\star$-relatively compact sequence in $\cP(\cC)$, and there exists a sequence of compact sets $(K_k)$ in $\cC$, such that
  \[
  \mathbb I(\rho^n)(K_k) > 1 -\frac{1}{k2^k}\,,
  \]
  for every $k\in\bN^*$ and every $n\in\bN^*$. Let $\varepsilon>0$. We define the relatively compact set in $\cP(\cC)$:
  \[
   \cK_\varepsilon :=  \left\{ \rho\in \cP(\cC)\, : \, \rho(K_k) > 1 -\frac{1}{k\varepsilon}\,, \, \text{for every $k\in\bN^*$, such that $k\varepsilon>1$} \right\}\,.
  \]
  The union bound and Markov's inequality yields:
  \begin{equation}\label{eq:setC}
    \proba\p{\rho^n\in \cK_\varepsilon} > 1 -\varepsilon  \,
  \end{equation}
  for every $n\in\bN^*$. 

  To be relatively compact in $\mathcal{P}_p(\mathcal{C})$, the set $\mathcal{K}_\varepsilon$ must satisfy Eq.~\eqref{eq:condUI}.
  Since the sequence $(\bI(\rho^n))$ has uniformly integrable $p$--moments, there exists a sequence $(a_{k,l})_{(k,l)\in(\bN^*)^2}$, such that  for every $l\in\bN^*$, $\lim_{k\to\infty}a_{k,l} =\infty$ , and 
  \[
   \forall (k,l)\in(\bN^*)^2,\,  \sup_{n\in\bN^*} \esp{\int \sup_{t\in [0,l]}\norm{x_t}^p \indicatrice_{\underset {t\in [0,l]}{\sup}\norm{x_t}>a_{k,l}} \dr \rho^n(x)} \le \frac {1}{kl2^{k+l}}\,.
  \]
For $\varepsilon>0$, we define a set that satisfies Eq.~\eqref{eq:condUI}:
  \[
    \mathcal U_\varepsilon := \left\{\rho\in\cP_p(\cC)\, :\,  \int \sup_{t\in [0,l]} \norm{x_t}^p\indicatrice_{\underset{t\in [0,l]} \sup \norm{x_t}> a_{k,l}} \dr \rho(x)\le \frac  1{\varepsilon kl},\, k,l\in\bN^* \right \}  \,.
  \] 
  Using Markov’s inequality and the union bound, we obtain
  \begin{equation}\label{eq:setUI}
    \bP\p{\rho^n\in \mathcal U_\varepsilon } > 1 -\varepsilon \,.
  \end{equation}
  Putting together Eq.~\eqref{eq:setC} and Eq.~\eqref{eq:setUI}, 
  \[
    \bP\p{\rho^n \in \cK_\varepsilon \cap \mathcal U_\varepsilon}          > 1-2\varepsilon                    \,.  
  \]
$\cK_\varepsilon \cap \mathcal U_\varepsilon$ is a relatively compact set in $\cP_p(\cC)$. Thus, $(\rho^n)$ is tight in $\cP_p(\cC)$.
  
{
{
\subsection{Proof of Lem.~\ref{lem:continuityG}}
\label{sec:continuity}
Given $G=G_{r,\phi,h_1,\dots,h_r,t,s,v_1,\dots,v_r} \in \cG_p$, we first want to 
show 
that $G(\rho^n) \to G(\rho^\infty)$ as $\rho^n\to\rho^\infty$ in $\cP_p(\cC)$. 
This last convergence is characterized by the fact that 
$\rho^n \to \rho^\infty$ in $\cP(\cC)$, and that the sequence $(\rho^n)$ has 
uniformly integrable $p$--moments defined by~\eqref{eq:condUI}.
We write $G(\rho^n) = \int g(x,\rho^n) d\rho^n(x)$, where
for \(x\) in \(\cC\) and $\rho\in\cP_p(\cC)$:
\begin{multline*}
  g(x,\rho) := \left(\phi(x_t) -\phi(x_s) -\right . \\ \left .\int_s^t\left(\langle\nabla\phi(x_u),b(x_u,\rho_u)\rangle +\sigma(x_u,\rho_u)^TH_\phi(x_u)\sigma(x_u,\rho_u)\right){d} u\right)h(x)\,,
\end{multline*}
and $h(x):=   \prod_{j=1}^r h_j(x_{t_j})$.

We claim that $g$ is a continuous bounded function on $\cC\times \overline{\{\rho^n\,:\, n\in\bN \}}$. The continuity is given by the assumptions on $b$.
Using Cauchy-Schwartz inequality, we state a useful inequality:
\begin{equation*}
  \abs{g(x,\rho)} \le C\p{1 + \int_s^t \norm{\nabla\phi(x_u)}\norm{b(x_u,\rho_u)} \dr u}\,,
\end{equation*}
where $C=\norm{h}_\infty \max\p{2\norm{\phi}_\infty + \norm{\sigma}_\infty^2(t-s)\norm{H_\phi}_\infty, 1}$.
Since $\phi$ is compactly supported, by Assumption~\ref{hyp:b}:
\[
b(x,\rho^n)\le C\p{1+ t\sup_{n\in\bN}\int  \sup_{u\in[0,t]} \norm{y_u}\dr \rho^n(y) }.
\]
The sequence $(\rho^n)$ has uniformly integrable $p$-moments in $\cP_p(\cC)$, consequently we obtain the bound:
\[
\sup_{x\in\cC, n\in\bN\cup\{\infty\}} b(x,\rho^n) <\infty\,.
\]

Let $\varepsilon>0$.
Since, $\rho^n \to \rho^\infty$ in $\cP_p(\cC)$, the set $\overline{\{\rho^n\,:\, n\in\bN \}}$ is a compact subspace of $\cP_p(\cC)$. Hence, there exists a compact subspace $\cK\subset \cC$ satisfying
\[
    \sup_{n\in\bN\cup\{\infty\}} \rho^n(\cK^c) \le \varepsilon, 
\]
where $\cK^c$ denotes set of function $x\in\cC$ that doesn't belong to $\cK$.
By Stone-Weierstrass's theorem, there exits $k_\varepsilon\in\bN^*$ and continuous bounded functions $(f_i,h_i)_{i\in [k_\varepsilon]} \in (C(\cC,\bR)\times C(\overline{\{\rho^n\,:\, n\in\bN \}},\bR))^{k_\varepsilon}$  satisfying 
\[
\forall (x,n)\in \cK\times \bN\cup\{\infty\}, \quad\abs{\sum_{i\in[k_\varepsilon]} f_i(x)h_i(\rho^n) - g(x,\rho^n) } \le \varepsilon.
\]
Note that for $n\in\bN$,
\begin{multline}\label{eq:decomposition}
    \abs{G(\rho^n) - G(\rho^\infty)} \le \abs{ G(\rho^n) - \sum_{i\in[k_\varepsilon]}\int f_i(x)\dr\rho^n(x)h_i(\rho^n)  } \\ + \abs{\sum_{i\in[k_\varepsilon]}\int f_i(x)\dr\rho^\infty(x)h_i(\rho^\infty) -G(\rho^\infty)} \\
    +\abs{\sum_{i\in[k_\varepsilon]}\int f_i(x)\dr\rho^n(x)h_i(\rho^n)  - \sum_{i\in[k_\varepsilon]}\int f_i(x)\dr\rho^\infty(x)h_i(\rho^\infty) }.
\end{multline}
For $n\in \bN\cup\{\infty\}$, we decompose $G(\rho^n)$ as follows
\begin{multline*}
    G(\rho^n) = \int \indicatrice_{x\in\cK} \p{g(x,\rho^n)-\sum_{i\in[k_\varepsilon]} f_i(x)h_i(\rho^n)   } \dr\rho^{n}(x) \\ +\!\int \!\indicatrice_{x\in\cK^c} \p{g(x,\rho^n)-\sum_{i\in[k_\varepsilon]} f_i(x)g_i(\rho^n)   } \dr\rho^{n}(x)   +\sum_{i\in[k_\varepsilon]}\int f_i(x)\dr\rho^n(x)h_i(\rho^n).
\end{multline*}
For every $\varepsilon>0$, since $g$ is bounded, we obtain
\begin{equation*}
    \sup_{n\in\bN \cup \{\infty\}}\abs{ G(\rho^n) - \sum_{i\in[k_\varepsilon]}\int f_i(x)\dr\rho^n(x)h_i(\rho^n)  } \le 2\varepsilon \,.
\end{equation*}
Consequently, using the latter result in Eq.~\eqref{eq:decomposition}, we obtain
\begin{multline*}
     \abs{G(\rho^n) - G(\rho^\infty)} \\ \le 4\varepsilon  +\abs{\sum_{i\in[k_\varepsilon]}\int f_i(x)\dr\rho^n(x)h_i(\rho^n)  - \sum_{i\in[k_\varepsilon]}\int f_i(x)\dr\rho^\infty(x)h_i(\rho^\infty) }.
\end{multline*}
Since, $f_i$ and $h_i$ are continuous bounded functions, we obtain for every $\varepsilon>0$
\[
\underset{n\to\infty}{\lim\sup}  \abs{G(\rho^n) - G(\rho^\infty)} \le 4\varepsilon\,,
\]
which concludes the proof.

}}

\bibliography{bib}

\begin{thebibliography}{CLRW24}

\bibitem[AGS08]{ambrosio2005gradient}
L.~Ambrosio, N.~Gigli, and G.~Savar\'{e}.
\newblock {\em Gradient flows in metric spaces and in the space of probability
  measures}.
\newblock Lectures in Mathematics ETH Z\"{u}rich. Birkh\"{a}user Verlag, Basel,
  second edition, 2008.

\bibitem[BCEM24]{benko2024convergenceratesparticleapproximation}
M.~Benko, I.~Chlebicka, J.~Endal, and B.~Miasojedow.
\newblock Convergence rates of particle approximation of forward-backward
  splitting algorithm for granular medium equations, 2024.

\bibitem[BDFR15]{budhiraja2015limits}
A.~Budhiraja, P.~Dupuis, M.~Fischer, and K.~Ramanan.
\newblock Limits of relative entropies associated with weakly interacting
  particle systems.
\newblock {\em Electronic Journal of Probability}, 20, 2015.

\bibitem[BDG72]{burkholder1972integral}
D.~L. Burkholder, B.~J. Davis, and R.~F. Gundy.
\newblock Integral inequalities for convex functions of operators on
  martingales.
\newblock In {\em Proceedings of the Sixth Berkeley Symposium on Mathematical
  Statistics and Probability}, volume~2, pages 223--240. Univ. California Press
  Berkeley, Calif., 1972.

\bibitem[Ben99]{ben-(cours)99}
M.~Bena{\"{\i}}m.
\newblock Dynamics of stochastic approximation algorithms.
\newblock In {\em S\'eminaire de {P}robabilit\'es, {XXXIII}}, volume 1709 of
  {\em Lecture Notes in Math.}, pages 1--68. Springer, Berlin, 1999.

\bibitem[BGG13]{bolley2013uniform}
F.~Bolley, I.~Gentil, and A.~Guillin.
\newblock Uniform convergence to equilibrium for granular media.
\newblock {\em Archive for Rational Mechanics and Analysis}, 208:429--445,
  2013.

\bibitem[Bil99]{billingsley2013convergence}
P.~Billingsley.
\newblock {\em Convergence of probability measures}.
\newblock Wiley Series in Probability and Statistics: Probability and
  Statistics. John Wiley \& Sons, Inc., New York, second edition, 1999.
\newblock A Wiley-Interscience Publication.

\bibitem[Bou89]{bourbaki1967elements}
N.~Bourbaki.
\newblock {\em General topology. {C}hapters 5--10}.
\newblock Elements of Mathematics (Berlin). Springer-Verlag, Berlin, 1989.
\newblock Translated from the French, Reprint of the 1966 edition.

\bibitem[BPRZ20]{bolte2020longtermdynamicssubgradient}
Jerome Bolte, Edouard Pauwels, and Rodolfo Rios-Zertuche.
\newblock Long term dynamics of the subgradient method for lipschitz path
  differentiable functions, 2020.

\bibitem[BRZ23]{bianchi2023closedmeasureapproachstochasticapproximation}
Pascal Bianchi and Rodolfo Rios-Zertuche.
\newblock A closed-measure approach to stochastic approximation, 2023.

\bibitem[BS00]{benaim2000ergodic}
M.~Bena{\"\i}m and S.~J. Schreiber.
\newblock Ergodic properties of weak asymptotic pseudotrajectories for
  semiflows.
\newblock {\em Journal of Dynamics and Differential Equations}, 12:579--598,
  2000.

\bibitem[CB18]{chizat2018global}
L.~Chizat and F.~Bach.
\newblock On the global convergence of gradient descent for over-parameterized
  models using optimal transport.
\newblock {\em Advances in neural information processing systems}, 31, 2018.

\bibitem[CCTT18]{carrillo2018analyticalframeworkconsensusbasedglobal}
J.~A. Carrillo, Y.-P. Choi, C.~Totzeck, and O.~Tse.
\newblock An analytical framework for a consensus-based global optimization
  method, 2018.

\bibitem[CD22]{Chaintron_2022}
L.-P. Chaintron and A.~Diez.
\newblock Propagation of chaos: A review of models, methods and applications.
  i. models and methods.
\newblock {\em Kinetic and Related Models}, 15(6):895, 2022.

\bibitem[CGM08]{cattiaux2008probabilistic}
P.~Cattiaux, A.~Guillin, and F.~Malrieu.
\newblock Probabilistic approach for granular media equations in the
  non-uniformly convex case.
\newblock {\em Probability theory and related fields}, 140:19--40, 2008.

\bibitem[CGPS20]{carrillo2020long}
J.~A. Carrillo, R.~S Gvalani, G.~A Pavliotis, and A.~Schlichting.
\newblock Long-time behaviour and phase transitions for the mckean--vlasov
  equation on the torus.
\newblock {\em Archive for Rational Mechanics and Analysis}, 235(1):635--690,
  2020.

\bibitem[Chi22]{chizatmean}
L.~Chizat.
\newblock Mean-field langevin dynamics: Exponential convergence and annealing.
\newblock {\em Transactions on Machine Learning Research}, 2022.

\bibitem[CLRW24]{chen2024uniform}
F.~Chen, Y.~Lin, Z.~Ren, and S.~Wang.
\newblock Uniform-in-time propagation of chaos for kinetic mean field langevin
  dynamics.
\newblock {\em Electronic Journal of Probability}, 29:1--43, 2024.

\bibitem[CMV03]{carrillo2003kinetic}
J.~A. Carrillo, R.~J. McCann, and C.~Villani.
\newblock Kinetic equilibration rates for granular media and related equations:
  entropy dissipation and mass transportation estimates.
\newblock {\em Revista Matematica Iberoamericana}, 19(3):971--1018, 2003.

\bibitem[CMV06]{carrillo2006contractions}
J.~A. Carrillo, R.~J. McCann, and C.~Villani.
\newblock Contractions in the 2-wasserstein length space and thermalization of
  granular media.
\newblock {\em Archive for Rational Mechanics and Analysis}, 179:217--263,
  2006.

\bibitem[Cor23]{cormier2023stability}
Q.~Cormier.
\newblock On the stability of the invariant probability measures of
  mckean-vlasov equations.
\newblock {\em arXiv preprint arXiv:2201.11612}, 2023.

\bibitem[DEGZ20]{durmus2020elementary}
A.~Durmus, A.~Eberle, A.~Guillin, and R.~Zimmer.
\newblock An elementary approach to uniform in time propagation of chaos.
\newblock {\em Proceedings of the American Mathematical Society},
  148(12):5387--5398, 2020.

\bibitem[DMT19]{del2019uniform}
P.~Del~Moral and J.~Tugaut.
\newblock Uniform propagation of chaos and creation of chaos for a class of
  nonlinear diffusions.
\newblock {\em Stochastic Analysis and Applications}, 37(6):909--935, 2019.

\bibitem[DS10]{dan-sav-10}
S.~Daneri and G.~Savar{\'e}.
\newblock Lecture notes on gradient flows and optimal transport.
\newblock {\em arXiv preprint arXiv:1009.3737}, 2010.

\bibitem[ELL21]{10.1214/21-EJP580}
X.~Erny, E.~L{\"o}cherbach, and D.~Loukianova.
\newblock {Conditional propagation of chaos for mean field systems of
  interacting neurons}.
\newblock {\em Electronic Journal of Probability}, 26(none):1 -- 25, 2021.

\bibitem[FKR24]{Fornasier_2024}
M.~Fornasier, T.~Klock, and K.~Riedl.
\newblock Consensus-based optimization methods converge globally.
\newblock {\em SIAM Journal on Optimization}, 34(3):2973–3004, September
  2024.

\bibitem[GLWZ22]{gui-liu-wu-zha-aap22}
A.~Guillin, W.~Liu, L.~Wu, and C.~Zhang.
\newblock {Uniform Poincaré and logarithmic Sobolev inequalities for mean
  field particle systems}.
\newblock {\em The Annals of Applied Probability}, 32(3):1590 -- 1614, 2022.

\bibitem[HR{\v{S}}S21]{hu2021mean}
K.~Hu, Z.~Ren, D.~{\v{S}}i{\v{s}}ka, and {\L}.~Szpruch.
\newblock Mean-field langevin dynamics and energy landscape of neural networks.
\newblock In {\em Annales de l'Institut Henri Poincare (B) Probabilites et
  statistiques}, volume~57, pages 2043--2065. Institut Henri Poincar{\'e},
  2021.

\bibitem[HT10]{herrmann2010non}
S.~Herrmann and J.~Tugaut.
\newblock Non-uniqueness of stationary measures for self-stabilizing processes.
\newblock {\em Stochastic Processes and their Applications}, 120(7):1215--1246,
  2010.

\bibitem[KJHK24]{karimi2024stochastic}
MR. Karimi~Jaghargh, YP. Hsieh, and A.~Krause.
\newblock Stochastic approximation algorithms for systems of interacting
  particles.
\newblock {\em Advances in Neural Information Processing Systems}, 36, 2024.

\bibitem[Leo23]{leoni2023first}
G.~Leoni.
\newblock {\em A first course in fractional Sobolev spaces}, volume 229.
\newblock American Mathematical Society, 2023.

\bibitem[Liu17]{liu2017stein}
Qiang Liu.
\newblock Stein variational gradient descent as gradient flow.
\newblock {\em Advances in neural information processing systems}, 30, 2017.

\bibitem[LLF23]{lacker2023sharp}
D.~Lacker and L.~Le~Flem.
\newblock Sharp uniform-in-time propagation of chaos.
\newblock {\em Probability Theory and Related Fields}, 187(1-2):443--480, 2023.

\bibitem[LW16]{liu2016stein}
Q.~Liu and D.~Wang.
\newblock Stein variational gradient descent: A general purpose bayesian
  inference algorithm.
\newblock {\em Advances in Neural Information Processing Systems}, 2016.

\bibitem[Mal01]{malrieu2001logarithmic}
F.~Malrieu.
\newblock Logarithmic sobolev inequalities for some nonlinear pde's.
\newblock {\em Stochastic processes and their applications}, 95(1):109--132,
  2001.

\bibitem[Mal03]{malrieu2003convergence}
F.~Malrieu.
\newblock Convergence to equilibrium for granular media equations and their
  euler schemes.
\newblock {\em The Annals of Applied Probability}, 13(2):540--560, 2003.

\bibitem[Ma{\~n}87]{mane}
R.~Ma{\~n}{\'e}.
\newblock {\em Ergodic Theory and Differentiable Dynamics}.
\newblock Ergebnisse der Mathematik und ihrer Grenzgebiete : a series of modern
  surveys in mathematics. Folge 3. Springer-Verlag, 1987.

\bibitem[MMN18]{mei2018mean}
S.~Mei, A.~Montanari, and P.-M. Nguyen.
\newblock A mean field view of the landscape of two-layer neural networks.
\newblock {\em Proceedings of the National Academy of Sciences},
  115(33):E7665--E7671, 2018.

\bibitem[MPZ21]{menozzi2021density}
S.~Menozzi, A.~Pesce, and X.~Zhang.
\newblock Density and gradient estimates for non degenerate brownian sdes with
  unbounded measurable drift.
\newblock {\em Journal of Differential Equations}, 272:330--369, 2021.

\bibitem[MRC87]{Mlard1987APO}
S.~M{\'e}l{\'e}ard and S.~Roelly-Coppoletta.
\newblock A propagation of chaos result for a system of particles with moderate
  interaction.
\newblock {\em Stochastic Processes and their Applications}, 26:317--332, 1987.

\bibitem[MRW24]{monmarche2024time}
P.~Monmarch{\'e}, Z.~Ren, and S.~Wang.
\newblock Time-uniform log-sobolev inequalities and applications to propagation
  of chaos.
\newblock {\em arXiv preprint arXiv:2401.07966}, 2024.

\bibitem[NWS22]{pmlr-v151-nitanda22a}
A.~Nitanda, D.~Wu, and T.~Suzuki.
\newblock Convex analysis of the mean field langevin dynamics.
\newblock In Gustau Camps-Valls, Francisco J.~R. Ruiz, and Isabel Valera,
  editors, {\em Proceedings of The 25th International Conference on Artificial
  Intelligence and Statistics}, volume 151 of {\em Proceedings of Machine
  Learning Research}, pages 9741--9757. PMLR, 28--30 Mar 2022.

\bibitem[Oel84]{Oelschlager1984AMA}
K.~Oelschlager.
\newblock A martingale approach to the law of large numbers for weakly
  interacting stochastic processes.
\newblock {\em Annals of Probability}, 12:458--479, 1984.

\bibitem[RVE22]{rotskoff2022trainability}
G.~Rotskoff and E.~Vanden-Eijnden.
\newblock Trainability and accuracy of artificial neural networks: An
  interacting particle system approach.
\newblock {\em Communications on Pure and Applied Mathematics},
  75(9):1889--1935, 2022.

\bibitem[SS20]{sirignano2020mean}
J.~Sirignano and K.~Spiliopoulos.
\newblock Mean field analysis of neural networks: A central limit theorem.
\newblock {\em Stochastic Processes and their Applications}, 130(3):1820--1852,
  2020.

\bibitem[SSR22]{salim2022convergence}
A.~Salim, L.~Sun, and P.~Richtarik.
\newblock A convergence theory for svgd in the population limit under
  talagrand’s inequality t1.
\newblock In {\em International Conference on Machine Learning}, pages
  19139--19152. PMLR, 2022.

\bibitem[Szn84]{sznitman1984nonlinear}
A.-S. Sznitman.
\newblock Nonlinear reflecting diffusion process, and the propagation of chaos
  and fluctuations associated.
\newblock {\em Journal of functional analysis}, 56(3):311--336, 1984.

\bibitem[Ver06]{veretennikov2006ergodic}
A.~Y. Veretennikov.
\newblock On ergodic measures for mckean-vlasov stochastic equations.
\newblock In {\em Monte Carlo and Quasi-Monte Carlo Methods 2004}, pages
  471--486. Springer, 2006.

\bibitem[Vil06]{villani2006mathematics}
C.~Villani.
\newblock Mathematics of granular materials.
\newblock {\em Journal of statistical physics}, 124(2-4):781--822, 2006.

\bibitem[Vil09]{villani2008optimal}
C.~Villani.
\newblock {\em Optimal transport}, volume 338 of {\em Grundlehren der
  mathematischen Wissenschaften [Fundamental Principles of Mathematical
  Sciences]}.
\newblock Springer-Verlag, Berlin, 2009.
\newblock Old and new.

\end{thebibliography}
\bibliographystyle{alpha}

\end{document}